                         \documentclass[12pt, oneside, psamsfonts]{amsart}

\newif\ifPDF
\ifx\pdfoutput\undefined\PDFfalse
\else \ifnum \pdfoutput > 0 \PDFtrue
        \else \PDFfalse
        \fi
\fi

\usepackage[centertags]{amsmath}
\usepackage{amsfonts}
\usepackage{mathrsfs}
\usepackage{textcomp}
\usepackage{amssymb}
\usepackage{amsthm}
\usepackage{newlfont}
\usepackage[all]{xy}

\ifPDF
  \usepackage[pdftex]{color, graphicx}
  \usepackage[pdftex, bookmarks, colorlinks]{hyperref}
  \hypersetup{colorlinks=false}

\else
 \usepackage[usenames,dvipsnames]{color}
  \usepackage[dvips]{graphicx}
  \usepackage[dvips]{hyperref}
\fi

\usepackage[scale=0.8]{geometry}

\newtheorem{thm}{Theorem}[section]
\newtheorem{cor}[thm]{Corollary}
\newtheorem{lem}[thm]{Lemma}

\theoremstyle{definition}
\newtheorem{defn}[thm]{Definition}
\theoremstyle{definition}
\newtheorem{rem}[thm]{Remark}

\numberwithin{equation}{section}

\theoremstyle{definition}
\newtheorem{NN}[thm]{}

\newcommand{\norm}[1]{\left\Vert#1\right\Vert}
\newcommand{\abs}[1]{\left\vert#1\right\vert}

\newcommand{\Int}{\mathbb Z}
\newcommand{\Comp}{\mathbb C}
\newcommand{\Ratn}{\mathbb Q}
\newcommand{\eps}{\epsilon}

\newcommand{\F}{\mathcal{F}}
\newcommand{\Kzero}{{K}_0}
\newcommand{\Kone}{{K}_1}
\newcommand{\tr}{\mathrm{T}}

\newcommand{\aff}{\mathrm{Aff}}

\newcommand{\beq}{\begin{eqnarray}}
\newcommand{\eneq}{\end{eqnarray}}
\newcommand{\tforal}{\,\,\,\text{for\,\,\,all}\,\,\,}

\newcommand{\andeqn}{\,\,\,\text{and}\,\,\,}
\newcommand{\rforal}{\,\,\,{\mathrm{for\,\,\,all}}\,\,\,}
\newcommand{\hm}{homomorphism}
\newcommand{\morp}{completely positive linear map}
\newcommand{\af}{\alpha}
\newcommand{\bt}{\beta}
\newcommand{\dt}{\delta}
\newcommand{\ep}{\epsilon}
\newcommand{\CA}{C*-algebra}

\newcommand{\R}{{\mathbb{R}}}
\newcommand{\p}{{\mathfrak{p}}}
\newcommand{\q}{{\mathfrak{q}}}
\newcommand{\fr}{{\mathfrak{r}}}
\newcommand{\Z}{{\mathbb Z}}
\newcommand{\Q}{{\mathbb Q}}
\newcommand{\T}{{\mathbb T}}

\begin{document}

\title{Homomorphisms into simple ${\cal Z}$-stable
$C^*$-Algebras}
\author{Huaxin Lin}
\email{hlin@uoregon.edu}

\author{Zhuang Niu}
\email{zniu@mun.ca}


\begin{abstract}
Let  $A$ and $B$ be  unital separable simple amenable \CA s which satisfy the Universal Coefficient
Theorem.  Suppose {that} $A$ and $B$ are $\mathcal Z$-stable and are of rationally tracial rank no more than one.
We prove the following: Suppose that $\phi, \psi: A\to B$ are unital {*-monomorphisms}. There  exists
a sequence of unitaries $\{u_n\}\subset B$ such that
$$
\lim_{n\to\infty} u_n^*\phi(a) u_n=\psi(a)\tforal a\in A,
$$
if and only if
$$
[\phi]=[\psi]\,\,\,\text{in}\,\,\, KL(A,B),\ \phi_{\sharp}=\psi_{\sharp}\andeqn\phi^{\ddag}=\psi^{\ddag},
$$
where $\phi_{\sharp}, \psi_{\sharp}: \aff(\tr(A))\to \aff(\tr(B))$ and $\phi^{\ddag}, \psi^{\ddag}:
U(A)/CU(A)\to U(B)/CU(B)$ are {the} induced maps (where $\tr(A)$ and $\tr(B)$ are {the} tracial state spaces
of $A$ and $B,$ and $CU(A)$ and $CU(B)$ are the closures of the commutator subgroups of the unitary groups
of $A$ and $B,$ respectively). We also show that this holds if $A$ is a rationally AH-algebra which is not necessarily simple.  Moreover, for any
{strictly positive unit-preserving} $\kappa\in KL(A,B)$, 
 any continuous affine map $\lambda: \aff(\tr(A))\to \aff(\tr(B))$ 
and any continuous group \hm\ $\gamma: U(A)/CU(A)\to U(B)/CU(B)$ 
which are compatible, we also show that
there is a unital \hm\, $\phi: A\to B$ so that
$([\phi],\phi_{\sharp},\phi^{\ddag})=(\kappa, \lambda, \gamma),$ at least in the  case that $K_1(A)$
is a free group.

\end{abstract}

\maketitle

\setcounter{tocdepth}{1}

\section{Introduction}
Let $X$ and $Y$ be two compact Hausdorff spaces, {and denote by $C(X)$ (or $C(Y)$) the C*-algebra of complex-valued continuous functions on $X$ (or $Y$).} {Any} continuous map $\lambda: Y\to X$ {induces} a homomorphism $\phi$ from the commutative C*-algebra $C(X)$ into the commutative C*-algebra $C(Y)$ {by $\phi(f)=f\circ\lambda,$ and any \hm\, from $C(X)$ to $C(Y)$
arises this way (in this paper, by homomorphisms or isomorphisms between C*-algebras, we mean *-homomorphisms or *-isomorphisms). It should be noted {that}, by the Gelfand-Naimark theorem, every unital commutative C*-algebra has the form $C(X)$ as above.





For non-commutative \CA s, one also studies \hm s. Let $A$ and $B$ be two unital \CA s and let $\phi, \psi: A\to B$ be two \hm s.
A fundamental problem in the study of \CA s is to determine when $\phi$ and $\psi$ are (approximately) unitarily equivalent.

The last two decades saw the rapid development of classification
of amenable \CA s, or otherwise known the Elliott program. For instance, all unital simple AH-algebras with slow dimension growth are classified by
their Elliott invariant (\cite{EGL}). In fact, {the class of} classifiable simple \CA s includes all unital separable amenable simple
\CA s with the tracial rank at most one which satisfy the Universal Coefficient Theorem (the UCT) (see
\cite{Lnuni1}). One of the crucial problems in the Elliott program
is the so-called uniqueness theorem which usually asserts that two monomorphisms are approximately unitarily equivalent if they induce
the same $K$-theory related maps under certain assumptions on \CA s
involved.

Recently, W. Winter's method (\cite{Winter-Z}) greatly advances the Elliott classification  program. The class of amenable separable simple \CA s that can be classified by the Elliott invariant has been enlarged so {that} it contains simple \CA s which no longer {are assumed to}
have finite tracial rank. In fact, with \cite{Winter-Z}, \cite{Lin-App},
\cite{L-N} and \cite{Lnclasn}, the classifiable C*-algebras now include any unital separable
simple ${\cal Z}$-stable C*-algebra $A$ satisfying the UCT 
such that $A\otimes U$ has the tracial
rank no more than one for some
UHF-algebra $U$ (it has recently been shown, for example,
$A\otimes U$ has tracial rank at most one for all UHF-algebras $U$ of infinite type, if $A\otimes C$ has tracial rank at most  one for one of infinite dimensional unital simple AF-algebra (see \cite{LS})). This class of C*-algebras is strictly larger than the class of AH-algebras without dimension growth. For example, it
contains the Jiang-Su algebra ${\cal Z}$ itself which is projectionless
and all simple unital inductive limits of so-called generalized dimension drop
algebras (see \cite{Lin-mdmrp}).

Recall that the Elliott invariant for a stably finite unital simple separable C*-algebra $A$ is
$$\mathrm{Ell}(A):=((K_0(A), K_0(A)_+, [1_A], \tr(A)), K_1(A)),$$
where $(K_0(A), K_0(A)_+, [1_A], \tr(A))$ is the quadruple consisting
of the  $\Kzero$-group, its positive cone, the order unit and  tracial simplex together with
their pairing, and $\Kone(A)$ is the $\Kone$-group.

Denote by ${\cal C}$ the class of all unital simple \CA s $A$ for which $A\otimes U$ has tracial rank no more than one for some
UHF-algebra
$U$ of infinite type.
Suppose that $A$ and $B$ are two unital separable
amenable \CA s in ${\cal C}$ which satisfy the UCT. The classification
theorem in \cite{Lnclasn} states that if the Elliott invariants of $A$ and $B$
are isomorphic, 
i.e., $$\mathrm{Ell(A)\cong\mathrm{Ell}(B)},$$
then there is an isomorphism $\phi: A\to B$ which carries the isomorphism above.

{However, the question when two isomorphisms are approximately unitarily equivalent was still left open.}
A more general question is:
for any two such  \CA s $A$ and $B,$ and, for any  two \hm s $\phi, \psi: A\to B,$
when are they approximately unitarily equivalent?

If $\phi$ and $\psi$ are approximately unitarily equivalent, then one must have,
$$
[\phi]=[\psi]\,\,\,\text{in}\,\,\, KL(A, B)\andeqn \phi_{\sharp}=\psi_{\sharp},
$$
where $\phi_{\sharp},\psi_{\sharp}: \aff(\tr(A))\to \aff(\tr(B))$  are the
 affine maps induced by $\phi$ and $\psi,$ respectively.
Moreover, as shown in \cite{Lin-AU},  one also has
$$
\phi^{\ddag}=\psi^{\ddag},
$$
where $\phi^{\ddag}, \psi^{\ddag}: U(A)/CU(A)\to U(B)/CU(B)$ are \hm s induced by $\phi, \,\psi,$ and 
$CU(A)$ and $CU(B)$ are the closures of the commutator subgroups of the unitary groups of $A$ and $B,$ respectively.

In this paper, we will show that the above conditions are also sufficient,
 that is,  the maps $\phi$ and $\psi$ are approximately unitarily equivalent
if and only if $[\phi]=[\psi]$ in $KL(A,B),$ $\phi_{\sharp}=\psi_{\sharp}$ and $\phi^{\ddag}=\psi^{\ddag}.$

Not surprisingly, the proof of this uniqueness theorem is based on the methods developed in the proof of
the classification result mentioned above, which can be found in
\cite{Lnclasn}, \cite{Lin-hmtp}, \cite{Lin-AU},
\cite{L-N} and \cite{Lin-Asy}.
Most technical tools are {developed} in those papers, either directly or implicitly.
In the present paper,
we will collect them 
and then assemble them into  production.

It should be noted that 
the above-mentioned uniqueness theorem still holds in a more general setting where
the source algebra $A$ is not necessary in the class $\mathcal C$. For example,
it is still valid for
all  AH-algebras $A$  which are not necessarily simple.
In particular, $A$ could be just $C(X)$ for any
compact metric space $X$.

In that situation, the first version of this kind of uniqueness theorem  was {proved} in \cite{Lin-Gong}, where $A=C(X)$ and
$B$ is a unital simple \CA\, with the unique tracial state and with real rank zero, stable rank one and weakly unperforated
$K_0(B).$

Then, in \cite{Lin-Tams-07},  it was shown that,
if $A=C(X)$ for some compact metric space $X$ and
$B$ is a unital simple \CA\, with tracial rank zero, then any unital monomorphisms $\phi$ and $\psi$ from $A$ to $B$ are approximately unitarily equivalent if and only if
$[\phi]=[\psi]$ in $KL(A,B)$ and $\phi_{\sharp}=\psi_{\sharp}.$ This result was
then generalized to the case that
$B$ has tracial rank  no more than one 
with the additional condition
$\phi^{\ddag}=\psi^{\ddag}$ in \cite{Lin-AU11}.   

From this point of view, the main result in this paper may also 
be regarded as a further generalization
of these uniqueness theorems. {In fact, in this paper, we also
allow the source algebra $A$ to be any unital \CA\, such that
$A\otimes U$ is a unital AH-algebra for all UHF-algebra $U$ of infinite type.}
One should also realize that these uniqueness theorems have a common
root: The Brown-Douglass-Fillmore theorem for essentially normal operators. One version of it can be stated as follows:
Two monomorphisms $\phi, \psi: C(X)\to B(H)/{\cal K}$---the Calkin algebra, which is a unital simple \CA\, with real rank zero---are unitarily equivalent if and only if $[\phi]=[\psi]$ in $KK(C(X), B(H)/{\cal K}).$

As this research was under way, we learned that H. Matui was conducting his
own investigation 
on the same problems. In fact, he proved the same uniqueness theorems mentioned under the assumption that $B\otimes U$ has tracial rank zero. Moreover, he actually showed the same result holds 
if the assumption that $B\otimes U$ has tracial rank zero is weaken to be 
that $B\otimes U$ has real rank zero, {stable rank one and weakly unperforated $K_0(B\otimes U),$} at least for the case
that quasi-traces are traces and there are only finitely many of extremal tracial states.

In \cite{L-N-Range}, it is shown
that, for any partially ordered simple weakly unperforated rationally Riesz group $G_0$ with order unit $u,$ any countable abelian group
$G_1,$ any metrizable Choquet simple $S,$ and any surjective affine continuous map $r: S\to S_u(G_0)$
(the state space of $G_0$) which preserves extremal points,
 there exists  one (and only one up to isomorphism)  unital
 separable simple amenable \CA\, $A\in {\cal C}$ which satisfies the UCT
 so that $\mathrm{Ell}(A)=(G_0, (G_0)_+, u, G_1, S, r).$

Then a natural question is:
Given two unital separable simple amenable C*-algebras
$A, \, B\in {\cal C}$ which satisfy the UCT,
and 
a \hm\, $\Gamma$ from ${\mathrm{Ell}}(A)$ to ${\mathrm{Ell}}(B),$
does there exist a
unital \hm\, $\phi: A\to B$ {which induces $\Gamma?$}
We will give an answer to this question.
Related to the uniqueness theorem discussed earlier and also related to the question {above}, one may also ask the following:
Given an element $\kappa\in  KL(A,B)$ which preserves the unit and order, an affine map $\lambda: \aff(\tr(A))\to \mathrm{Aff}(\tr(B))$ and
 a \hm\, $\gamma: U(A)/CU(A)\to U(B)/CU(B)$ which are compatible, does there exist a unital \hm\, $\phi: A\to B$
 so that  $[\phi]=\kappa,$ $\phi_{\sharp}=\lambda$ and $\phi^{\ddag}=\gamma?$ We will, at least, partially answer this question.


\section{Preliminaries}

\begin{NN}\label{NN1}
Let $A$ be a unital stably finite \CA. Denote by $\mathrm{T}(A)$ the simplex of tracial
states of $A$ and denote by $\textrm{Aff(T}(A))$ the space of all real
affine continuous functions on $\mathrm{T}(A).$ Suppose that $\tau\in \mathrm{T}(A)$ is a
tracial state. We will also {denote by} $\tau$ the trace $\tau\otimes
\mathrm{Tr}$ on $\textrm{M}_k(A)=A\otimes \textrm{M}_k(\Comp)$ (for every integer $k\ge 1$), where
$\mathrm{Tr}$ is the standard trace on $\mathrm{M}_k(\Comp).$  A trace $\tau$ is faithful if $\tau(a)>0$ for any $a\in A_+\setminus\{0\}.$ Denote by $\mathrm{T}_{\mathtt{f}}(A)$ the convex subset of $\mathrm{T}(A)$ consisting of all faithful tracial states.

Denote by $\textrm{M}_{\infty}(A)$ the set $\displaystyle{\cup_{k=1}^{\infty}\textrm{M}_k(A)},$ where $\textrm{M}_k(A)$ is regarded as a C*-subalgebra of $\textrm{M}_{k+1}(A)$ by the embedding
$a\mapsto  \begin{pmatrix} a & 0\\0 & a\end{pmatrix}.$
{For any projection $p\in \mathrm{M}_\infty(A)$,
the restriction $\tau\mapsto\tau(p)$ defines a positive affine function on
$\mathrm{T}(A)$. This induces a canonical positive homomorphism
$\rho_A: \Kzero(A)\to \textrm{Aff(T}(A))$.}

Denote by $\mathrm{U}(A)$ the unitary group of $A$, and denote by $\mathrm{U}(A)_0$ the connected component of $\mathrm{U}(A)$ containing the identity.
{Let} $C$ be another unital \CA\, and {let}  $\phi: C\to A$ be a unital *-homomorphism. Denote by $\phi_{\mathrm{T}}: \mathrm{T}(A)\to \mathrm{T}(C)$
the continuous affine map induced by $\phi,$ i.e., $$\phi_\mathrm{T}(\tau)(c)=\tau\circ \phi(c)$$ for all $c\in C$ and $\tau\in \mathrm{T}(A).$ Denote by $\phi_\sharp:\mathrm{Aff}(\mathrm{T}(C))\to \mathrm{Aff}(\mathrm{T}(A))$ {the map defined} by
$$
\phi_\sharp(f)(\tau)=f(\phi_{\mathrm{T}}(\tau)) \,\,\,{\text{ for\,\,\, all}}\,\,\,\tau\in\mathrm{T}(A).
$$
\end{NN}
\vspace{-0.2in}
\begin{defn}\label{ddag}
Let $A$ be a unital \CA.
Denote by $CU(A)$ the closure of the subgroup
generated by commutators {of $U(A)$}.  If $u\in U(A),$ {its} image in {the quotient} $U(A)/CU(A)$ will be denoted by $\overline{u}.$
Let $B$ be another unital \CA\, and let $\phi: A\to B$ be a unital \hm.
{it is clear that} $\phi$ maps $CU(A)$ into $CU(B).$ {Let $\phi^{\ddag}$ denote the} induced \hm\ from $U(A)/CU(A)$
into $U(B)/CU(B)$.

Let $n\ge 1$ be any integer. Denote by $U_n(A)$ the unitary group
of $M_n(A),$ and denote by $CU(A)_n$ the closure of commutator subgroup
of $U_n(A).$ {Regard} $U_n(A)$ as a subgroup of $U_{n+1}(A)$ via the embedding 
$ u\mapsto \left(\begin{array}{cc}u&0\\0&1\end{array}\right)$ and denote by $U_{\infty}(A)$ the union
of {all} $U_n(A).$

{Consider the union $CU_\infty(A):=\bigcup_n CU_n(A)$. It is then a normal subgroup of $U_\infty(A)$, and the quotient $U(A)_{\infty}/CU_{\infty}(A)$ is in fact isomorphic to the inductive limit
of $U_n(A)/CU_n(A)$ (as abelian groups).} 
We will use $\phi^{\ddag}$ for the \hm\ {induced by $\phi$} from
$U_{\infty}(A)/CU_{\infty}(A)$ into $U_{\infty}(B)/CU_{\infty}(B)$.

\end{defn}

\begin{defn}\label{DDet}
Let $A$ be a unital \CA, and {let} $u\in U(A)_0.$ {Let} $u(t)\in C([0,1],A)$ be
a piecewise-smooth path of unitaries such that $u(0)=u$ and $u(1)=1.$
Then the de la Harpe--Skandalis determinant of $u(t)$ is defined by
$$
{\mathrm{Det}}(u(t))(\tau)=\frac{1}{2\pi i}\int_0^1\tau({{d}u(t)\over{{d}t}}u(t)^*) {d}t \tforal \tau\in \tr(A),
$$
{{which induces}}
 a \hm\, $${{\mathrm{Det}}:U(A)_0\to \aff(\tr(A))/\overline{\rho_A(K_0(A))}}.$$

{The determinant ${{\mathrm{Det}}}$ can be extended to a map from} $U_{\infty}(A)_0$ into $\aff(\tr(A))/\overline{\rho_A(K_0(A))}.$
{It is easy to see that the determinant} vanishes on the {closure} of commutator subgroup
of $U_{\infty}(A).$ {In fact, by a result of K. Thomsen (\cite{Thomsen-rims}),
the closure of the commutator subgroup is exactly the kernel of this map,
that is, it induces an isomorphism} $\overline{{\mathrm{Det}}}: U_{\infty}(A)_0/CU_{\infty}(A)\to \aff(\tr(A))/\overline{\rho_A(K_0(A))}.$
{Moreover, by (\cite{Thomsen-rims}),
one has the following short exact sequence}
\beq\label{Texact}
 0\to \aff(\tr(A))/\overline{\rho_A(K_0(A))}{\to} U_{\infty}(A)/CU_{\infty}(A){\stackrel{\varPi}{\to}} K_1(A)\to 0
\eneq
which splits ({with the embedding of $\aff(\tr(A))/\overline{\rho_A(K_0(A))}$ induced by $(\overline{{\mathrm{Det}}})^{-1}$}). We will fix a splitting map $s_1: K_1(A)\to U_{\infty}(A)/CU_{\infty}(A).$  {The notation $\varPi$ and} $s_1$
will be used late without further warning.

For each $\bar{u}\in s_1(K_1(A))$, select and fix one element $u_c\in\bigcup_{n=1}^\infty M_n(A)$ such that $\overline{u_c}=\bar{u}$. Denote this set by $U_c(A)$.

In the case that {$A$ has tracial rank at most one (see \ref{TA1} below)},
by Corollary 3.4 of \cite{Thomsen-rims}, one has
$$U_{\infty}(A)_0/CU_{\infty}(A)=U(A)_0/CU(A)$$ and thus the following splitting short exact sequence:
\beq\label{Texact-2}
 0\to \aff(\tr(A))/\overline{\rho_A(K_0(A))}\to U(A)/CU(A)\to K_1(A)\to 0.
\eneq
\end{defn}



%

\begin{defn}\label{DKL}

Let $A$ be a  unital \CA\, and let $C$ be a  separable \CA\, which satisfies the Universal Coefficient Theorem.
{Recall that $KL(C, A)$ is the quotient of $KK(C, A)$ modulo pure extensions.}  By {a} result of D{\u a}d{\u a}rlat and Loring in \cite{DL}, {one has}
\beq\label{N2-1}
{KL}(C,A)=\mathrm{Hom}_{\Lambda}(\underline{{K}}(C), \underline{{K}}(A)),
\eneq
where
$$
\underline{{K}}(B)=({K}_0(B)\oplus{K}_1(B))\oplus( \bigoplus_{n=2}^{\infty}({K}_0(B,\Z/n\Z)\oplus{K}_1(B,\Z/n\Z) ) )
$$
for any \CA\, $B.$ { {Then, in the rest of the paper,}} we will identify {$KL(C, A)$ with $\mathrm{Hom}_{\Lambda}(\underline{{K}}(C), \underline{{K}}(A))$}. 


{Denote by $\kappa_i: K_i(C)\to K_i(A)$ the \hm\, given by $\kappa$ with $i=0,1,$} and
denote by ${KL}(C,A)^{++}$ the set of those ${\kappa}\in \mathrm{Hom}_{\Lambda}(\underline{{K}}(C), \underline{{K}}(A))$ such that
$$
{\kappa_0}(\Kzero^+(C)\setminus\{0\})\subseteq \Kzero^{+}(A)\setminus \{0\}.
$$
Denote by ${KL}_e(C, A)^{++}$ the set of those elements ${\kappa}\in {KL}(C,A)^{++}$ such that ${\kappa_0}([1_C])=[1_A].$
Suppose that both $A$ and $C$ are unital, $\mathrm{T}(C)\not=\O$ and $\mathrm{T}(A)\not=\O.$
Let $\lambda_\mathrm{T}: \mathrm{T}(A)\to \mathrm{T}(C)$ be a continuous affine map.
{Let $h_0: K_0(C)\to K_0(A)$ be a positive \hm.}
We say $\lambda_{\mathrm{T}}$ is compatible with {$h_0$} if
for any projection $p\in \mathrm{M}_{\infty}(C),$  ${\lambda_{{\mathrm{T}}}(\tau)(p)}=\tau({h_0([p]))}$ for all $\tau\in \mathrm{T}(A).$
{Let ${\lambda}: \aff(\tr_{\mathrm{f}}(C))\to \aff(\tr(A))$ be an affine continuous map.
We say $\lambda$ and $h_0$ {are} compatible if $h_0$ is compatible to $\lambda_{{\mathrm{T}}},$ where $\lambda_{{\mathrm{T}}}:\tr(A)\to  \tr_{\mathrm{f}}(C)$ is the map {$\lambda_{\mathrm{T}}(\tau)(a)=\lambda(a^*)(\tau), \forall a\in C^+$ and $\tau\in\tr(A)$, where $a^*\in\aff(\tr_{\mathrm{f}}(C))$ is the affine function induced by $a$.} 
We say $\kappa$ and $\lambda$ (or {$\lambda_T$}) are compatible, if $\kappa_0$ is positive and
$\kappa_0$ and $\lambda$ {are} compatible. }

Denote by ${KLT}_{{e}}(C,A)^{++}$ the set of those pairs
$({\kappa}, \lambda_{{\mathrm{T}}})$ { (or,} ${(\kappa, {\lambda})}$){,} where ${\kappa}\in {KL}_e(C,A)^{++}$
and $\lambda_{{\mathrm{T}}}: \mathrm{T}(A)\to \mathrm{T}_{\mathtt{f}}(C)$ {(or, ${\lambda}:
\aff(\tr_{\mathrm{f}}(C))\to \aff(\tr(A)))$} is a continuous affine map
which is compatible with ${\kappa}.$
{If ${\lambda}$ is compatible with $\kappa$, {then} ${\lambda}$ maps $\rho_{{C}}(K_0({C}))$ into $\rho_{{A}}(K_0({A})).$
Therefore ${\lambda}$ induces a continuous \hm\, $\overline{{\lambda}}:
\aff(\tr_{{\mathrm{f}}}({C}))/\overline{\rho_{{C}}(K_0({C}))}\to \aff(\tr(A))/\overline{\rho_{{A}}(K_0(A))}.$}
Suppose that ${\gamma}: U_{\infty}({C})/CU_{\infty}({C})\to U_{\infty}({A})/CU_{\infty}({A})$ is a {continuous} \hm\,
{and $h_i: K_i(C)\to K_i(A)$ are \hm s for which $h_0$ is positive. We say {that} ${\gamma}$ and $h_1$ are compatible
if ${\gamma}(U_{\infty}(C)_0/CU_{\infty}(C))\subset  U_{\infty}(A)_0/CU_{\infty}(A)$ and ${\gamma}\circ s_1=s_1\circ h_1,$
we say {that} $h_0, h_1, \lambda$ and $\gamma$ are compatible, if ${\lambda}$ and $h_0$ are compatible, {$\gamma$ and $h_1$ are compatible} and }
$$
{\overline{\mathrm{Det}}_{{A}}\circ{\gamma}|_{U_{\infty}({C})_0/CU_{\infty}({C})}=  \overline{{\lambda}}\circ \overline{\mathrm{Det}}_{{C}},}
$$
and  we also say {that} $\kappa, {\lambda}$ and  ${\gamma}$ are compatible, {if $\kappa_0,\kappa_1, {\lambda}$ and ${\gamma}$ are compatible.}

\end{defn}





\begin{NN}\label{snn}
For each prime number {$p$}, let $\eps_p$ be a number in $\{0, 1, 2, ..., +\infty\}$. Then a supernatural number is the formal product $\mathfrak{p}=\prod_p p^{\eps_p}$. {Here we insist that
 there are either infinitely many $p$ in the product, or,
 one of $\eps_p$ is infinite.} Two supernatural numbers $\mathfrak{p}=\prod_p p^{\eps_p(\p)}$ and $\mathfrak{q}=\prod_p p^{\eps_p(\q)}$ are relatively prime if for any prime number $p$, at most one of $\eps_p(\p)$ and $\eps_p(\q)$ is nonzero. A supernatural number $\p$ is called of infinite type if for any prime number, either $\eps_p(\p)=0$ or $\eps_p(\p)=+\infty$.
{For each }supernatural number $\mathfrak{p}$, there is a UHF-algebra $M_\mathfrak{p}$ associated to it, and the UHF-algebra is unique up to isomorphism  (see \cite{Dix-JFA}).
\end{NN}

\begin{NN}\label{ratn-pq}
{Denote by} $Q$ the UHF-algebra with ${(K_0(Q), K_0(Q)_+, [1_A])=(\Q, \Q_+, 1)}$
(the supernatural number associated to $Q$ is $\prod_p p^{+\infty}$), and let $M_\mathfrak{p}$ and $M_\mathfrak{q}$ be two UHF-algebras with $M_\mathfrak{p}\otimes M_\mathfrak{q}\cong Q$ and $\mathfrak{p}=\prod_p p^{\eps_p(\p)}$ and $\mathfrak{q}=\prod_p p^{\eps_p(\q)}$ relatively prime. Then it follows that $\p$ and $\q$ are of infinite type. Denote by
\beq\nonumber
\mathbb Q_{\mathfrak{p}}&=&\mathbb Z[\frac{1}{p_1}, ..., \frac{1}{p_n},...]\subseteq\mathbb Q,\quad\mathrm{where}\ \eps_{p_n}(\p)=+\infty {\andeqn}\\\nonumber
\mathbb Q_{\mathfrak{q}}&=&\mathbb Z[\frac{1}{p_1}, ..., \frac{1}{p_n},...]\subseteq \mathbb Q,\quad\mathrm{where}\ \eps_{p_n}(\q)=+\infty.
\eneq

Note that {$(K_0(M_\p), K_0(M_\p)_+, [1_{M_\p}])=(\Q_\p, (\Q_p)_+,1)$ and
$(K_0(M_\q), K_0(N_\q)_+,[1_{M_\q}])=(\Q_\q, (\Q_\q)_+, 1).$}
Moreover, $\mathbb Q_\mathfrak{p}\cap\mathbb Q_{\mathfrak{q}}={\Z}$
and $\mathbb Q=\mathbb Q_\mathfrak{p}+\mathbb Q_{\mathfrak{q}}$.
\end{NN}

\begin{NN}\label{zpq}
For any pair of {relatively prime} supernatural numbers $\mathfrak{p}$ and $\mathfrak{q}$, define the C*-algebra $\mathcal Z_{\mathfrak{p}, \mathfrak{q}}$ by
$$\mathcal Z_{\mathfrak{p}, \mathfrak{q}}=\{f: [0, 1]\to M_\p\otimes M_\q;\ f(0)\in M_{\mathfrak{p}}\otimes 1_{M_{\mathfrak{q}}}\ \textrm{and}\ f(1)\in 1_{M_{\mathfrak{p}}}\otimes M_{\mathfrak{q}}\}.$$

The Jiang-Su algebra $\mathcal Z$ is the unital inductive limit of dimension drop interval algebras with unique trace, and $(\Kzero(\mathcal Z), \Kzero(\mathcal Z), [1])=(\Int, \Int^+, 1)$ (see \cite{JS-Z}). By Theorem 3.4 of \cite{RW-Z}, for any pair of relatively prime supernatural numbers $\p$ and $\q$ of infinite type, the Jiang-Su algebra $\mathcal Z$ has a stationary inductive limit decomposition:
\begin{displaymath}
\xymatrix{
\mathcal{Z}_{\frak{p}, \frak{q}}\ar[r] &\mathcal{Z}_{\frak{p}, \frak{q}}\ar[r]&\cdots \ar[r]& \mathcal{Z}_{\frak{p}, \frak{q}}\ar[r]&\cdots\ar[r]&\mathcal Z
}.
\end{displaymath}

By Corollary 3.2 of \cite{RW-Z}, the C*-algebra $\mathcal Z_{\p, \q}$ absorbs
the Jiang-Su algebra:
$\mathcal{Z}_{\frak{p}, \frak{q}}\otimes\mathcal Z\cong
\mathcal{Z}_{\frak{p}, \frak{q}}.$
{A \CA\, $A$ is said to be {\em ${\cal Z}$-stable} if
$A\otimes {\cal Z}\cong A.$}


\end{NN}


\begin{defn}\label{TA1}
A unital simple C*-algebra $A$ has tracial rank at most one,
denoted by $\mathrm{TR}(A)\leq 1$, if for any finite subset
$\mathcal F\subset A$, any $\ep>0$, and any nonzero $a\in A^+$,
there exist {a} nonzero projection $p\in A$ and a C*-subalgebra
{$I\cong \bigoplus_{i=1}^m C(X_i)\otimes M_{r(i)}$}
with $1_I=p$ for some finite CW complexes $X_i$ with dimension at most  one
such that
\begin{enumerate}
\item $\norm{[x,\ p]}\leq\ep$ for any $x\in\mathcal F$,
\item for any $x\in\mathcal F$, there is $x'\in I$ such that $\norm{pxp-x'}\leq\ep$, and
\item $1-p$ is Murray-von Neumann equivalent to a projection in $\overline{aAa}$.
    \end{enumerate}

{Moreover}, if {the C*-subalgebra} $I$ {above} can be chosen to be a finite dimensional \CA,
then $A$ is said to have tracial rank zero, and in such case, we write
$\mathrm{TR}(A)=0.$ It is a theorem of Guihua Gong \cite{Gong-AH} that every unital simple AH-algebra
with no dimension growth has tracial rank  at most one.
It has been proved in \cite{Lnclasn} that every ${\cal Z}$-stable unital simple AH-algebra
has tracial rank at most one.
\end{defn}

\begin{defn}\label{Classes}
Denote by ${\cal N}$ the class of all separable amenable \CA s which satisfy the Universal Coefficient Theorem {(UCT)}.
Denote by ${\cal C}$ the class of all simple  \CA s $A$ for which ${\mathrm{TR}}(A\otimes M_\p)\le 1$ for some
UHF-algebra $M_\p,$ where $\p$ is a supernatural number of infinite type.
Note,  by \cite{L-N-Range}, that, if $\mathrm{TR}(A\otimes M_\p)\le 1$ for some supernatural number $\p$ then
$\mathrm{TR}(A\otimes M_\p)\le 1$ for all supernatural number $\p.$

Denote by ${\cal C}_0$ the class of all simple \CA s $A$ for which ${{\mathrm{TR}}}(A\otimes M_\p)=0$
for some supernatural number $\p$ {of infinite type} (and hence for all supernatural number $\p$ {of infinite type}).
\end{defn}

\begin{thm}[Theorem 5.10 \cite{Lin-AU11}]\label{Uniq}
Let $C$ be a unital AH-algebra
and let $A$ be a unital simple
C*-algebra with $\mathrm{TR}(A) \leq 1$. Suppose that $\phi,\psi : C\to A$ are two unital monomorphisms. Then $\phi$ and $\psi$ are approximately unitarily equivalent if and only if
$$[\phi] = [\psi]\ \mathrm{in}\ KL(C, A),$$
$$\phi_\sharp = \psi_\sharp\quad \textrm{and}\quad \phi^{\ddag} = \psi^{\ddag}.$$
\end{thm}

\begin{rem}
{One of the main purposes of this paper is to generalize this result so that $A$ can be allowed to be in the class $\mathcal C$, and $C$ can be rationally AH; that is, $C\otimes U$ is an AH-algebra for all UHF-algebra $U$ of infinite type.}
\end{rem}

\begin{NN}
Let $A$ and $B$ be two unital C*-algebras. Let $h: A\to B$ be a homomorphism and $v\in\mathrm{U}(B)$ be such that $$[h(g),\ v]=0\quad \textrm{for any}\ g\in A.$$ We then have a homomorphism $\overline{h}: A\otimes\mathrm{C}(\mathbb T)\to B$ defined by $f\otimes g\mapsto h(f)g(v)$ for any $f\in A$ and $g\in\mathrm{C}(\mathbb T)$. The tensor product induces two injective homomorphisms:
$$\beta^{(0)}:\Kzero(A)\to\Kone(A\otimes\mathrm{C}(\mathbb T)){\andeqn}
\beta^{(1)}:\Kone(A)\to\Kzero(A\otimes\mathrm{C}(\mathbb T)).$$
The second one is the usual Bott map. Note that, in this way, one writes $${K}_i(A\otimes\mathrm{C}(\mathbb T))={K}_i(A)\oplus\beta^{(i-1)}({K}_{i-1}(A)).$$ Let us use $\widehat{\beta^{(i)}}:{K}_i(A\otimes\mathrm{C}(\mathbb T))\to\beta^{(i-1)}({K}_{{i-1}}(A))$ to denote the {quotient map}.

For each integer $k\geq2$, one also has the following injective homomorphisms:
$$\beta_k^{(i)}:{K}_i(A, \Int/k\Int)\to{K}_{i-1}(A\otimes\mathrm{C}(\mathbb T), \Int/k\Int),\quad i=0, 1.$$ Thus, we write
$${K}_i(A\otimes\mathrm{C}(\mathbb T), \Int/k\Int)={K}_i(A, \Int/k\Int)\oplus\beta^{(i-1)}({K}_{i-1}(A), \Int/k\Int).$$ Denote by $\widehat{\beta^{(i)}_k}:{K}_i(A\otimes\mathrm{C}(\mathbb T), \Int/k\Int)\to\beta^{(i-1)}({K}_{i-1}(A), \Int/k\Int)$ the map analogous to $\widehat{\beta^{(i)}}$. If $x\in\underline{{K}}(A)$, we use $\boldsymbol{\beta}(x)$ for $\beta^{(i)}(x)$ if $x\in{K}_i(A)$ and for $\beta^{(i)}_k(x)$ if $x\in{K}_i(A, \Int/k\Int)$. Thus we have a map $\boldsymbol{\beta}: \underline{{K}}(A)\to\underline{{K}}(A\otimes\mathrm{C}(\mathbb T))$ as well as $\widehat{\boldsymbol{\beta}}:\underline{{K}}(A\otimes \mathrm{C}(\mathbb T))\to\boldsymbol{\beta}( \underline{{K}})$. Therefore, we may write $\underline{{K}}( A\otimes \mathrm{C}(\mathbb T))=\underline{{K}}(A)\oplus\boldsymbol{\beta}(\underline{{K}}(A))$.
On the other hand, $\overline{h}$ induces homomorphisms $$\overline{h}_{*i, k}: {K}_i(A\otimes\mathrm{C}(\mathbb T), \Int/k\Int)\to {K}_i(B, \Int/k\Int),$$ $k=0, 2, ... ,$ and $i=0, 1$.

We use $\mathrm{Bott}(h, v)$ for all homomorphisms $\overline{h}_{*i, k}\circ\beta_k^{(i)}$, and we use $\mathrm{bott}_1(h, v)$ for the homomorphism $\overline{h}_{1, 0}\circ\beta^{(1)}:\Kone(A)\to\Kzero(B)$, and $\mathrm{bott}_0(h, v)$ for the homomorphism $\overline{h}_{0, 0}\circ\beta^{(0)}:\Kzero(A)\to\Kone(B)$.
{$\mathrm{Bott}(h, v)$ as well as  $\mathrm{bott}_i(h, v)$ $(i=0,1$) may be defined for a unitary $v$ which
only approximately commutes {with} $h.$ In fact, given a finite subset $\mathcal P\subset\underline{{K}}(A)$, there exists a finite subset $\mathcal F\subset A$ and $\delta_0>0$ such that $$\mathrm{Bott}(h, v)|_{\mathcal P}$$ is well defined if $$\norm{[h(a), v]}<\delta_0$$ for all $a\in\mathcal F$. See 2.11 of \cite{Lin-Asy}, 2.10 of \cite{LnHomtp}, 2.21 of \cite{Lin-LAH} for more details.}

\end{NN}

We have the following generalized Exel's formula for the traces of Bott elements.
\begin{thm}[Theorem 3.5 of \cite{Lnclasn}]\label{Exel}
There is $\delta>0$ satisfying the following: Let $A$ be
a unital separable simple C*-algebra with $\mathrm{TR}(A)\leq 1$ and let $u, v \in U(A)$ be two unitaries such that $\|uv-vu\|<\delta$.
Then $\mathrm{bott}_1(u, v)$ is well defined and
$$\tau(\mathrm{bott}_1(u, v))=\frac{1}{2\pi i}(\tau(\log(vuv^*u^*)))$$ for all $\tau \in \mathrm{T}(A)$.
\end{thm}

\section{Rotation maps}

In this section, we collect several facts on {the} rotation map which are going to be used frequently in the rest of the paper. Most of them {can be found in the literature.}

\begin{defn}\label{D-Maf}
Let $A$ and $B$ be two unital C*-algebras, and let $\psi$ and $\phi$ be two
unital monomorphisms from $B$ to $A$. Then the mapping torus $M_{\phi, \psi}$ is the C*-algebra defined by $$M_{\phi, \psi}:=\{f\in\textrm{C}([0, 1], A);\ f(0)=\phi(b)\ \textrm{and}\ f(1)=\psi(b)\ \textrm{for some}\ b\in B\}.$$
%
%
For any $\psi, \phi\in\mathrm{Hom}(B, A)$, denoting by $\pi_0$ the evaluation of $M_{\phi, \psi}$ at $0$, we have the short exact sequence
\begin{displaymath}
\xymatrix{
0\ar[r]&\mathrm{S}(A)\ar[r]^{\imath}&M_{\phi,\psi}\ar[r]^{\pi_0}&B\ar[r]&0,}
\end{displaymath}
where $\mathrm{S}(A)=C_0((0,1),A).$
If {$\phi_{*i}=\psi_{*i}$ ($i=0,1$),}
then the corresponding six-term exact sequence breaks down to the following two extensions:
\begin{displaymath}
\eta_{{i}}(M_{\phi, \psi}):\ \xymatrix{
0\ar[r]& K_{i+1}(A)\ar[r]& K_{{i}}(M_{\phi, \psi})\ar[r]& K_{i}(B)\ar[r]&0,
}\,\,\,{(i=0,1).}
\end{displaymath}
\end{defn}

\begin{NN}
Suppose that, in addition,
\beq\label{Dr-2}
\tau\circ \phi=\tau\circ \psi\tforal \tau\in \textrm{T}(A).
\eneq
For any {continuous}  piecewise smooth path of unitaries $u(t)\in M_{\phi, \psi}$, consider the path of unitaries $w(t)=u^*(0)u(t)$ in $A$. Then it is a continuous {and piecewise smooth}  path with $w(0)=1$ {and
$w(1)=u^*(0)u(1).$}  Denote by $R_{\phi, \psi}(u)=\mathrm{Det}(w)$ the determinant of $w(t)$.
It is clear {{with the assumption of (\ref{Dr-2})}} that $R_{\phi, \psi}(u)$ depends only on the homotopy class of $u(t)$. Therefore, it induces a {\hm}, denoted by $R_{\phi, \psi}$, from $\Kone(M_{\phi, \psi})$ to $\mathrm{Aff}(\mathrm{T}(A)).$
\end{NN}


\begin{defn}\label{R0}
Fix two unital C*-algebras $A$ and $B$ with $\tr(A)\neq\O$. {Define} $\mathcal R_0$ to be the subset of $\mathrm{Hom}(\Kone(B), \aff(\tr(A)))$ consisting of those homomorphisms $h\in \mathrm{Hom}(\Kone(B), \aff(\tr(A)))$
{for which}  there exists a homomorphism $d: \Kone(B)\to\Kzero(A)$ such that $$h=\rho_A\circ d.$$ It is clear that $\mathcal R_0$ is a subgroup of $\mathrm{Hom}(\Kone(B), \aff(\tr(A)))$.
\end{defn}

\begin{NN}\label{R-bar}
If $[\phi]=[\psi]$ in $KK(B, A)$, {then} the exact sequences $\eta_i(M_{\phi, \psi})$ ($i=0, 1$) above  split. In particular, there is a lifting $\theta: \Kone(B)\to\Kone(M_{{\phi, \psi}})$. Consider the map $$R_{\phi, \psi}\circ\theta: \Kone(B)\to\aff(\tr(A)).$$
If a different lifting $\theta'$ is chosen, then,
$\theta-\theta'$ maps  $K_1(B)$ into  $K_0(A).$ Therefore  $$R_{\phi, \psi}\circ\theta-R_{\phi, \psi}\circ\theta'\in \mathcal{R}_0.$$

Then {define} $$\overline{R}_{\phi, \psi}=[R_{\phi, \psi}\circ\theta]\in \mathrm{Hom}(\Kone(B), \aff(\tr(A)))/\mathcal R_0.$$

If $[\phi]=[\psi]$ in $KL(B, A),$ then the exact sequences $\eta_i(M_{\phi, \psi})$ ($i=0, 1$) are pure, i.e., any finitely generated subgroup in the quotient groups has a lifting. In particular, for any finitely generated subgroup $G\subseteq \Kone(B)$, one has a map
$$R_{\phi, \psi}\circ\theta_G: G\to\aff(\tr(A)),$$
where $\theta_G: G\to K_1(M_{\phi, \psi})$ is a lifting.
Let $G\subset K_1(B)$ be a finitely generated subgroup.
Denote by ${\cal R}_{0, G}$ the set of those elements $h$ in $\mathrm{Hom}(G, {\mathrm{Aff}}(T(A)))$ such that there exists a \hm\, $d_G: G\to K_0(A)$ such that
$h|_{G}=\rho_A\circ d_G.$

If $[\phi]=[\psi]$ in $KL(B,A)$ and
$R_{\phi, \psi}({\Kone}(M_{\phi, \psi}))\subset \rho_A(K_0(A)),$
then $\theta_G\in {\cal R}_{0,G}$ for any
finitely generated subgroup $G\subset K_1(B)$ and any lifting $\theta_G.$
In this case, we will also write
$$
\overline{R}_{\phi, \psi}=0.
$$


See 3.4 of \cite{Lnclasn} for more details.

\end{NN}

\begin{lem}[Lemma 9.2 of \cite{Lnclasn}]\label{lem-dense}
Let $C$ and $A$ be {unital} C*-algebras {with} $\tr(A)\not=\O$.
Suppose that $\phi,\,\psi: C\to A$ are two unital homomorphisms such that
$$
[\phi]=[\psi]\ \mathrm{in}\ KL(C, A),\,\,\phi_{\sharp}=\psi_{\sharp}\andeqn \phi^{\ddag}=\psi^{\ddag}.
$$
Then the image of $R_{\phi, \psi}$ is in the $\overline{\rho_A(\Kzero(A))}{\subseteq} \mathrm{Aff}(\mathrm{T}(A))$.
\end{lem}

\begin{proof}
Let $z\in K_1(C).$ Suppose that $u\in U_n(C)$ for some integer $n\ge 1$ such that $[u]=z.$
Note that $\psi(u)^*\phi(u)\in CU_n(A).$
Thus, by \ref{DDet},  for any continuous and piecewise smooth path of unitaries $\{w(t): t\in [0,1]\}\subset U(A)$
with $w(0)=\psi(u)^*\phi(u)$ and $w(1)=1,$
\beq\label{lem-dense-1}
{\mathrm{Det}}(w)(\tau)=\int_0^1\tau({dw(t)\over{dt}}w(t)^*) dt\in \overline{\rho_A(K_0(A))}.
\eneq
Suppose that $\{(v)(t): t\in [0,1]\}$ is a continuous and piecewise smooth path of unitaries in $U_n(A)$ with
$v(0)=\phi(u)$ and $v(1)=\psi(u).$ Define $w(t)=\psi(u)^*v(t).$ Then
$w(0)=\psi^*(u)\phi(u)$ and $w(1)=1.$
Thus, by (\ref{lem-dense-1}),
\beq
R_{\phi, \psi}(z)(\tau)&=&\int_0^1 \tau({dv(t)\over{dt}}v(t)^*)dt\\
&=& \int_0^1\tau(\psi(u)^*{dv(t)\over{dt}}v(t)^*\psi(u))dt\\
&=&\int_0^1\tau({dw(t)\over{dt}}w(t)^*)dt\in \overline{\rho_A(K_0(A))}.
\eneq
\end{proof}

\begin{NN}\label{Exel-rot}
Let $A$ be a unital C*-algebra and let $u$ and $v$ be two unitaries with $\|u^*v-1\|<2$. Then $h=\frac{1}{2\pi i}\log(u^*v)$ is a well-defined self-adjoint element of $A$, and $w(t):=u\exp(2\pi i ht)$ is a smooth path of unitaries connecting $u$ and $v$. It is {a} straightforward calculation that for any $\tau\in\tr(A)$, $$\mathrm{Det}(w(t))(\tau)=\frac{1}{2\pi i}\tau(\log(u^*v)).$$
\end{NN}

\begin{NN}\label{Zero-rot}

Let $A$ be a unital C*-algebra, and let $u$ and $w$ be two unitaries. {Suppose  that  $w\in U_0(A).$}  Then $w=\prod_{k=0}^m \exp(2\pi i h_k)$ for some self-adjoint elements $h_0, ..., h_m$. Define the path
\begin{displaymath}
w(t)= (\prod_{k=0}^{l-1}\exp(2\pi i h_k))\exp(2\pi i h_l mt),\quad\textrm{if}\ t\in [(l-1)/m, l/m],
\end{displaymath}
and {define  $u(t)=w^*(t)uw(t)$ for $t\in [0,1].$} Then, $u(t)$ is {continuous and} piecewise smooth, and $u(0)=u$ and $u(1)=w^*uw$. A straightforward calculation shows that $\mathrm{Det}(u(t))=0.$

In general, if $w$ is not in the path-connected component {containing the identity}, one can consider unitaries $\mathrm{diag}(u, 1)$ and $\mathrm{diag}(w, w^*)$. Then, the same argument as above shows that there is a piecewise smooth path $u(t)$ of unitaries in $M_2(A)$ such that $u(0)=\mathrm{diag}(u, 1)$, $u(1)=\mathrm{diag}(w^*uw, 1)$, and $$\mathrm{Det}(u(t))=0.$$

\end{NN}

\begin{lem}[Lemma 3.5 of \cite{Lin-Asy}]\label{K0-image}
{Let $B$ and $C$ be two unital C*-algebras
with $\tr(B)\not=\O.$} Suppose that $\phi, \psi: C \to B$ are two unital monomorphisms such
that $[\phi]=[\psi]$ in $KL(C, B)$ and
$$ \tau\circ \phi = \tau \circ \psi$$ for all $\tau \in \tr(B).$
Suppose that $u \in U_l(C)$ is a unitary and $w \in U_l(B)$ such that
$$
{\|(\phi\otimes {\mathrm{id}}_{M_l})(u)w^*(\psi\otimes{\mathrm{id}}_{M_l})(u^*)w-1\|<2.}$$ Then, for any unitary $U\in U_l(M_{\phi, \psi})$ with $U(0)=(\phi\otimes {\mathrm{id}}_{M_l})(u)$ and $U(1)=(\psi\otimes {\mathrm{id}}_{M_l})(u)$, one has that
\beq\label{K0-1}
\frac{1}{2\pi i}\tau(\log({(\phi\otimes {\mathrm {id}}_{M_l})}(u^*)w^*({\psi\otimes {\mathrm{id}}_{M_l})}(u)w))-{R}_{\phi, \psi}([U])(\tau)\in \rho_B(K_0(B)).
\eneq
\end{lem}


\begin{proof}

Without loss of generality, one may assume that $u\in C$.  Moreover, to prove the {lemma}, it is enough to show that {(\ref{K0-1})} holds for one path of unitaries $U(t)$ in $M_2(B)$ with $U(0)=\mathrm{diag}(\phi(u), 1)$ and $U(1)=\mathrm{diag}(\psi(u), 1)$.

Let $U_1$ be the path of unitaries specified in \ref{Exel-rot} with $U_1(0)=\mathrm{diag}(\phi(u), 1)$ and $U_1(1/2)=\mathrm{diag}(w^*\psi(u)w, 1)$, and let $U_2$ be the path specified in \ref{Zero-rot} with $U_2(1/2)=\mathrm{diag}(w^*\psi(u)w, 1)$ and $U_2(1)=\mathrm{diag}(\psi(u), 1)$.

Set $U$ the path of unitaries by connecting $U_1$ and $U_2$. {Then} $U(0)=\mathrm{diag}(\phi(u), 1)$ and $U(1)=\mathrm{diag}(\psi(u), 1).$
 {By applying \ref{Exel-rot} and \ref{Zero-rot},} for any $\tau\in\tr(B)$, {one computes that}
$$R_{\phi, \psi}([U])=\mathrm{Det}(U(t))(\tau)=\mathrm{Det}(U_1(t))(\tau)+\mathrm{Det}(U_2(t))(\tau)=\frac{1}{2\pi i}\tau(\phi(u^*)w^*\psi(u)w),$$ as desired.
\end{proof}

\section{Homotopy lemma}

In this section, we collect several results from \cite{LN2} on the homotopy lemma.
\begin{defn}
Let $A$ be a unital \CA. In the following, for any invertible element $x\in A$, let $\langle x \rangle$ denote the unitary $x(x^*x)^{-\frac{1}{2}}$, and let $\overline{x}$ denote the element $\overline{\langle x \rangle}$ in $U(A)/CU(A)$. Consider a subgroup $\Int^k\subseteq \Kone(A)$, and write the unitary $\{u_1, ..., u_k\}\subseteq U_c(A)$ the unitary corresponding to the standard generators $\{e_1, e_2,...,e_k\}$ of $\Int^k$. Suppose that $\{u_1, u_2,...,u_k\}\subset M_n(A)$ for some integer $n\ge 1$. Let $\Phi: A\to B$ be a  unital positive linear map and $\Phi\otimes {\mathrm{id}}_{M_n}$ is
at least $\{u_1, ..., u_k\}$-$1/4$-multiplicative (hence each $\Phi\otimes {\mathrm{id}}_{M_n}(u_i)$ is invertible), then the map $\Phi^\ddag|_{s_1(\Int^k)}: \Int^k\to U(B)/CU(B)$ is defined by
$$\Phi^\ddag|_{s_1(\Int^k)}(e_i)= \overline{\langle\Phi\otimes {\mathrm{id}}_{M_n}(u_i)\rangle},\quad 1\leq i\leq k.$$
Thus, for any finitely generated subgroup
$G\subset {\overline{U_c(A)}},$ there exists $\dt>0$ and a finite subset ${\cal G}\subset A$ such that, for any unital $\dt$-${\cal G}$-multiplicative \morp\,
$L: A\to  B$ (for any unital \CA\, $B$), the map $L^{\ddag}$ is well defined
on  $s_1(G).$ (Please see \ref{Texact} for $U_c(A)$ and $s_1.$)
\end{defn}

The following  theorems are taken from \cite{LN2}.

\begin{thm}[3.10 of \cite{LN2}]\label{hp-mat}
Let $C=PM_n(\mathrm{C}(X))P$, where $X$ is a compact subset of a finite CW-complex and $P$ a projection in $M_n(\mathrm{C}(X))$ with an integer $n\geq 1$. Let $\Delta: (0, 1) \to (0, 1)$ be a non-decreasing map. For any $\eps>0$ and any finite subset $\mathcal F \subseteq C$, there exists $\delta>0$, $\eta>0$, $\gamma>0$, a finite subsets $\mathcal G \subseteq C$, $\mathcal P\subseteq \underline{K}(C)$,  a finite
 subset ${\cal Q}=\{x_1, x_2,...,x_k\}\subset K_0(C)$ which generates a free
 subgroup and $x_i=[p_i]-[q_i],$ where $p_i, q_i\in M_m(C)$
 (for some integer $m\ge 1$) are projections,
 satisfying the following:

Suppose that $A$  is a unital simple C*-algebra with $TR(A)\leq 1$, $\phi: C\to A$ is a unital homomorphism and $u\in A$ is a unitary, and suppose that
$$\norm{[\phi(c), u]}<\delta,\ \forall c\in\mathcal G\quad\mathrm{and}\quad \mathrm{Bott}(\phi, u)|_{\mathcal P}=0,$$ and
$$\mu_{\tau\circ \phi}(O_a)\geq\Delta(a)\ \forall \tau\in T(A\otimes D),$$
where $O_a$ is any open ball in $X$ with radius $\eta\leq a<1$ and $\mu_{\tau\circ \phi}$ is the Borel probability measure defined
by $\tau\circ \phi$. Moreover, for each $1\leq i\leq k$,
there is $v_i\in CU(M_m(A))$ such that
$$
\norm{\langle(1_m-\phi(p_i)+\phi(p_i)u)
(1_m-\phi(q_i)+\phi(q_i)u^*\rangle-v_i}<\gamma.
$$
Then there is a continuous path of unitaries $\{u(t): t\in[0, 1]\}$ in $A$ such that
$$u(0)=u, u(1)=1,\ \textrm{and}\ \norm{[\phi(c), u(t)]}<\eps$$
for any $c\in\mathcal F$ and for any $t\in[0, 1]$.
\end{thm}




\begin{thm}[3.14 of \cite{LN2}]\label{BB-exi-mat}
Let $C=PM_n(\mathrm{C}(X))P$, where $X$ is a compact subset of a finite CW-complex and $P$ a projection in $M_n(\mathrm{C}(X))$ for some integer $n\geq 1$. Let
 $G\subset K_0(C)$ be a finitely generated
 subgroup. Write
 ${G}=\Z^k\oplus {\mathrm{Tor}}({G})$ with
 $\Z^k$ generated by
$$
\{x_{1}=[p_1]-[q_1], x_2=[p_2]-[q_2], ..., x_{k}=[p_{k}]-[q_{k}]\},
$$
where $p_i, q_i\in M_m(C)$ (for some integer $m\ge 1$) are projections,
$i=1,...,k$.

Let $A$ be a simple C*-algebra with $TR(A)\leq 1$. Suppose that $\phi: C\to A$ is a monomorphism. Then, for any finite subsets  $\mathcal F\subseteq C$ and $\mathcal P\subseteq \underline{K}(C)$, any $\eps>0$ and $\gamma>0$, any homomorphism $$\Gamma: \Int^k \to U_0(A)/CU(A),$$ there is a unitary $w\in A$ such that
$$\norm{[\phi(f), w]}<\eps\quad\forall f\in\F$$
$$\mathrm{Bott}(\phi, w)|_{\mathcal P}=0,$$ and
$$\mathrm{dist}(
\overline{\langle({1}_m-\phi(p_i)+\phi(p_i)w)({1}_m-\phi(q_i)+\phi(q_i)w^*\rangle},
 \Gamma(x_i))<\gamma, \quad\forall 1\leq i\leq k,$$
where $U_0(A)/CU(A)$ is identified as $U_0(M_m(A))/CU(M_m(A))$, and the distance above is understood as the distance in $U_0(M_m(A))/CU(M_m(A))$.

\end{thm}

%

\begin{thm}[3.16 of \cite{LN2}]\label{B1B2-alg}
Let $C$ be an AH-algebra, and let $A$ be a simple C*-algebra with $TR(A)\leq 1$.  Suppose that $h: C\to A$ is a monomorphism. Then, for any $\eps>0$, any finite subset $\F\subseteq C$ and any finite subset $\mathcal{P}\subseteq\underline{K}(C)$, there is a C*-algebra $C'\cong PM_n(\mathrm{C}(X'))P$ for some finite CW-complex $X'$ with
$K_1(C')=\Z^k\oplus {\mathrm{Tor}}(K_1(C'))$ and a homomorphism $\iota: C'\to C$ with  $\mathcal P\subseteq[\iota](\underline{K}(C'))$, a finite subset $\mathcal{Q}\subseteq \Z^k \subset \Kone(C')$ and $\delta>0$ satisfying the following: Suppose that $\kappa\in\mathrm{Hom}_\Lambda(\underline{K}(C'\otimes {\mathrm{C}}(\T)), \underline{K}(A))$ with
$$\abs{\rho_A\circ\kappa(\boldsymbol{\beta}(x))(\tau)}<\delta,\quad\forall x\in\mathcal{Q},\ \forall \tau\in\mathrm{T}(A).$$
Then there exists a unitary $u\in A$ such that
$$\norm{[h(c), u]}<\eps\quad\forall c\in\F\quad\textrm{and}\quad \mathrm{Bott}(h\circ\iota, u)=\kappa\circ\boldsymbol{\beta}.$$

Moreover, there is a sequence of C*-algebras $C_n$ with the form $C_n=P_nM_{r(n)}(C(X_n))P_n$, where each $X_n$ is a finite CW-complex and $P_n\in M_{r(n)}(C(X_n))$ a projection, such that $C=\varinjlim(C_n, \phi_n)$ for a sequence of unital homomorphisms $\phi_n: C_n\to C_{n+1}$ and one may choose $C'=C_n$ and $\iota=\phi_n$ for some integer $n\geq 1$.
\end{thm}


\section{Approximately unitary equivalence}

First we begin with the following lemma which is a simple combination of the uniqueness theorem \ref{Uniq}
{and}
the proof of Theorem 4.2 in \cite{L-N}. In what follows, if ${\cal G}$ is a subset of a group,
 we will use $G({\cal P})$ for the subgroup generated by ${\cal G}.$


\begin{lem}\label{lem2}
Let $A$ be a {simple} C*-algebra with $\mathrm{TR}(A)\leq 1$, and let $C$ be {a unital} AH-algebra. If there are monomorphisms $\phi, \psi: C\to A$ such that
$$
[\phi]=[\psi]\,\,\,\text{in}\,\,\,KL(C,A),\ 
\phi_{\sharp}=\psi_{\sharp},\andeqn \phi^{\ddag}=\psi^{\ddag},
$$
then, for any $2>\eps>0$, any finite subset $\mathcal F\subseteq C$, any finite subset of unitaries $\mathcal P\subset U_n(C)$ for some $n\ge 1$, there exist {a finite subset
$\mathcal G\subset K_1(C)$}
with {$\overline{\mathcal P} \subseteq \mathcal G$
(where $\overline{\mathcal P}$ is the image of ${\cal P}$ in $K_1(C)$)} and $\delta>0$ such that, for any map $\eta: G(\mathcal G)\to {\mathrm{Aff}({\tr}(A))}$ with $\abs{\eta(x)(\tau)}<\delta$ for {all} $\tau\in\mathrm{T}(A)$ and $\eta(x)-{\overline{R}}_{\phi, \psi}(x)\in\rho_A(\Kzero(A))$ for all $x\in\mathcal G$,  there is a unitary $u\in A$ such that
$$\norm{\phi(f)-u^*\psi(f)u}<\eps\quad\forall f\in\mathcal F,$$
and ${\tau(\frac{1}{2\pi i}\log((\phi\otimes {\mathrm{id}}_{M_n}(x^*))(u\otimes 1_{M_n})^*(\psi\otimes {\mathrm{id}}_{M_n}(x))(u\otimes 1_{M_n})))}=\tau(\eta([x]))$ for all $x\in\mathcal P$ and for all $\tau\in\mathrm{T}(A).$
\end{lem}

\begin{proof}
Without loss of generality, one may assume that any element in $\mathcal F$ has norm at most one.
Let $\ep>0.$  Choose $\ep>\theta>0$ and a finite subset ${\cal F}\subset {\cal F}_0\subset C$ satisfying the following:
For all $x\in {\cal P},$
$
\tau(\frac{1}{2\pi i}\log({\phi(x^*)}w_j^*\psi(x)w_j))
$
is well defined and
\begin{eqnarray}\label{L1-01}
&&\tau(\frac{1}{2\pi i}\log(\phi(x^*)w_j^*\psi(x)w_j))\\
&=&\tau({1\over{2\pi i}}\log(\phi(x^*)v_1^*\psi(x)v_1)
+\cdots +\tau({1\over{2\pi i}}\log(\phi(x^*)v_j^*\psi(x)v_j))\tforal \tau\in  T(A),
\end{eqnarray}
whenever
$$
\|\phi(f)-v_j^*\psi(f)v_j\|<\theta\tforal f\in {\cal F}_0,
$$
where $v_j$ are unitaries in $A$ and $w_j=v_1\cdots v_j,$ $j=1,2,3.$
In the above, if $x\in U_n(C),$ we denote by $\phi$ and $\psi$ the extended maps  $\phi\otimes{\mathrm{id}}_{M_n}$
and $\psi\otimes {\mathrm{id}}_{M_n},$ and replace $w_j,$ and $v_j$ by
${\mathrm{diag}}(w_j,...,w_j)$ and ${\mathrm{diag}}(v_j,...,v_j),$ respectively.


Let $C'$, $\iota: C'\to C$, $\delta'>0$ (in the place of $\delta$) and $\mathcal G'\subseteq \Kone(C')$ (in the place of $\mathcal Q$) the constant and finite subset with respect to $C$ (in the place of $C$), $\mathcal F_0$ (in the place of $\F$), $\mathcal P$ (in the place of $\mathcal P$), and $\psi$ (in the place of $h$) required by \ref{B1B2-alg}. Put $\delta=\delta'/2$.

{Fix a decomposition $(\iota)_{*1}(C')=\Int^k\oplus \mathrm{Tor}((\iota)_{*1}(C'))$ (for some integer $k\geq 0$), and let $\mathcal G$ be a set of standard generators of $\Int^k$. Let $\mathcal G''\subset U_{m}(C)$ be a finite subset containing a representative for each element of $\mathcal G$. Without loss of generality, one may assume that $\mathcal P\subseteq \mathcal G''$.}
By {Theorem 5.10 of \cite{Lin-AU11}}, the maps $\phi$ and $\psi$ are approximately unitary equivalent. Hence, for any finite subset $\mathcal Q$ and any $\delta_1$, there is a unitary $v\in A$ such that $$\norm{\phi(f)-v^*\psi(f) v}<\delta_1,\quad\forall f\in\mathcal Q.$$

By choosing $\mathcal Q\supseteq \cal F_0$ sufficiently large and $\delta_1<\eta/2$ sufficiently small, the map $$[x]\mapsto\tau(\frac{1}{2\pi i}\log(\phi^*(x)v^*\psi(x)v)),\ x\in \mathcal G'',$$ induces a homomorphism
{$\eta_1: (\iota)_{*1}(\Kone(C'))\to\mathrm{Aff}(\mathrm{T}(A))$ (note that $\eta_1(\mathrm{Tor}((\iota)_{*1}(\Kone(C'))))=\{0\}$)}, and moreover, $\|\eta_1(x)\|<\delta$ for all $x\in \mathcal G$.


By Lemma \ref{K0-image}, {the image} of $\eta_1-\overline{R}_{\phi, \psi}$ is in $\rho(\Kzero(A))$.
Since $\eta(x)-{\overline{R}}_{\phi, \psi}(x)\in\rho_A(\Kzero(A))$
for all $x\in {\cal G},$  the image
{$(\eta-\eta_1)((\iota)_{*1}(\Kone(C')))$} is also in $\rho_A(\Kzero(A))$.
Since $\eta-\eta_1$ factors through $\Int^k$, there is a map {$h: (\iota)_{*1}(\Kone(C'))\to\Kzero(A)$} such that $\eta-\eta_1=\rho_A\circ h.$ Note that $|\tau(h(x))|<2\delta=\delta'$ for all
$\tau\in\mathrm{T}(A)$ and $x\in\mathcal G$.

{By the universal multi-coefficient theorem, there is $\kappa\in\mathrm{Hom}_{\Lambda}(\underline{K}(C'\otimes\mathrm{C}(\mathbb T)), \underline{K}(A))$ with
$$\kappa\circ\boldsymbol{\beta}|_{\Kone(C')}=h\circ(\iota)_{*1}.$$}

Applying \ref{B1B2-alg}, there is a unitary $w$ such that
$$\norm{[w, \psi(f)]}<\theta/2,\quad\forall f\in\mathcal F_0,$$ and
{$\mathrm{Bott}(w, \psi\circ\iota)=\kappa$. In particular, $\mathrm{bott}_1(w, \psi)(x)=h(x)$ for all
$x\in\mathcal P$.
}

Set $u=wv$. One then has $$\norm{\phi(f)-u^*\psi(f)u}<\theta,\quad\forall f\in{\mathcal F_0},$$ and for any $x\in\mathcal{P}$ and any $\tau\in\mathrm{T}(A)$,
\begin{eqnarray*}
\tau(\frac{1}{2\pi i}\log({\phi(x^*)}u^*\psi(x)u))&=&\tau(\frac{1}{2\pi i}\log(\phi(x)v^*w^*\psi(z)wv))\\
&=&\tau(\frac{1}{2\pi i}\log({\phi(x^*)}v^*\psi(x)v  v^*{\psi(x^*)}   w^*\psi(x)wv)\\
&=&\tau(\frac{1}{2\pi i}\log({\phi(x^*)}v^*\psi(x)v)+\tau(\frac{1}{2\pi i}\log({\psi(x^*)}   w^*\psi(x)w)\\
&=&\eta_1({[x]})(\tau)+h({[x]})(\tau)
=\eta([x])(\tau).
\end{eqnarray*}
\end{proof}

\begin{rem}\label{Rlem2}
In the case that $\mathrm{TR}(A)=0,$
{in fact} one can apply
Theorem 3.6 of \cite{Lnindiana} as the uniqueness theorem in which {case}
the condition $\phi^{\ddag}=\psi^{\ddag}$ is not needed, and moreover, one can apply Corollary 17.9 of \cite{LnHomtp} (homotopy lemma).  This special case of lemma is also observed by
H. Matui in \cite{M}.
\end{rem}

\begin{thm}\label{AHtoC}
Let $A$ be a {simple} C*-algebra with ${\mathrm{TR}}(A\otimes Q)\leq 1$, and let $C$ be  {a unital}  AH-algebra. Suppose {that} there are two unital monomorphisms $\phi, \psi: C\to A$ with
$$
[\phi]=[\psi]\,\,\,\text{in}\,\,\, KL(C,A), \,\,\,\phi_{\sharp}=\psi_{\sharp} \andeqn \phi^{\ddag}=\psi^{\ddag}.
$$
Then, for any finite subset $\mathcal F\subseteq C$, there exists
{a unitray} $u\in A\otimes \mathcal Z$ such that $$\norm{\phi(x)\otimes 1-u^*(\psi(x)\otimes 1)u}<\eps,\quad\forall x\in\mathcal{F}.$$
\end{thm}
\begin{proof}
We first note, by \cite{L-N-Range},
that $TR(A\otimes M_\fr)\le 1$ for any supernatural number.

Write
$C=\lim_{n\to\infty} (C_n,\phi_n),$ where each $C_n$ has the form
$P_nM_{m(n)}(C(X_n))P_n,$ where $X_n$ is a finite CW-complex and
$P_n\in M_{m(n)}(C(X_n))$ is a projection.
Let ${\cal F}\subseteq C$ be a finite subset, and let $\ep>0.$
Without loss of generality, we may assume that
${\cal F}\subseteq \phi_{n, \infty}(C_n)$ for some integer $n\ge 1.$
We may write $\phi_{n, \infty}(C_n)=PM_m(C(X))P,$ where $X$ is a
compact subset of a finite CW-complex.
 Then, to simplify notation, without loss of generality, in the rest
 of the proof, we may assume that $C=PM_m(C(X))P,$ where
 $X$ is a compact subset of a finite CW complex and $P\in M_m(C(X))$ is a projection.

Fix a metric on $X$. For any $a\in(0, 1)$, denote by
$$\Delta(a)=\inf\{\mu_{\tau\circ\psi}(O_a);\ \tau\in T(A), \textrm{$O_a$ an open ball of radius $a$ in $X$}\}.$$ Since $A$ is simple, one has that $0<\Delta(a)\leq 1$ and $\Delta(a)\to 0$ as $a\to 0$.

Assume that every element in $\mathcal F$ has norm at most one. Let $\p$ and $\q$ be a pair of relatively prime supernatural numbers of infinite type with $\mathbb Q_\p+\mathbb Q_\q=\mathbb Q$. Denote by $M_\p$ and $M_\q$ the UHF-algebras associated to $\p$ and $\q$ respectively.

 Let $\delta>0$, $\gamma>0$, $d>0$ (in place of $\eta$), $\mathcal G\subseteq C$ a finite subset,  $\mathcal P\subseteq \underline{K}(C)$ a finite subset and $\mathcal Q=\{x_1, ..., x_k\} \subseteq K_0(C)$
which generates a free subgroup required by Theorem \ref{hp-mat} corresponding to
$\mathcal F$, $\eps/2$ (in place of $\eps$) and $\Delta$. We may assume that $x_i=[p_i]-[q_i],$ where
 $p_i, q_i\in M_n(C)$ are projections and $i=1,2,...,k.$

In the rest of of the proof, for a \hm\, $h: C'\to B'$
(for any C*-algebras $C'$ and $B'$), we will use $h$ instead of $h\otimes \mathrm{id}_{M_n}:
M_n(C')\to M_n(B')$ when it is inconvenient.

Without loss of generality, we may assume that $\delta<\eps/2$ is small enough and $\mathcal G$ is large enough so that for any homomorphism $h: C\to A$, the maps $\mathrm{Bott}(h, u_j)$  and
$\mathrm{Bott}(h, w_j)$ are well defined and
$$
\mathrm{Bott}(h, w_j)=\mathrm{Bott}(h, u_1)+\cdots + \mathrm{Bott}(h, u_j)
$$
on the subgroup generated by $\mathcal P,$ if $u_j$ is any unitaries with $\norm{[h(x), u_j]}<\delta$ for all $x\in\mathcal G$, where $w_j=u_1\cdots u_j,$ $j=1,2,3,4.$

We may also assume that
\beq\label{AHtoC-add1}
\|h(p_i),\, u_j]\|<1/16\andeqn \|h(q_i),\, u_j]\|<1/16,\,\,\,1\le i\le k, j=1,2,3,4
\eneq
(by choosing larger ${\cal G}$ and smaller $\dt$)

Let $\imath_\fr: A\to A\otimes M_\fr$ be the embedding defined by
$\imath_\fr(a)=a\otimes 1$ for all $a\in A,$ where $\fr$ is a supernatural number.
Define $\phi_\fr=\imath_\fr\circ \phi$ and $\psi_\fr=\imath_\fr\circ \psi.$

For any supernatural number $\fr=\p, \q$, the C*-algebra
$A\otimes M_\fr$ has tracial rank at most one.
Denote by $C'=P' M_n(\mathrm{C}(X'))P'$,
$\imath: C'\to C$, $\delta_\fr$ (in place of $\dt$) and $\mathcal Q_\fr\subseteq \Kone(C')$
(in place of $\mathcal Q$) which generates a free subgroup
corresponding to $\delta/8$ (in place of $\ep$), $\mathcal G$, $\mathcal P$ and
$\psi_\fr$ required by Theorem \ref{B1B2-alg}.
Let $0<\delta_2<\min\{\delta_\p, \delta_\q, \eps, \gamma\}$, and
let $\mathcal H \subseteq \underline{K}(C')$ be a finite  set of generators.
Denoted by $\mathcal H_1=\mathcal H\cap\Kone(C')$, we may assume that ${\cal Q}_{\fr}\subset {\cal H}_1.$  Pick  a finite subset
{$\mathcal U\subset U_n(C)$ for some integer $n\ge 1$} such that any element in $\imath_{*1}(\mathcal H_1)$ has a representative in $\mathcal U$.
Let $S\subset C$ be a finite subset such that, if $u=(a_{ij})\in {\cal U},$
then $a_{i,j}\in S.$

{Furthermore, one may assume that $\delta_2$ is sufficiently small such that for any unitaries $z_1, z_2$ in a C*-algebra with tracial states, $\tau(\frac{1}{2\pi i}\log(z_iz_j^*))$ ($i, j=1,2,3$) is well defined and
\vspace{-0.1in}
$$
\tau(\frac{1}{2\pi i}\log(z_1z_2^*))=\tau(\frac{1}{2\pi i}\log(z_1z_3^*))+ \tau(\frac{1}{2\pi i}\log(z_3z_2^*))
$$
for any tracial state $\tau$, whenever $\|z_1-z_3\|<\delta_2$ and $\|z_2-z_3\|<\delta_2$.}

Let $\mathcal {Q}_1\subset\Kone(C)$ (in place of $\mathcal G$)
and $\delta_3$ (in place of $\delta$) be the finite subset and constant of Lemma \ref{lem2} with respect to {$\mathcal G\cup S$} (in place of ${\cal F}$), $\mathcal U$ (in place of ${\cal P}$) and $\delta_2/n^2$ (in place of $\ep$).

By Lemma \ref{lem-dense}, the image of $R_{\phi, \psi}$ is in the closure of $\rho_A(\Kzero(A))$. Note that kernel of $R_{\phi, \psi}$ contains $\mathrm{Tor}(G(\mathcal {Q}_1))$ and $G(\mathcal {Q}_1)$ is finitely generated. There exists a {\hm\,} $\eta: \mathcal {Q}_1 \to\mathrm{Aff}(\mathrm{T}(A))$ such that $\eta(x)-{\overline{R}}_{\phi, \psi}(x)\in \rho_A(\Kzero(A))$ and $\|\eta(x)\|<\delta_3$ for all $x\in \mathcal {Q}_1$.
Then the image of $(\imath_\p)_{\sharp}\circ \eta-{\overline{R}}_{\phi_{p}, \psi_\p}$
is in $\rho_{A\otimes M_\p}(\Kzero(A\otimes M_\p)))$. The same holds for $\q$.
 By Lemma \ref{lem2} there {exist} unitaries $u_\p$ and $u_\q$ such that
 $$\norm{\phi_\p(g)-u_\p^*\psi_\p(g)u_\p}<{\delta_2/n^2}
{\andeqn}
\norm{\phi_\q(g)-u_\q^*\psi_\q(g)u_\q}<\delta_2/n^2,
\quad\forall g\in\mathcal G \cup S.
$$
\vspace{-0.2in}
Moreover,
\beq\nonumber
\tau(\frac{1}{2\pi i}\log({\phi_\p(x^*)}u^*_\p\psi_\p(x)u_\p))
=(\imath_{\p})_{\sharp}\circ \eta({[x]})(\tau)\tforal \tau\in \mathrm{T}(A_\p)
\andeqn \\
\tau(\frac{1}{2\pi i}\log({\phi_\q(x^*)}u^*_\q\psi_\q(x)u_\q))
=(\imath_{\q})_{\sharp}\circ\eta({[x]})(\tau)\tforal \tau\in\mathrm{T}(A_\q)
\eneq
 and for all $x\in{\mathcal U},$
where we identify $\phi$ and $\psi$ with $\phi\otimes {\mathrm{id}}_{M_n}$ and $\phi\otimes {\mathrm{id}}_{M_n},$ and
$u$ with $u\otimes 1_{M_n},$ respectively.

Let $\infty$ be the supernatural number associated with $\Q.$
Let $e_\p: A\otimes M_\p\to A\otimes Q$ and $e_\q: A\otimes M_\q\to A\otimes Q$ be the standard embeddings.
Then, one computes that,  for all $x\in {{\cal U}}$, by the Exel formula (see {\ref{Exel}} ),
\begin{eqnarray}
\tau(\mathrm{bott}_1(\psi(x)\otimes 1, u_\p u^*_\q))&=&
 \tau(\frac{1}{2\pi i}\log(u_\p u_\q^*(\psi(x)\otimes1)u_\q u_\p^*({\psi(x^*)}\otimes1)))\\
&=&{\tau(\frac{1}{2\pi i}\log(u_\q^*(\psi(x)\otimes1)u_\q u_\p^*(\psi(x^*)\otimes 1)u_\p)}\\
&=&{\tau(\frac{1}{2\pi i}\log(u_\q^*(\psi(x)\otimes1)u_\q(\phi(x^*)\otimes 1))}\\
&&{+\tau({1\over{2\pi i}}\log((\phi(x^*)\otimes 1)u_\p^*(\psi(x)\otimes 1)
u_\p)}\\
&=&{-}(e_{{\q}})_{\sharp}\circ (\imath_{{\q}})_{\sharp}\circ \eta([x])(\tau){+}(e_\p)_{\sharp}\circ (\imath_\p)_{\sharp}\circ\eta([x])(\tau)\\
&=&{-}(\imath_{\infty})_{\sharp}\circ \eta([x])(\tau){+}(\imath_{\infty})_{\sharp}\circ \eta([x])(\tau)=0
\end{eqnarray}
for all $\tau\in T(A\otimes Q),$
{where we identify $\phi$ and $\psi$ with
$\phi\otimes {\mathrm{id}}_{M_n}$ and
$\psi\otimes {\mathrm{id}}_{M_n},$ and
$u_\p$ and $u_\q$ with $u_\p\otimes 1_{M_n}$ and $u_\q$ with
$u_\q\otimes 1_{M_n},$ respectively.}
Therefore, the image of the map $\mathrm{bott}_1(\psi\otimes 1, u_\p u^*_\q)$ is in
$\mathrm{\mathrm{ker}\rho}_{A\otimes Q}$. Note that $\Kzero(A\otimes Q)\cong\Kzero(A)\otimes \mathbb Q$ is torsion free. Hence the map $\mathrm{bott}_1(\psi\otimes 1, u_\p u^*_\q)$ factors through the torsion-free part of $G(\imath_{*1}(\mathcal H_1))$. Since $\mathcal H_1$ is a set of generators of $\Kone(C')$, one may assume that the domain of the map $\mathrm{bott}_1(\psi\otimes 1, u_\p u^*_\q)$ is $\imath_{*1}(\Kone(C'))$.
Note that there is a short exact sequence
$$
\xymatrix{
0\ar[r] & \ker \rho_A \ar[r] & \Kzero(A) \ar[r]^-{\rho_A} & \rho_A(\Kzero(A))\ar[r] & 0.
}
$$
Since $D:=\Ratn$, $\Ratn_\p$ or $\Ratn_\q$ is flat, one has
$$
\xymatrix{
0\ar[r] & \ker \rho_A\otimes D \ar[r] & \Kzero(A) \ar[r]^-{\rho_A\otimes\mathrm{id}_{D}}\otimes D & \rho_A(\Kzero(A))\otimes D\ar[r] \ar[r] & 0.
}
$$
Since the UHF-algebra $R:=Q$, $M_\p$ or $M_\q$ have unique trace, the map $\rho_A\otimes\mathrm{id}_D$ is the same as the map $\rho_{A\otimes R}$ if $\Kzero(A\otimes R)$ is identified as $\Kzero(A)\otimes D$ respectively.

Hence $\ker \rho_{A\otimes Q}=\ker\rho_A\otimes\Ratn$ and $\ker\rho_{A\otimes M_\fr}=(\ker\rho)\otimes \Q_\fr,$ $\fr=\p,$ or $\fr=\q.$
Moreover, since $\p$ and $\q$ are relative prime,  any rational number $r$ can be written as
$r=r_\p+r_\q$ with $r_\p\in\mathbb Q_\p$ and $r_\q\in\mathbb Q_\q$
(see \ref{ratn-pq}).
Since $\ker\rho_{A\otimes Q}$ is torsion free, $\mathrm{bott}_1((\psi\otimes 1\circ \imath)\otimes 1, u_\p u^*_\q)$
maps $\mathrm{Tor}(K_1(C'))$ to zero. Write $K_1(C')=\Z^r\oplus \mathrm{Tor}(K_1(C'))$ and let $\{e_1, e_2,...,e_r\}$ be a set of generators of $\Z^r.$
Suppose that $\mathrm{bott}_1((\psi\otimes 1\circ \imath)\otimes 1, u_\p u^*_\q)$ maps $e_i$ to $\sum_{j=1}^{m_i}x_{i,j}\otimes r_{i,j},$ where
$x_{i,j}\in \ker\rho_A$ and $r_{i,j}\in \Q,$ $j=1,2,...,m_i$ and  $i=1,2,...,r.$
There are $r_{i,j,\p}\in \Q_\p$ and $r_{i,j,\q}\in \Q_\q$ such that
$r_{i,j}=r_{i,j,\p}-r_{i,j, \q},$
$j=1,2,...,m_i$ and $i=1,2,...,r.$
Define two \hm s $\theta_\p: K_1(C')\to \ker\rho_{A\otimes M_\p}$ and $\theta_\q: K_1(C')\to \ker\rho_{A\otimes M_\q}$ as follows:
$(\theta_\fr)|_{\mathrm{Tor}(K_1(C'))}=0,$ $\fr=\p,\, \q.$
Define $\theta_\fr(e_i)=\sum_{j=1}^{m_i}x_{i,j}\otimes r_{i,j, \fr}$ by regarding $\sum_{j=1}^{m_i}x_{i,j}\otimes r_{i,j, \fr}$
as an element of $K_0(A\otimes M_\fr)$),
$\fr=\p,\,\q$ and $i=1,2,...,r.$ Then
$$\mathrm{bott}_1((\psi\otimes 1\circ \imath)\otimes 1, u_\p u^*_\q)=
{(j_\p)_{*0}\circ \theta_\p-(j_\q)_{*0}\circ \theta_\q},$$ where
$j_\fr: A\otimes M_\fr\to A\otimes Q$ is the embedding.
The same argument shows there are \hm s
 $\alpha_\p:
\Kzero(C')\to\Kone(A\otimes M_\p)$ and $\alpha_\q:
\Kzero(C')\to\Kone(A\otimes M_\q)$ such that $$\mathrm{bott}_0((\psi\circ \imath)\otimes 1, u_\p u^*_\q)=
(j_\p)_{*1}\circ \alpha_\p-(j_\q)_{*1}\circ \alpha_\q.$$

By the universal multi-coefficient theorem,  there is $\kappa_\p\in\mathrm{Hom}_{\Lambda}(\underline{K}(C'\otimes{\mathrm{C}}(\T)), \underline{K}(A\otimes M_\p))$ such that
\beq\label{AHtoC-n3}
\kappa_\p|_{\boldsymbol{\bt}(K_0(C'))}=-\af_\p\circ {\boldsymbol{\bt}}^{-1} \andeqn \kappa_\p|_{\boldsymbol{\bt}(K_1(C'))}=-\theta_\p\circ {\boldsymbol{\bt}}^{-1}.
\eneq
Similarly, there is $\kappa_\q\in \mathrm{Hom}_{\Lambda}(\underline{K}(C'\otimes{\mathrm{C}}(C(\T))), \underline{K}(A\otimes M_\q))$ such that
\beq\label{AHtoC-n4}
\kappa_\q|_{{\boldsymbol{\bt}}(K_0(C'))}=-\af_\q\circ {\boldsymbol{\bt}}^{-1}\andeqn
\kappa_\q|_{{\boldsymbol{\bt}}(K_1(C'))}=-\theta_\q\circ {\boldsymbol{\bt}}^{-1}.
\eneq

To apply \ref{B1B2-alg}, we verify that
\beq
|\rho_{A\otimes M_\p}\circ\kappa_\p({\boldsymbol{\bt}}(x))|=0<\dt_\p\tforal x\in {\cal Q}_\p\andeqn\\
|\rho_{A\otimes M_\q}\circ\kappa_\p({\boldsymbol{\bt}}(x))|=0<\dt_\q\tforal x\in {\cal Q}_\q.
\eneq
Then, by Theorem \ref{B1B2-alg}, there are unitaries $w_\p\in A\otimes M_\p$ and $w_\q\in A\otimes M_\q$ such that $$\|[w_\p, \psi_\p(g)]\|< \delta/8,\quad  \|[w_\q, \psi_\q(g)]\|<\delta/8,$$ for any $g\in \mathcal G$, and
$$
\mathrm{Bott}(\psi_\p\circ \imath, w_\p)=\kappa_\p\circ {\boldsymbol{\bt}} \quad\mathrm{and}\quad  \mathrm{Bott}(\psi_\q\circ \imath, w_\q)=\kappa_\q\circ{\boldsymbol{\bt}}.
$$

Consider the unitaries $v_\p=w_\p u_\p$ and $v_\q=w_\q u_\q$. One then has that
$$\norm{\phi(g)\otimes 1-u^*_\p w^*_\p(\psi(g)\otimes 1)w_\p u_\p}<\delta/4
{\andeqn} \norm{\phi(g)\otimes 1-u_\q^* w^*_\q(\psi(g)\otimes 1)w_\q u_\q}<\delta/4,\quad\forall g\in\mathcal G.$$
 Hence $$\norm{[w_\p u_\p u^*_\q w^*_\q, \psi(g)\otimes 1]}<\delta/2,\quad \forall g\in\mathcal G.$$
In the following computation, we use $\psi\otimes 1$
for the map from $C$ to $A\otimes Q$ induced by $\psi.$ We have, by (\ref{AHtoC-n3}) and
(\ref{AHtoC-n4}), that
 \beq\label{AHtoC-n10}
 &&\mathrm{bott}_0(\psi\otimes 1, w_\p u_\p u^*_\q w^*_\q)|_{K_0(C)\cap {\cal P}}\\&=&
 \mathrm{bott}_0(\psi\otimes 1, w_\p)|_{K_0(C)\cap {\cal P}}+
 \mathrm{bott}_0(\psi\otimes 1, u_\p u^*_\q)|_{K_0(C)\cap {\cal P}}+
 \mathrm{bott}_0(\psi\otimes 1, w^*_\q)|_{K_0(C)\cap {\cal P}}\\
 &=&-(j_\p)_{*1}\circ \af_\p|_{K_0(C)\cap {\cal P}}+((j_\p)_{*1}\circ \af_\p-(j_\q)_{*1}\circ \af_\q)|_{K_0(C)\cap {\cal P}}
 +(j_\q)_{*1}\circ \af_\q|_{K_0(C)\cap {\cal P}}=0.
 \eneq

The same computation shows that
 \beq\label{AHtoC-n11}
 &&\mathrm{bott}_1(\psi\otimes 1, w_\p u_\p u^*_\q w^*_\q)|_{K_1(C)\cap {\cal P}}\\&=&
 \mathrm{bott}_1(\psi\otimes 1, w_\p)|_{K_1(C)\cap {\cal P}}+
 \mathrm{bott}_1(\psi\otimes 1, u_\p u^*_\q)|_{K_1(C)\cap {\cal P}}+
 \mathrm{bott}_1(\psi\otimes 1, w^*_\q)|_{K_1(C)\cap {\cal P}}\\
 &=&-(j_\p)_{*0}\circ\theta_\p|_{K_1(C)\cap {\cal P}}+((j_\p)_{*0}\circ \theta_\p-
 (j_\q)_{*0}\theta_\q)|_{K_1(C)\cap {\cal P}}+(j_\q)_{*0}\circ \theta_\q|_{K_1(C)\cap {\cal P}}=0.
 \eneq

Since $K_i(A\otimes Q)$ is  torsion free ($i=0,1$), the aboves imply
that
\beq\label{AHtoC-12}
\mathrm{Bott}(\psi\otimes 1, w_\p u_\p u^*_\q w^*_\q)|_{\cal P}=0.
\eneq

By the construction of $\Delta$, it is clear that $$\mu_{\tau\circ(\psi\otimes 1)}(O_a)\geq\Delta(a)$$ for all $a$, where $O_a$ is any open ball of $X$ with radius $a$; in particular, it holds for all $a\geq d$.

For each $1\leq i\leq k$, define (see (\ref{AHtoC-add1}))
$$L_{i, w_\p u_\p}=\overline{\langle (\mathbf 1_n-\psi(p_i)\otimes 1+(\psi(p_i)\otimes 1)w_\p u_\p)(\mathbf 1_n-\psi(q_i)\otimes 1+(\psi(q_i)\otimes 1)u^*_\p w^*_\p)} \rangle$$ and
$$L_{i, w_\q u_\q}=\overline{\langle (\mathbf 1_n-\psi(p_i)\otimes 1+(\psi(p_i)\otimes 1)w_\q u_\q)(\mathbf 1_n-\psi(q_i)\otimes 1+(\psi(q_i)\otimes 1)u^*_\q w^*_\q) \rangle},$$
and define the map
$\Gamma_\p: \Int^k\to U(A\otimes M_\p)/CU(A\otimes M_\p)$ by $\Gamma_\p(x_i)= L_{i, w_\p u_\p}$ and the map $\Gamma_\q: \Int^k\to U(A\otimes M_\q)/CU(A\otimes M_q)$ by $\Gamma_\q(x_i)= L_{i, w_\q u_\q}$.

By Corollary \ref{BB-exi-mat}, there are unitaries $\zeta_\p\in A\otimes M_{\p}, \zeta_\q\in A\otimes M_\q$ such that $$\norm{[\zeta_\p, \psi(g)\otimes 1_{M_\p}]}<\delta/4, \quad \norm{[\zeta_\q, \psi(g)\otimes 1_{M_\q}]}<\delta/4,\quad\forall g\in\mathcal G$$
$$\mathrm{Bott}(\psi\otimes 1_{M_\p}, \zeta_\p)|_{\mathcal P}=0,\quad \mathrm{Bott}(\psi\otimes 1_{M_\q}, \zeta_\q)|_{\mathcal P}=0,$$ and for any $1\leq i\leq k$,
$$
\mathrm{dist}(L_{i, \zeta^*_\p }, \Gamma_\p(x_i))\leq\gamma/2\andeqn
\mathrm{dist}(L_{i, \zeta^*_\q }, \Gamma_\q(x_i))\leq\gamma/2,
$$
 where
$$L_{i, \zeta^*_\p}=\overline{\langle (\mathbf 1_n-\psi(p_i)\otimes 1_{M_\p}+(\psi(p_i)\otimes 1_{M_\p})\zeta^*_{\p})(\mathbf 1_n-\psi(q_i)\otimes 1_{M_\p}+(\psi(q_i)\otimes 1_{M_\p})\zeta_\p) \rangle},$$ and
$$L_{i, \zeta^*_\q}=\overline{\langle (\mathbf 1_n-\psi(p_i)\otimes 1_{M_\q}+(\psi(p_i)\otimes 1_{M_\q})\zeta^*_\q)(\mathbf 1_n-\psi(q_i)\otimes 1_{M_\q}+(\psi(q_i)\otimes 1_{M_\q})\zeta_\q)\rangle}.$$

In particular, if denote by $v_0=\zeta_\p w_\p u_\p u^*_\q w^*_\q\zeta_\q^*$, one has that for any $1\leq i\leq k$,
$$
\mathrm{dist}(\overline{\langle (\mathbf 1_n-\psi(p_i)\otimes 1_Q+
(\psi(p_i)\otimes 1_Q)v_0)(\mathbf 1_n-\psi(q_i)\otimes 1_Q+
(\psi(q_i)\otimes 1_Q)v_0^*)\rangle}, \overline{1_n})<\gamma.
$$

Then, by Theorem \ref{hp-mat}, there is a continuous path of unitaries $v(t)$
in $A\otimes Q$ such that $v(1)=1$ and $v(0)=v_0 $,
and
$$\norm{[v(t), \psi(x)\otimes 1_Q]}<\eps/2\quad\forall x\in\mathcal F,
\ {\forall} t\in[0, 1].
$$

Consider the unitary $u(t)=v(t)\zeta_\q w_\q u_\q\in A\otimes \mathcal Z_{\mathfrak{p}, \mathfrak{q}}$, and it has the property $$\norm{\phi(f)\otimes 1-u^*(\psi(f)\otimes 1)u}<\eps,\quad\forall f\in\mathcal F.$$
One then embeds ${\mathcal Z}_{\p, \q}$ into ${\mathcal Z}$ to get the desired conclusion.
\end{proof}

Recall that $\mathcal C$ is the class of all simple separable C*-algebras $A$ for which $\mathrm{TR(A\otimes M_\fr)}\leq 1$ form some UHF-algebra $M_\fr$, where $\fr$ is a supernatural number of infinite type.

\begin{cor}\label{C1}
Let $C$ be a unital AH-algebra and let $A$ be a
unital  separable simple ${\cal Z}$-stable C*-algebra in ${\cal C}.$ Let $\phi, \psi: C\to A$ be two unital monomorphisms. Then there exists a sequence of unitaries
$\{u_n\}\subset A$ such that
$$
\lim_{n\to\infty} u_n^*{\psi}(c)u_n=\phi(c)\tforal c\in C,
$$
if and only if
$$
[\phi]=[\psi]\,\,\,\text{in}\,\,\, KL(C,A),\ \phi_{\sharp}=\psi_{\sharp}
\andeqn \phi^{\ddag}=\psi^{\ddag}.
$$
\end{cor}

{
\begin{proof}
We only show the ``if" part. Suppose that $\phi$ and $\psi$ satisfy the condition.
Let $\ep>0$, and let ${\cal F}\subset C$ be a finite subset.
Then,
by \ref{AHtoC}, there exists a unitary
$v\in A\otimes {\cal Z}$ such that
\beq\label{C1-1}
\|v^*(\psi(a)\otimes 1)v-\phi(a)\otimes 1\|<\ep/3\tforal a\in {\cal F}.
\eneq
Let $\imath: A\to A\otimes {\cal Z}$ be defined by $\imath(a)=a\otimes 1$ for $a\in A.$
There exists an isomorphism $j: A\otimes {\cal Z}\to A$ such that
$j\circ \imath$ is approximately inner.
So there is a
unitaries  $w\in A$ such that
\beq\label{C1-2}
\|j(\psi(a)\otimes 1)-w^*\psi(a)w\|<\ep/3\andeqn
\|w^*\phi(a)w-j(\phi(a)\otimes 1)\|<\ep/3
\eneq
for all $a\in {\cal F}.$
Let $u=wj(v)w^*\in A$; then, for $a\in {\cal F},$
\beq
\|u^*\psi(a)u-\phi(a)\| &=&\|wj(v)^*w^*\psi(a)wj(v)w^*-\phi(a)\|\\
&\le &\|wj(v)^*w^*\psi(a)wj(v)w^*-wj(v)^*(j(\psi(a)\otimes 1)j(v)w^*\|\\
&&+\|wj(v)^*(j(\psi(a)\otimes 1)j(v)w^*-w(j(\phi(a)\otimes 1)w^*\|\\
&&+\|w(j(\phi(a)\otimes 1)w^*-\phi(a)\|\\
&<& \ep/3+\ep/3+\ep/3=\ep \,\rforal a\in {\cal F}.
\eneq
\end{proof}
}

{A version of the following is also obtained by H. Matui.}

\begin{cor}\label{C2}
Let $C$ be a unital AH-algebra and let $A$ be a unital separable simple \CA\, in ${\cal C}_0$ which is ${\cal Z}$-stable.  Suppose that $\phi, \psi: C\to A$ are two unital monomorphisms.
Then there exists a sequence of unitaries
$\{u_n\}\subset A$ such that
$$
\lim_{n\to\infty} u_n^*\phi(c)u_n=\psi(c)\tforal c\in C,
$$
if and only if
$$
[\phi]=[\psi]\,\,\,\text{in}\,\,\, KL(C,A),\ \phi_{\sharp}=\psi_{\sharp}
\andeqn \phi^{\ddag}=\psi^{\ddag}.
$$
\end{cor}

\begin{proof}
The proof is exactly the same as that of \ref{AHtoC} {and \ref{C1}}. At where Theorem
\ref{Uniq} is applied, one applies Theorem 3.6  of \cite{Lnindiana} instead. One also uses Remark \ref{Rlem2}.
\end{proof}

\begin{lem}\label{N2N}
Let $A$ be a unital C*-algebra such that
$A\otimes M_\fr$ is an AH-algebra for
any supernatural number $\fr$ of infinite type.
Let $B\in {\cal C}$ be a unital separable C*-algebra, and let $\phi, \psi: A\to B$ be two unital monomorphisms.
Suppose that
\beq\label{L1-1}
[\phi]=[\psi]\,\,\,\,{\text in}\,\,\, KL(A, B),\\
\phi_{\sharp}=\psi_{\sharp}\andeqn \phi^{\ddag}=\psi^{\ddag}.
\eneq
Let $\p$ and $\q$ be two relatively prime supernatural numbers of infinite type with $M_\p\otimes M_\q=Q$.
Then, for any $\ep>0$ and any finite subset ${\cal F}\subset A\otimes {\cal Z}_{\mathfrak{p},\mathfrak{q}},$ there exists a unitary
$v\in B\otimes {\cal Z}_{\mathfrak{p},\mathfrak{q}}$  such that
\beq\label{L1-2}
\|v^*((\phi\otimes \mathrm{id})(a)) v-(\psi\otimes \mathrm{id})(a)\|<\ep\rforal a\in {\cal F}.
\eneq
\end{lem}

{The proof of this lemma will be lengthy and technical in {nature}. However, the outline is the same as that of Theorem \ref{AHtoC}, that is, using homotopy lemmas, one could find a certain path of unitaries in $B\otimes Q$ such that it implements the approximate equivalence above when it is regarded as a unitary in $B\otimes\mathcal Z_{\mathfrak{p}, \mathfrak{q}}$. But since the domain C*-algebra $A$ is only assumed to be rational {tracial rank at most one,} in order to apply the homotopy lemmas, one also {needs} to interpolate paths in $A\otimes{\cal Z}_{\mathfrak{p},\mathfrak{q}}$, and this increases the technical difficulty of the proof.

\begin{proof}

Let $\fr$ be a supernatural number. Denote by $\imath_\fr: A\to A\otimes M_\fr$ the embedding defined by $\imath_\fr(a)=a\otimes 1$ for all $a\in A.$
Denote by $j_\fr: B\to B\otimes M_\fr$ the embedding defined by $j_\fr(b)=b\otimes 1$ for all $b\in B.$
{Without loss of generality, one may assume that $\mathcal F=\mathcal F_1\otimes\mathcal F_2,$
where $\mathcal F_1\subseteq A$ and $\mathcal F_2\subseteq {\cal Z}_{p, q}$ are finite subsets
and $1_A\in {\cal F}$ and $1_{{\cal Z}_{\p,\q}}\in {\cal F}_2.$} Moreover, one may assume that any element in $\mathcal F_1$ or $\mathcal F_2$ has norm at most one.

Let $0=t_0<t_1<\cdots<t_m=1$ be a partition of $[0, 1]$ such that
\begin{equation}\label{L3-01}
\norm{b(t)-b(t_i)}<\ep/4\quad\forall b\in\mathcal F_2,\ \forall t\in[t_{i-1}, t_i],\ i=1,...,m.
\end{equation}
Consider
\beq\nonumber
\mathcal E&=&\{a\otimes b(t_i);\ a\in\mathcal F_1, b\in\mathcal F_2, i=0,..., m\}\subseteq A\otimes Q{,}\\
\mathcal E_\p&=&\{a\otimes b(t_0);\ a\in\mathcal F_1, b\in\mathcal F_2\}\subseteq A\otimes M_\p\subset A\otimes Q
\andeqn\\
\mathcal E_\q&=&\{a\otimes b(t_m);\ a\in\mathcal F_1, b\in\mathcal F_2\}\subseteq A\otimes M_\q\subset A\otimes Q.
\eneq


Since $A\otimes Q$ is an AH-algebra, without loss of generality, one may assume that the finite subset $\mathcal E$ is in a C*-subalgebra of $A\otimes Q$ which is isomorphic to $C:=PM_n(\mathrm{C}(X))P$ (for some $n\geq 1$) for some compact metric space $X$.
Since $PM_n(\mathrm{C}(X))P=\lim_{m\to\infty}(P_mM_n(\mathrm{C}(X_m))P_m)$,
where $X_m$ are closed subspaces of finite CW-complexes, then, without
loss of generality, one may assume further that $X$ is a closed subset of a
finite CW-complex.

Fix a metric on $X$, and for any $a\in(0, 1)$, denote by $$\Delta(a)=\inf\{\mu_{\tau\circ(\phi\otimes\mathrm{id})}(O_a);\ \tau\in T(B), \textrm{$O_a$ an open ball of radius $a$ in $X$}\}.$$ Since $B$ is simple, one has that $0<\Delta(a)\leq 1$.


Let $\mathcal H\subset { C}$, $\mathcal P\subseteq \underline{K}({C})$, $\mathcal Q=\{{x_1, x_2, ..., x_m}\}\subset K_0(C)$
which generates a free subgroup of $\Kzero(C)$,
$\delta>0$, ${\gamma>0}$, and {$d>0$ (in the place of $\eta$)} be the constants  of Theorem \ref{hp-mat} with respect to $\mathcal E$, $\ep/8$, and $\Delta$.
{ We may assume that
$x_i=[p_i]-[q_i],$ where $p_i, q_i\in M_n(C)$ are projections
(for some integer $n\ge 1$), $i=1,2,...,m.$ {Moreover, we may assume
that $\gamma<1.$}}

Denote by $\infty$ the supernatural number associated with $\Q.$
Let ${\cal P}_i={\cal P}\cap K_i(A\otimes Q),$ $i=0,1.$
There is a finitely generated free subgroup $G({\cal P})_{i,0}\subset K_i(A)$
such that if one sets
\beq
G({\cal P})_{i,\infty, 0}={G}(\{ gr: g\in (\imath_{\infty})_{*i}(G({\cal P})_{i,0})\andeqn
r\in {D_0}\}),
\eneq
where $1\in {D_0}\subset \Q$ is a finite subset, then
$G({\cal P})_{i,\infty,0}$ contains the subgroup generated by ${\cal P}_i,$ $i=0,1.$
Moreover, we may assume that, if $r=k/m,$ where $k$ and $m$ are
nonzero integers, and $r\in {D_0},$ then
$1/m\in {D_0}.$ Let ${\cal P}_i'\subset K_i(A)$ be a finite subset
which generates $G({\cal P})_{i,0},$ $i=0,1.${Also denote by $\mathcal P'=\mathcal P'_0\cup\mathcal P'_1$.}

{Denote by $j: C\to A\otimes Q$ the embedding.}

{
Write the subgroup generated by the image of $\mathcal Q$ in $\Kzero(A\otimes Q)$ as  $\Int^k$ (for some integer $k\geq 1$). Choose $\{x'_1, ..., x'_k\}\subseteq \Kzero(A)$ and $\{r_{ij};\ 1\leq i\leq m, 1\leq j\leq k\}\subseteq \Ratn$ such that
$$
{j_{*0}(x_i)}=\sum_{j=1}^k r_{ij} x'_j,\quad 1\leq i\leq m,\ 1\leq j\leq k,
$$
and moreover, $\{x'_1, ..., x_k'\}$ generates a free subgroup of $\Kzero(A)$ of rank $k$.
Choose projections $p'_j, q'_j\in M_n(A)$ such that $x'_j=[p_j']-[q_j']$, $1\leq j\leq k$.
Choose an integer $M$ such that $Mr_{ij}$ are integers}  for $1\le i\le m$ and $1\le j\le k.$ In particular $Mx_i$ is the linear combination of $x_j'$ with integer coefficients.

{
Also noting that the subgroup of $\Kzero(A\otimes Q)$ generated by $\{(\imath_\infty)_{*0}(x_1'), ..., (\imath_\infty)_{*0}(x_k')\}$ is isomorphic to $\Int^k$} and the subgroup of $\Kzero(A\otimes M_\fr)$ generated by $\{(\imath_\fr)_{*0}(x_1'), ..., (\imath_\fr)_{*0}(x_k')\}$ has to be isomorphic to $\Int^k$, where $\fr=\p$ or $\fr=\q.$

{Since $A\otimes M_\fr$ is an AH-algebra, one can choose a $C^*$-subalgebra $C_\fr$  of $A\otimes M_\fr$ which is isomorphic to $P_\fr M_{n_\fr}(\mathrm{C}(X_\fr))P_\fr $} (for some $n_\fr\geq 1$) such that $\mathcal E_\fr\subseteq C_\fr$ and projections $\{p'_{1, \fr}, ..., p'_{k, \fr}, q_{1, \fr}', ..., q_{k, \fr}'\}\subseteq M_n(C_\fr)$ such that for any $1\leq j\leq k$,
\begin{equation}\label{approx-est1}
\norm{p'_j\otimes 1_{M_\fr}-p'_{j, \fr}}<\gamma/(32(1+\sum_{i, j'} \abs{Mr_{ij'}}))<1
\end{equation} and
\begin{equation}\label{approx-est2}
\norm{q'_j\otimes 1_{M_\fr}-q'_{j, \fr}}< \gamma/(32(1+\sum_{i, j'} \abs{Mr_{ij'}}))<1,
\end{equation}
where $X_\fr$ is a closed subset of a finite CW-complex, and $\fr=\p$ or $\fr=\q.$

Denote by $x'_{j, \fr}=[p'_{j, \tau}]-[q'_{j, \fr}]$, $1\leq j\leq k$, and denote by $G_\fr$ the subgroup of $\Kzero(C_\fr)$ generated by $\{x'_{1, \fr}, ..., x'_{k, \fr}\}$, and write $G_\fr=\Int^r\oplus\mathrm{Tor}(G_\fr)$. Since $G_\fr$ is generated by $k$ elements, one has that $r\leq k$ and $r=k$ if and only if $G_\fr$ is torsion free. Note that the image of $G_\fr$ in $\Kzero(A\otimes M_\fr)$ is the group generated by $\{[p'_1\otimes 1_{M_\fr}]-[q'_1\otimes 1_{M_\fr}], ..., [p'_k\otimes 1_{M_\fr}]-[q'_k\otimes 1_{M_\fr}] \}$, which is isomorphic to $\Int^k$ (with $\{[p'_j\otimes 1_{M_\fr}]-[q'_j\otimes 1_{M_\fr}];\ 1\leq j\leq k\}$ as the standard generators). Hence $G_\fr$ is torsion free and $r=k$.

Without loss of generality, one may assume that $\imath_\fr(\mathcal P')\subseteq\underline{K}(C_\fr)$.
{Assume that $\mathcal H$ is sufficiently large and $\delta$ is sufficiently small such that for any homomorphism $h$ from $A\otimes Q$ to $B\otimes Q$ and any unitary $z_j$ {($j=1,2,3,4$),} the map $\mathrm{Bott}(h, z_j)$
and ${\mathrm{Bott}}(h, w_j)$ {are} well defined on the subgroup generated by $\mathcal P$ and
$$
\mathrm{Bott}(h, w_j)=\mathrm{Bott}(h, z_1)+\cdots +\mathrm{Bott}(h, z_j)
$$
on the subgroup generated by ${\cal P},$
if $\|[h(x), z_j]\|<\delta$ for any $x\in\mathcal H,$ where $w_j=z_1\cdots z_j,$ $j=1,2,3,4.$  }

{By choosing larger ${\cal H}$ and smaller $\dt$, one may also assume that
\beq\label{AHtoC-add2}
\|h(p_i),\, z_j]\|<1/16\andeqn \|h(q_i),\, z_j]\|<1/16,\,\,\,1\le i\le m, j=1,2,3,4,
\eneq
and for any $1\leq i\leq m$,

\begin{equation}\label{AHtoC-add3}
\mathrm{dist}(\zeta_{i, z_1}^M, \prod_{j=1}^k (\zeta'_{j, z_1})^{Mr_{ij}})<\gamma/8,
\end{equation}
where
$$\zeta_{i, z_1}=\overline{\langle (\mathbf 1_n-h(p_i)+ h(p_i))z_1)(\mathbf 1_n-h(q_i)+h(q_i))z^*_1)\rangle},$$ and
$$\zeta'_{j, z_1}=\overline{\langle (\mathbf 1_n-h(p'_j\otimes 1_{A\otimes Q})+ h(p'_j\otimes 1_{A\otimes Q}))z_1)(\mathbf 1_n-h(q'_j\otimes 1_{A\otimes Q})+h(q'_j\otimes 1_{A\otimes Q}))z^*_1)\rangle}.$$

By choosing even smaller $\dt,$ without loss of generality,
we may assume that
$$
{\cal H}={\cal H}^0\otimes {\cal H}^\p\otimes {\cal H}^\q,
$$
where ${\cal H}^0\subset A,$
${\cal H}^\p\subset M_\p$ and ${\cal H}^\q\subset M_\q$ are finite subsets,
and $1\in {\cal H}^0,$ $1\in {\cal H}^\p$ and $1\in {\cal H}^\q.$

Moreover, choose $\mathcal H^0$}, {$\mathcal H^\p$ and $\mathcal H^\q$} even larger and $\delta$ even smaller so that for any homomorphism $h_\fr: A\otimes M_\fr\to B\otimes M_\fr$ and unitaries $z_1, z_2\in B\otimes M_\fr$ with $\norm{h_\fr(x), z_i}<\delta$ for any $x\in\mathcal H_0\otimes\mathcal H_\fr$, one has
\beq\label{AHtoC-add4}
\|h_\fr(p'_{i, \fr}),\, z_j]\|<1/16\andeqn \|h_\fr(q'_{i, \fr}),\, z_j]\|<1/16,\,\,\,1\le i\le k, j=1,2,
\eneq
and
$$\mathrm{dist}(\zeta_{i, z_1z_2}, \overline{(1_{B\otimes M_\fr})_n})<\mathrm{dist}(\zeta_{i, z^*_1}, \zeta_{i, z_2})+\gamma/(32(1+\sum_{i', j}\abs{Mr_{i'j}})),$$ where
$$
\zeta_{i, z'}=\overline{\langle (\mathbf 1_n-h_\fr(p'_{i, \fr})+ h_\fr(p'_{i, \fr}))z')(\mathbf 1_n-h_\fr(q'_{i, \fr})+h_\fr(q'_{i, \fr}))(z')^*)\rangle}, \quad z'=z_1z_2, z^*_1, z_2.
$$

Denote by {$C'=P'M_n(\mathrm{C}(\tilde{X}))P'$,}
${\iota}: C'\to A\otimes Q$, $\delta_2$ {(in the place of $\dt$)} {the constant}, {$\mathcal G\subseteq \Kone(\mathrm{C}(\tilde{X}))$
(in the place of $\mathcal Q$)} the finite subset in
Theorem \ref{B1B2-alg} with respect to {$A\otimes Q$ (in the place of $C$), $B\otimes Q$ (in the place of $A$), $\phi\otimes\mathrm{id}_Q$ (in the place of $h$)}, $\delta/4$
{(in the place of $\epsilon$)}, $\mathcal H$ {(in the place of $\F$)} and $\mathcal P$.
{Note that $\tilde{X}$ is a finite CW-complex.}

Let $\mathcal H'\subseteq A\otimes Q$
be a finite subset and assume that $\delta_2$ is small enough such that for any homomorphism $h$ from $A\otimes Q$ to $B\otimes Q$ and
any unitary $z_j${ ($j=1,2,3,4$),} the map $\mathrm{Bott}(h, z_j)$
and ${\mathrm{Bott}}(h, w_j)$ is well defined on the subgroup $[\iota](\underline{K}(C'))$
and
$$
\mathrm{Bott}(h, w_j)=\mathrm{Bott}(h, z_1)+\cdots +\mathrm{Bott}(h, z_j)
$$
on the subgroup $[\iota](\underline{K}(C'))$, 
if $\|[h(x), z_j]\|<\delta_2$ for any $x\in\mathcal H',$ where $w_j=z_1\cdots z_j,$ $j=1,2,3,4.$
Furthermore, as above,
one may assume, without loss of generality, that
$$
{\cal H}'={\cal H}^{0'}\otimes {\cal H}^{\p'}\otimes {\cal H}^{\q'},
$$
where ${\cal H}^0\subseteq {\cal H}^{0'}\subset A,$
${\cal H}^\p\subseteq {\cal H}^{\p'}\in M_\q$ and ${\cal H}^\q\subseteq {\cal H}^{\q'}\subset M_\q$ are finite subsets.

Let $\delta'_2>0$ be a constant such that for any unitary with $\|u-1\|<\delta'_2$, one has that $\|\log u\|<\delta_2/4$. Without loss of generality, one may assume that $\delta'_2<\delta_2/4<\eps/4$ and $\delta_2'<\delta$.

Let $C'_\fr:=P_\fr M_{n}\mathrm{C}(X'_\fr) P_\fr$ (in the place of $C'$), $\iota'_\fr: C_\fr\to A\otimes M_\fr$ (in the place of $\iota$), $\mathcal R_\fr\subset \Kone(C'_\fr))$ (in the place of $\mathcal Q$)
and $\delta_\fr$ (in the place of $\delta$) be the finite subset and constant of Theorem \ref{B1B2-alg} with respect to $A\otimes M_\fr$ (in the place of $C$), $B\otimes M_\fr$ (in the place of $A$), $\phi\otimes\mathrm{id}_{M_\fr}$
(in the place of $h$), ${\cal H}^{0'}\otimes {\cal H}^{\fr'}$ (in place of ${\cal F}$)
and $(\imath_\fr)_{*0}({\cal P}_0') \cup (\imath_\fr)_{*1}({\cal P}_1')$
(in the place of $\mathcal P$) and $\delta'_2/8$ (in place of $\ep$)
 ($\fr=\p$ or $\fr=\q$). Note that $X'_\fr$ is a finite CW-complex with $\Kone(C'_\fr)=\Int^{k_\fr}\oplus\mathrm{Tor}(\Kone(C'_\fr))$.
Let ${\cal R}_\fr^{(i)}={(\iota_\fr')_{*i}(K_i(C'_\fr))},$ $i=0,1$. There is a finitely generated subgroup
$G_{i,0,\fr}\subset K_i(A)$ and a finitely generated subgroup $D_{0, \fr}\subseteq \Ratn_\fr$ so that
$$
G_{i,0,\fr}':=G(\{gr: g \in (\imath_\fr)_{*i}(G_{i,0,\fr})\andeqn
r\in  D_{0, \fr}\})
$$
contains the subgroup ${\cal R}_\fr^{(i)},$ $i=0,1.$ {Without loss of generality, one may assume that $D_{0, \p}=\{\frac{k}{m_\p};\ k\in\Int\}$ and $D_{0, \q}=\{\frac{k}{m_\q};\ k\in\Int\}$
for an integer $m_\p$ divides $\p$ and an integer $m_\q$ divides $\q$.}

{Let ${\mathcal R}\subset \underline{K}(A\otimes Q)$
be a finite subset which generates a subgroup containing $$\frac{1}{m_\p m_\q}((\imath_{\p, \infty})_*(G'_{0,0, \p}\cup G'_{1, 0, \p})\cup(\imath_{\q, \infty})_*(G'_{0,0, \q}\cup G'_{1, 0, \q}))$$ in $\underline{K}(A\otimes Q)$, where $\imath_{\fr, \infty}$ is the canonical embedding $A\otimes M_\fr\to A\otimes Q$, $\fr=\p, \q$.
Without loss of generality, one may also assume that $\mathcal R\supseteq \iota_{*1}(\mathcal G)$.}
Let $\mathcal{H}_\fr\subset A\otimes M_\fr$ be a finite subset and
$\delta_3>0$ such that for any homomorphism $h$ from
$A\otimes M_\mathfrak{r}$ to $B\otimes M_\mathfrak{r}$ ($\fr=\p$ or $\fr=\q$)
any unitary $z_j$ {($j=1,2,3,4$),} the map $\mathrm{Bott}(h, z_j)$
and ${\mathrm{Bott}}(h, w_j)$ {are} well defined on the subgroup $[\iota'_\fr](\underline{K}(C'_\fr))$
and
$$
\mathrm{Bott}(h, w_j)=\mathrm{Bott}(h, z_1)+\cdots +\mathrm{Bott}(h, z_j)
$$
on the subgroup generated by $[\iota'_\fr](\underline{K}(C'_\fr)),$
if $\|[h(x), z_j]\|<\delta_3$ for any $x\in {\cal H}_\fr,$ where $w_j=z_1\cdots z_j,$ $j=1,2, 3, 4.$
Without loss of generality, we assume that
$\mathcal H^{0}\otimes {\cal H}^\p\subset {\cal H}_\p$ and
${\cal H}^0\otimes {\cal H}^\q\subset {\mathcal H}_\q.$ Furthermore,
we may also assume that
$$
{\cal H}_\fr={\cal H}_{0,0}\otimes {\cal H}_{0,\fr}
$$
{for} some finite subsets $\mathcal H_{0, 0}$ and $\mathcal H_{0, \fr}$ with ${\cal H}^{0'}\subset  {\cal H}_{0, 0}\subset A,$ ${\cal H}^{\p'}\subset {\cal H}_{0,\p}\subset M_\p$ and ${\cal H}^{\q'}\subset {\cal H}_{0,\q}$. In addition, we may also assume that $\dt_3<\dt_2/2.$

Furthermore, one may assume that $\delta_3$ is sufficiently small such that, for any unitaries $z_1, z_2, z_3$ in a C*-algebra with tracial states, $\tau(\frac{1}{2\pi i}\log(z_iz_j^*))$ ($i, j=1,2,3$) is well defined and $$\tau(\frac{1}{2\pi i}\log(z_1z_2^*))=\tau(\frac{1}{2\pi i}\log(z_1z_3^*))+ \tau(\frac{1}{2\pi i}\log(z_3z_2^*)) $$ for any tracial state $\tau$, whenever $\|z_1-z_3\|<\delta_3$ and $\|z_2-z_3\|<\delta_3$.

To simply notation, we also assume that,
for any unitary $z_j$, ($j=1,2,3,4$) the map $\mathrm{Bott}(h, z_j)$
and ${\mathrm{Bott}}(h, w_j)$ {are} well defined on the subgroup generated by $\mathcal R$ and
$$
\mathrm{Bott}(h, w_j)=\mathrm{Bott}(h, z_1)+\cdots + \mathrm{Bott}(h, z_j)
$$
on the subgroup generated by ${\cal R},$
if $\|[h(x), z_j]\|<\delta_3$ for any $x\in\mathcal H'',$ where $w_j=z_1\cdots z_j,$ $j=1,2,...,4,$
and assume that
$$
{\cal H}''={\cal H}_{0,0}\otimes {\cal H}_{0,\p}\otimes {\cal H}_{0,\q}.
$$





  Let ${\cal R}^i={\cal R}\cap K_i(A\otimes Q).$
There is a finitely generated subgroup $G_{i,0}$ of $K_i(A)$ and there is a
  finite subset $D'_0\subset \Q$ such that
  $$
  G_{i,\infty}:={G}(\{gr: g\in (\imath_{\infty})_{*i}(G_{i,0}))\andeqn r\in D_0'\})
  $$
  contains the subgroup generated by ${\cal R}^i,$ $i=0,1.$ Without loss of generality, we may assume that $G_{i,\infty}$ is the subgroup generated by
  ${\cal R}^i.$
  Note that we may also assume that
  $G_{i,0}\supset G({\cal P})_{i,0}$ and $1\in  D_0'\supset D_0.$
  Moreover, we may assume that, if $r=k/m,$ where $m,k$ are relatively prime non-zero integers,
  and $r\in D_0',$ then $1/m\in D_0'.$
  We may also assume that $G_{i,0,\fr}\subseteq G_{i,0}$ for $\fr=\p, \q$ and $i=0,1.$
  Let
  ${\cal R}^{i'}\subset K_i(A)$ be a finite subset which generates $G_{i,0},$ $i=0,1.$ Choose a finite subset $\mathcal U\subset U_n(A)$ for some $n$
  such that for any element of ${\mathcal {R}^1}'$, there is a representative in $\mathcal U$. Let $S$ be a finite subset of $A$ such that if $(z_{i, j})\in \mathcal U$, then $z_{i, j}\in S$.

Denote by $\delta_4$ and
$\mathcal{Q}_\fr\subset \Kone(A\otimes M_\fr)\cong\Kone(A)\otimes \Q_\fr$ the constant
and finite subset of Lemma \ref{lem2} corresponding to
$\mathcal E_\fr\cup {\mathcal{H}}_\fr\otimes 1 \cup\imath_\fr(S)$ (in the place of $\F$), $\imath_\fr(\mathcal U)$ (in the place of $\mathcal P$)
and
${\frac{1}{n^2}}\min\{\delta'_2/8, \delta_3/4\}$ (in the place of $\eps$) ($\fr=\p$ or $\fr=\q$).
 We may assume that ${\cal Q}_\fr=\{x\otimes r: x\in {\cal Q}'\andeqn r\in D_\fr''\},$
  where ${\cal Q}'\subset K_1(A)$ is a finite subset and
  $D_\fr''\subset \Q_\fr$ is also a finite subset.
  Let $K=\max\{|r|: r\in D''_\p\cup D''_\q\}.$
Since
$
[\phi]=[\psi]\,\,\,\,{\textrm{in}}\,\,\, KL(A, B),
$
$
\phi_{\sharp}=\psi_{\sharp}\andeqn \phi^{\ddag}=\psi^{\ddag},
$
by Lemma \ref{lem-dense},  $\overline{R}_{\phi, \psi}(\Kone(A))\subseteq \overline{\rho_B(K_0(B))}\subset \aff(\mathrm{T}(B))$. Therefore, there is a map $\eta: G(\mathcal Q') \to \overline{\rho_B(K_0(B))}\subset \mathrm{Aff}(\mathrm{T}(B))$ such that
\beq\label{N2N-etaadd}
(\eta- {\overline{R}}_{\phi, \psi}) ([z])\in \rho_B(\Kzero(B))\andeqn \|\eta(z)\|<{\delta_4\over{1+K}}\tforal z\in \mathcal Q'
\eneq

Consider the map $\phi_{\fr}=\phi\otimes\mathrm{id}_{M_\mathfrak{r}}$ and $\psi_{\fr}=\psi\otimes\mathrm{id}_{M_\mathfrak{r}}$ ($\fr=\p$ or $\fr=\q$).
Since $\eta$ vanishes on the torsion part of $G( {\cal Q}'),$
there is {a \hm\,}$\eta_\fr: G((\imath_\fr)_{*1}({\cal Q}'))\to
\overline{\rho_{B\otimes M_\fr}(K_0(B\otimes M_\fr))}\subset \aff(\tr(B\otimes M_\fr))$ such that
\beq\label{N2N-eta}
\eta_\fr\circ (\imath_\fr)_{*1}=\eta.
\eneq
Since
${\overline{\rho_{B\otimes M_\fr}(K_0(B\otimes M_\fr))}}=\overline{\R\rho_B(K_0(B))}$
is divisible, one can
extend $\eta_\fr$ so it defines on  $\Kone(A)\otimes \mathbb Q_\mathfrak{r}.$ We will  continue to use $\eta_\fr$ for the extension.
{ It follows from (\ref{N2N-etaadd})  that $\eta_\fr(z)-{\overline{R}}_{\phi_{\fr}, \psi_{\fr}}(z)\in \rho_{B\otimes M_\mathfrak{r}}(\Kzero(B\otimes M_\mathfrak{r}))$} and $ \|\eta_\fr(z)\|<\delta_4$ for all $z\in \mathcal Q_\fr$. By Lemma \ref{lem2}, there exists a unitary $u_\p\in B\otimes M_\p$ such that
\begin{equation}\label{end-p}
\|u^*_\p(\phi\otimes\mathrm{id}_{M_\p})(c)u_\p- (\psi\otimes\mathrm{id}_{M_\p})(c)\|<{\frac{1}{n^2}}\min\{\delta'_2/8, \delta_3/4\},\quad\forall c\in\mathcal E_\p\cup \mathcal {\cal H}_\p \cup\imath_\p(S).
\end{equation}
{Note that $$\|u^*_\p(\phi\otimes\mathrm{id}_{M_\p})(z)u_\p- (\psi\otimes\mathrm{id}_{M_\p})(z)\|<\delta_3\quad\textrm{for any}\ z\in\mathcal U.$$ Therefore}
 $\tau(\frac{1}{2\pi i}\log(u^*_\p(\phi\otimes\mathrm{id}_{\p})(z)
u_\p({\psi\otimes\mathrm{id}_{\p})(z^*)}))=\eta_\p({[z]})(\tau)$ for all $z\in
{\imath_\p({\cal U})}$,
{where we identify $\phi$ and $\psi$ with
$\phi\otimes {\mathrm{id}}_{M_n}$ and
$\psi\otimes {\mathrm{id}}_{M_n},$ and
$u_\p$ with $u_\p\otimes 1_{M_n}$, respectively.}

The same argument shows that there is a unitary $u_\q\in B\otimes M_\q$ such that
\begin{equation}\label{end-q'}
\|u^*_\q(\phi\otimes\mathrm{id}_{M_\q})(c)u_\q- (\psi\otimes\mathrm{id}_{M_\q})(c)\|<{\frac{1}{n^2}}\min\{\delta'_2/8, \delta_3/4\}, \quad\forall c\in\mathcal E_\q\cup \mathcal{H}_\q{\cup\imath_\p(S)},
\end{equation}
and $\tau(\frac{1}{2\pi i}\log(u^*_\q(\phi\otimes\mathrm{id}_{\q})(z)u_
\q({\psi\otimes\mathrm{id}_{\q})(z^*})))=\eta_\q({[z]})(\tau)$ for all $z\in\mathcal \imath_\q({\mathcal U})$,
{where we identify $\phi$ and $\psi$ with
$\phi\otimes {\mathrm{id}}_{M_n}$ and
$\psi\otimes {\mathrm{id}}_{M_n},$ and
$u_\q$ with
$u_\q\otimes 1_{M_n},$ respectively.}
We will also identify $u_\p$ with $u_\p\otimes 1_{M_\q}$ and
$u_\q$ with $u_\q\otimes 1_{M_\p}$ respectively. Then $u_\p u_\q^*\in A\otimes Q$ and
one estimates that for any
$c\in  {\cal H}_{00}\otimes {\cal H}_{0,\p}\otimes {\cal H}_\q,$
\begin{eqnarray}
\|u_\q u_\p^*(\phi\otimes 1_Q(c)) u_\p u_\q^*-(\phi\otimes 1_Q)(c)\|<\delta_3,
\end{eqnarray}
and hence $\mathrm{Bott}(\phi\otimes \mathrm{id}_Q, u_\p u^*_\q)(z)$ is well defined on the subgroup generated by $\mathcal R$. Moreover, for any $z\in{\mathcal U}{,}$ by the Exel formula (see {\ref{Exel}}) and applying
(\ref{N2N-eta}),
\begin{eqnarray}
&&\hspace{-0.5in}
\tau(\mathrm{bott}_1(\phi\otimes \mathrm{id}_Q, u_\p u^*_\q)((\imath_{\infty})_{*1}({[z]})))\\
&=&
\tau(\mathrm{bott}_1(\phi\otimes \mathrm{id}_Q, u_\p u^*_\q)({\imath_{\infty}}(z)))\\
&=&\tau(\frac{1}{2\pi i}\log(u_\p u^*_\q(\phi\otimes \mathrm{id}_Q)({\imath_{\infty}}(z)))u_\q u^*_\p(\phi\otimes \mathrm{id}_Q)({\imath_{\infty}}(z))^*)\\
&=&\tau(\frac{1}{2\pi i}\log(u^*_{{\q}}(\phi\otimes\mathrm{id}_{Q})({\imath_{\infty}}(z))))
u_{{\q}}({\psi\otimes\mathrm{id}_{Q})({\imath_{\infty}}(z^*}))))\\
&&-\tau(\frac{1}{2\pi i}\log(u^*_\p(\phi\otimes\mathrm{id}_{Q})({\imath_\infty (z)})
u_\p({\psi\otimes\mathrm{id}_{Q})({\imath_{\infty}}(z^*}))))\\
&=&\eta_{{\q}}((\imath_{{\q}})_{*1}({[z]}))(\tau)-\eta_{{\p}}((\imath_{{\p}})_{*1}({[z]}))(\tau)\\
&=&\eta({[z]})(\tau)-\eta({[z]})(\tau)=0\tforal \tau\in T(B),
\end{eqnarray}
{where we identify $\phi$ and $\psi$ with
$\phi\otimes {\mathrm{id}}_{M_n}$ and
$\psi\otimes {\mathrm{id}}_{M_n},$ and
$u_\p$ and $u_\q$ with $u_\p\otimes 1_{M_n}$ and $u_\q$ with
$u_\q\otimes 1_{M_n},$ respectively.}

Now suppose that $g\in G_{1,\infty}.$ Then $g=(k/m)(\imath_{\infty})_{*1}({[z]})$ for some $z\in {\mathcal U},$ where
 $k,m$ are non-zero integers.
It follows that
\beq
\tau(\mathrm{bott}_1(\phi\otimes \mathrm{id}_Q, u_\p u^*_\q)(mg))=
k\tau(\mathrm{bott}_1(\phi\otimes \mathrm{id}_Q, u_\p u^*_\q)(({[z]}))=0
\eneq
for all $\tau\in \tr(B).$ Since $\aff(\tr(B))$ is torsion free, it follows
that
\beq
\tau(\mathrm{bott}_1(\phi\otimes \mathrm{id}_Q, u_\p u^*_\q)(g)=0
\eneq
for all $g\in G_{1,\infty}$ and $\tau\in \tr(B).$
Therefore, the image of ${\cal R}^1$ under $\mathrm{bott}_1(\phi\otimes\mathrm{id}_Q, u_\p u^*_\q)$ is in
$\ker\rho_{B\otimes Q}$.
{One may write}
$$G_{1, 0}=\Int^r\oplus\Int/p_1\Int\oplus\cdots\oplus\Int/p_s\Int,$$ where $r$ is a non-negative integer and $p_1, ..., p_s $ are powers of primes numbers. Since $\p$ and $\q$ are relatively prime, one then has the decomposition $$G_{1, 0}=\Int^r\oplus \mathrm{Tor}_\p(G_{1, 0})\oplus \mathrm{Tor}_\q(G_{1, 0}) \subseteq \Kone(A),$$ where $\mathrm{Tor}_\p(G_{1, 0})$ consists of the torsion-elements with their orders divide $\p$ and $\mathrm{Tor}_\q(G_{1, 0})$ consists of the torsion-elements with their orders divide $\q$. Fix this decomposition.

Note that the restriction of $(\imath_\p)_{*1}$ to $\Int^r\oplus \mathrm{Tor}_\q(G_{1, 0})$ is injective and the restriction to $\mathrm{Tor}_\p(G_{1, 0})$ is zero, and the restriction of $(\imath_\q)_{*1}$ to $\Int^r\oplus \mathrm{Tor}_\p(G_{1, 0})$ is injective and the restriction to $\mathrm{Tor}_\q(G_{1, 0})$ is zero.

Moreover, using the assumption that $\p$ and $\q$ are relatively prime again, for any element $k\in (\imath_\q)_{*1}(\Int^r\oplus \mathrm{Tor}_\p(G_{1, 0}))$ and any nonzero integer $q$ which divides $\q$, the element $k/q$ is well defined in $\Kone(A\otimes M_\q)$; that is, there is a unique element $s\in\Kone(A\otimes M_\q)$ such that $qs=k$.

Denote by $e_1, ..., e_r$ the standard generators of $\Int^r$. It is also clear that $$(\imath_\infty)_{*1}(\mathrm{Tor}_\p(G_{1, 0}))=(\imath_\infty)_{*1}(\mathrm{Tor}_\q(G_{1, 0}))=0.$$

Recall that $D_{0, \p}=\{k/m_\p; \ k\in \Int\}\subset \Ratn_\p$ and $D_{0, \q}=\{k/m_\q; \ k\in \Int\}\subset \Ratn_\q$ for an integer $m_\p$ dividing  $\p$ and an integer $m_\q$ dividing $\q$. Put $m_\infty=m_\p m_\q$.

{Consider $\frac{1}{m_\infty}\Int^r\in\Kone(A\otimes Q)$, and} for each $e_i$, $1\leq i\leq r$, consider $$\frac{1}{m_\infty}\mathrm{bott}_1(\phi\otimes \mathrm{id}_Q, u_\p u^*_\q)((\imath_\infty)_{*1}(e_i))\in\ker\rho_{B\otimes Q}.$$
{Since $\ker\rho_{B\otimes Q}\cong(\ker{\rho}_B)\otimes\Ratn$, $\ker\rho_{B\otimes M_\p}\cong(\ker{\rho}_B)\otimes\Ratn_\p$, and $\ker\rho_{B\otimes M_\q}\cong(\ker{\rho}_B)\otimes\Ratn_\q$,} using the same arguments as that of Theorem \ref{AHtoC}, there are $g_{i, \p}\in\ker \rho_{B\otimes M_\p}$ and $g_{i, \q}\in\ker \rho_{B\otimes M_\q}$ such that   $$
\mathrm{bott}_1(\phi\otimes \mathrm{id}_Q, u_\p u^*_\q)(\frac{1}{m_\infty}((\imath_\infty)_{*1}(e_i)))={(j_\p)_{*0}(g_{i, \p})+(j_\q)_{*0}(g_{i, \q})},
$$ where $g_{i, \p}$ and $g_{i, \q}$ are identified as their images in $\Kzero(A\otimes Q)$.

Note that the subgroup $(\imath_\p)_{*1}(G_{1, 0})$ in $\Kzero(A\otimes M_\p)$ is isomorphic to $\Int^r\oplus\mathrm{Tor}_\q$ and  $\frac{1}{m_\p}(\Int^r\oplus\mathrm{Tor}_\q)$ is well defined in $\Kzero(A\otimes M_\p)$, and the subgroup $(\imath_\q)_{*1}(G_{1, 0})$ in $\Kzero(B\otimes M_\q)$ is isomorphic to $\Int^r\oplus\mathrm{Tor}_\p$ and $\frac{1}{m_\q}(\Int^r\oplus\mathrm{Tor}_\p)$ is well defined in $\Kzero(A\otimes M_\q)$.
One then defines the maps $\theta_\p: \frac{1}{m_\p}(\imath_\p)_{*1}(G_{1, 0})\to\ker\rho_{B\otimes M_\p}$ and $\theta_\q: \frac{1}{m_\q}(\imath_\q)_{*1}(G_{1, 0})\to\ker\rho_{B\otimes M_\q}$ by
$$\theta_\p(\frac{1}{m_\p}(\iota_\p)_{*1}(e_i))=m_\q g_{i, \p}\quad\mathrm{and}\quad \theta_\q(\frac{1}{m_\q}(\iota_\q)_{*1}(e_i))=m_\p g_{i, \q}$$
for $1\leq i\leq r$ and
$$\theta_\p|_{\mathrm{Tor}( (\imath_\p)_{*1}(G_{1, 0}))}=0 \quad\mathrm{and}\quad \theta_\q|_{\mathrm{Tor}( (\imath_\q)_{*1}(G_{1, 0}))}=0.$$
Then, for each $e_i$, one has
\begin{eqnarray*}
&&{(j_\p)_{*0}\circ\theta_\p\circ (\imath_{\p})_{*1}(e_i)+(j_\q)_{*0}\circ \theta_\q\circ (\imath_{q})_{*1}(e_i)}\\
&=&{m_\p(\frac{1}{m_\p}(j_\p)_{*0}\circ \theta_\p\circ (\imath_{\p})_{*1}(e_i))+m_\q(\frac{1}{m_\q}(j_\q))_{*0}\circ \theta_\q\circ (\imath_{q})_{*1}(e_i))}\\
&=& {m_\p m_\q ((j_\p)_{*0}(g_{i, p}) + (j_\q)_{*0}(g_{i, q}))}\\
&=& m_\infty \mathrm{bott}_1(\phi\otimes\mathrm{id}_Q, u_\p u^*_\q)\circ (\imath_{\infty})_{*1}(e_i/m_\infty)\\
&=& \mathrm{bott}_1(\phi\otimes\mathrm{id}_Q, u_\p u^*_\q)\circ (\imath_{\infty})_{*1}(e_i).
\end{eqnarray*}
Since the restriction of $\theta_\p\circ(\imath_\p)_{*1}$, $\theta_\q\circ(\imath_\q)_{*1}$ and $\mathrm{bott}_1(\phi\otimes\mathrm{id}_Q, u_\p u^*_\q)\circ (\imath_{\infty})_{*1}$ to the torsion part of $G_{1, 0}$ is zero, one has
$$\mathrm{bott}_1(\phi\otimes\mathrm{id}_Q, u_\p u^*_\q)\circ (\imath_{\infty})_{*1}={(j_\p)_{*0}\circ \theta_\p\circ (\imath_{\p})_{*1}+(j_\q)_{*0}\circ \theta_\q\circ (\imath_{q})_{*1}}.$$

The same argument shows that there also exist maps
$\alpha_\p: \frac{1}{m_\p}((\imath_\p)_{*0}(G_{0,0}))\to\Kone(B\otimes M_\p)$ and $\alpha_\q: \frac{1}{m_\q}((\imath_\q)_{*0}(G_{0,0}))\to\Kone(B\otimes M_\q)$ such that
$$\mathrm{bott}_0(\phi\otimes\mathrm{id}_Q, u_\p u^*_\q)\circ(\imath_{\infty})_{*0}=(j_\p)_{*1}\circ \alpha_\p\circ (\imath_{\p})_{*0}+(j_\q)_{*1}\circ \alpha_\q\circ (\imath_{\q})_{*0}$$
on $G_{0,0}.$

Note that $G_{i, 0, \fr}\subseteq G_{i, 0}$, $i=0, 1$, $\fr=\p, \q$. In particular, one has that $(\imath_\fr)_{*i}(G_{i, 0, \fr})\subseteq (\imath_\fr)_{*i}(G_{i, 0})$, and therefore $G'_{1, 0, \p}\subseteq \frac{1}{m_\p} (\imath_\p)_{*0}(G_{1, 0})$ and $G'_{1, 0, \q}\subseteq \frac{1}{m_\q} (\imath_\q)_{*0}(G_{1, 0})$. Then the maps $\theta_\p$ and $\theta_\q$ can be restricted to $G'_{1, 0, \p}$ and $G'_{1, 0, \q}$ respectively. Since the group $G'_{i, 0, \fr}$ contains $(\iota'_\fr)_{*i}(K_i(C_\fr'))$, the maps $\theta_\p$ and $\theta_\q$ can be restricted further to $(\iota'_\p)_{*1}(K_1(C_\p'))$ and $(\iota'_\q)_{*1}(K_1(C_q'))$ respectively.

For the same reason, the maps $\alpha_\p$ and $\alpha_\q$ can be restricted to $(\iota'_\p)_{*0}(K_0(C_\p'))$ and $(\iota'_\q)_{*0}(K_0(C_q'))$ respectively. We keep the same notation for the restrictions of these maps $\alpha_\p$, $\alpha_\q$, $\theta_\p$, and $\theta_\q$.

{By the universal multi-coefficient theorem, there is $\kappa_\p\in\mathrm{Hom}_{\Lambda}(\underline{K}(C'_\p\otimes\mathrm{C}(\mathbb T)), \underline{K}(B\otimes M_\p))$ such that
$$\kappa_{\p}|_{{\boldsymbol{\bt}}(K_1(C'_\p))} =-\theta_\p\circ(\iota'_\p)_{*1}\circ {\boldsymbol{\bt}}^{-1}\andeqn
\kappa_{\p}|_{{\boldsymbol{\bt}}(\Kzero(C'_\p))}=-\alpha_\p\circ(\iota'_\p)_{*0}\circ {\boldsymbol{\bt}}^{-1}.
$$
}
{Similarly, there exists
 $\kappa_\q\in\mathrm{Hom}_{\Lambda}(\underline{K}(C'_\q\otimes\mathrm{C}(\mathbb T))), \underline{K}(B\otimes M_\q))$ such that
$$\kappa_{\q}|_{{\boldsymbol{\bt}}(\Kone(C'_\q))} =-\theta_\q\circ(\iota'_\q)_{*1}\circ {\boldsymbol{\bt}}^{-1}\andeqn
\kappa_{\q}|_{{\boldsymbol{\bt}}(\Kzero(C'_\q))}=-\alpha_\q\circ (\iota'_\q)_{*0} \circ {\boldsymbol{\bt}}^{-1}.$$
}

{Note that since $g_{i,\fr}\in {\mathrm ker}\rho_{A\otimes M_\fr},$ $\kappa_\fr({\boldsymbol{\bt}}(\Kone(C_\fr')))\subseteq \ker \rho_{B\otimes M_\fr}$, $\fr=\p$ or $\fr=\q$.} By Theorem \ref{B1B2-alg}, there {exist} unitaries $w_\p\in B\otimes M_\p$ and $w_\q\in B\otimes M_\q$ such that
$$\|[w_\p, {(\phi\otimes\mathrm{id}_{M_\p})}(x)]\|<\delta'_2/8,\quad  \|[w_\q, {(}\phi\otimes\mathrm{id}_{M_\q}{)}(y)]\|<\delta'_2/8,$$
for any $x\in {\cal H}^{0'}\otimes {\cal H}^{\p'}$ and $y\in {\cal H}^{0'}\otimes {\cal H}^{\q'},$  and
$$\mathrm{Bott}(\phi\otimes\mathrm{id}_{M_\p}, w_\p)\circ[\iota'_\p]=\kappa_p\circ\boldsymbol{\beta}\quad\mathrm{and}\quad \mathrm{Bott}(\phi\otimes\mathrm{id}_{M_\q}, w_\q)\circ[\iota'_\q]=\kappa_q\circ\boldsymbol{\beta}.$$

For $\fr=\p$ or $\fr=\q$ and each $1\leq j\leq k$, define
\beq\nonumber
&&{\zeta_{j, w_\fr u_\fr}=}\\\nonumber
&&\overline{\langle (\mathbf 1_n-(\phi\otimes\mathrm{id}_{M_\fr})(p'_{j, \fr})+ ((\phi\otimes\mathrm{id}_{M_\fr})(p'_{j, \fr}))w_\fr u_\fr)(\mathbf 1_n-(\phi\otimes\mathrm{id}_{M_\fr})(q'_{j, \fr})+((\phi\otimes\mathrm{id}_{M_\fr})(q'_{j, \fr}))u^*_\fr w^*_\fr)\rangle}.
\eneq
It is an element in $U(B\otimes M_\fr)/CU(B\otimes M_\fr)$.

{Define the map $\Gamma_\fr: \Int^{k}\to U(B\otimes M_\p)/CU(B\otimes M_\p)$ by
$$\Gamma_\fr(x'_{j, \fr})= \zeta_{j, w_\fr u_\fr},\quad 1\leq j\leq k.$$
}

{Applying Corollary \ref{BB-exi-mat} to $C_\fr$ (in the place of $C$), $G(x'_{1, \fr}, ..., x'_{k, \fr})$ (in the place of $G$), $B\otimes M_\fr$ (in the place of $A$), and $(\phi\otimes\mathrm{id}_{M_\fr})|_{C_\fr}$ (in the place of $\phi$), there is a unitary} {$c_\fr\in B\otimes M_\fr$} such that
$$\|[c_\fr, {(\phi\otimes\mathrm{id}_{M_\fr})}(x)]\|<\delta'_2/16$$
for any $x\in {\cal H}^{0'}\otimes {\cal H}^{\fr'}$,
$$\mathrm{Bott}(\phi\otimes\mathrm{id}_{M_\fr}, c_\fr)|_{\imath_\fr(\mathcal P')}=0,$$
and
\begin{equation}\label{pre-alg-k1-est-01}
\mathrm{dist}(\zeta_{j, c_\fr^*}, \Gamma_\fr(x_{j, \fr}))\leq \gamma/(32(1+\sum_{i, j} \abs{Mr_{ij}})),\quad 1\leq j\leq k,
\end{equation}
where
$$
\zeta_{j, c^*_\fr}=\overline{\langle (\mathbf 1_n-(\phi\otimes\mathrm{id}_{M_\fr})(p'_{j, \fr})+ ((\phi\otimes\mathrm{id}_{M_\fr})(p'_{j, \fr}))c^*_\fr)(\mathbf 1_n-(\phi\otimes\mathrm{id}_{M_\fr})(q'_{j, \fr})+((\phi\otimes\mathrm{id}_{M_\fr})(q'_{j, \fr}))c_\fr)\rangle}.
$$

Put $v_\fr=c_\fr w_\fr u_\fr$. Then, by \eqref{AHtoC-add4} and \eqref{pre-alg-k1-est-01}, for $1\leq j\leq k$,
\begin{equation}\label{alg-k1-est-01}
\mathrm{dist}({\zeta}_{j, v_\fr}, \overline{(1_{B\otimes M_\fr})_n}) < \mathrm{dist}(\zeta_{j, c_\fr^*}, \zeta_{j, w_\fr u_\fr}) +  \gamma/(32(1+\sum_{i, j} \abs{Mr_{ij}}))<  \gamma/(16(1+\sum_{i, j} \abs{Mr_{ij}})),
\end{equation}
where
$${\zeta}_{j, v_\fr}=\overline{\langle (\mathbf 1_n-(\phi\otimes\mathrm{id}_{M_\fr})(p'_{j, \fr})+ ((\phi\otimes\mathrm{id}_{M_\fr})(p'_{j, \fr}))v_\fr)(\mathbf 1_n-(\phi\otimes\mathrm{id}_{M_\fr})(q'_{j, \fr})+((\phi\otimes\mathrm{id}_{M_\fr})(q'_{j, \fr})) v^*_\fr)\rangle}.$$

 Recall that $[x_j']=[p_j']-[q_j']$. Define
$$
{\zeta}_{x'_j, v_\fr}=\overline{\langle (\mathbf 1_n-\phi(p'_j)\otimes 1_{M_\fr}+ (\phi(p'_j)\otimes 1_{M_\fr})v_\fr)(\mathbf 1_n-\phi(q'_j)\otimes 1_{M_\fr}+(\phi(q_j')\otimes 1_{M_\fr})v^*_\fr)\rangle}.$$

 By \eqref{approx-est1} and \eqref{approx-est2}, one has
 $$\mathrm{dist}(\zeta_{x_j', v_\fr}, \zeta_{j, v_r})< \gamma/(16(1+\sum_{i, j'} \abs{Mr_{ij'}})),$$
 and hence by \eqref{alg-k1-est-01},
 $$\mathrm{dist}(\zeta_{x_j', v_\fr}, \overline{(1_{B\otimes M_\fr})_n})< \gamma/(8(1+\sum_{i, j'} \abs{Mr_{ij'}})).$$
 Regard $\zeta_{x_j', v_\fr}$ as its image in $B\otimes Q$, one has
$$\mathrm{dist}(\zeta_{x_j', v_\fr}, \overline{(1_{B\otimes Q})_n})< \gamma/(8(1+\sum_{i, j'} \abs{Mr_{ij'}})),$$
and hence for any $1\leq i\leq m$,
$$\mathrm{dist}(\prod_{j=1}^k (\zeta_{x_j', v_\fr})^{Mr_{ij}}, \overline{(1_{B\otimes Q})_n})< \gamma/8.$$

By \eqref{AHtoC-add3}, one has
$$\mathrm{dist}(\overline{\langle (1-(\phi\otimes\mathrm{id}_{Q})(p_{i})+ (\phi\otimes\mathrm{id}_{Q})(p_{i}) v_\fr ) (1-(\phi\otimes\mathrm{id}_{Q})(q_{i})+ (\phi\otimes\mathrm{id}_{Q})(q_{i}) v^*_\fr ) \rangle^M}, \overline{(1_{B\otimes Q})_n})<\gamma/4,$$
and then, by Theorem 6.10 (and Theorem 6.11) of \cite{Lnuni1},
\begin{eqnarray*}
&&\mathrm{dist}(\overline{\langle (1-(\phi\otimes\mathrm{id}_{Q})(p_{i})+ (\phi\otimes\mathrm{id}_{Q})(p_{i}) v_\fr ) (1-(\phi\otimes\mathrm{id}_{Q})(q_{i})+ (\phi\otimes\mathrm{id}_{Q})(q_{i}) v^*_\fr ) \rangle}, \overline{(1_{B\otimes Q})_n})\\
&<&\gamma/(4M)<\gamma/4.
\end{eqnarray*}

In particular,
\begin{eqnarray*}
&&\mathrm{dist}(\overline{ \langle (1-(\phi\otimes\mathrm{id}_{Q})(p_{i})+ (\phi\otimes\mathrm{id}_{Q})(p_{i}) v_\q v_p^* ) (1-(\phi\otimes\mathrm{id}_{Q})(q_{i})+ (\phi\otimes\mathrm{id}_{Q})(q_{i}) v_pv_q^* ) \rangle }, \overline{(1_{B\otimes Q})_n})\\
&\leq&\mathrm{dist}(\overline{ \langle (1-(\phi\otimes\mathrm{id}_{Q})(p_{i})+ (\phi\otimes\mathrm{id}_{Q})(p_{i}) v_\q ) (1-(\phi\otimes\mathrm{id}_{Q})(q_{i})+ (\phi\otimes\mathrm{id}_{Q})(q_{i}) v^*_{\q} ) \rangle }, \overline{(1_{B\otimes Q})_n})\\
&&+\mathrm{dist}(\overline{ \langle (1-(\phi\otimes\mathrm{id}_{Q})(p_{i})+ (\phi\otimes\mathrm{id}_{Q})(p_{i}) v_\p  ) (1-(\phi\otimes\mathrm{id}_{Q})(q_{i})+ (\phi\otimes\mathrm{id}_{Q})(q_{i}) v_\p^* ) \rangle }, \overline{(1_{B\otimes Q})_n})\\
& <&\gamma/2.
\end{eqnarray*}

That is
\begin{equation}\label{bu-small-mid}
\mathrm{dist}({\zeta}_{i, v_\q v_\p^*}, \overline{1_n})<\gamma/2,
\end{equation}
where
${\zeta}_{i, v_\q v_\p^*}=\overline{ \langle (1-(\phi\otimes\mathrm{id}_{Q})(p_{i})+ (\phi\otimes\mathrm{id}_{Q})(p_{i}) v_\q v_p^* ) (1-(\phi\otimes\mathrm{id}_{Q})(q_{i})+ (\phi\otimes\mathrm{id}_{Q})(q_{i}) v_pv_q^* ) \rangle }.$

Moreover, one also has
$$\norm{\psi\otimes \mathrm{id}_Q(x)-v^*_\p(\phi\otimes {\mathrm{id}}_Q(x))v_\p}<\delta'_2/4,\quad\forall x\in {\cal H}^{0'}\otimes {\cal H}^{p'} \otimes {\cal H}^{\q'}{\andeqn}$$
$$\norm{\psi\otimes \mathrm{id}_Q(x)-v_\q^*(\phi\otimes {\mathrm{id}}_Q)(x)v_\q}<\delta'_2/4,\quad\forall x\in {\cal H}^{0'}\otimes {\cal H}^{p'}\otimes {\cal H}^{q'}.$$
Hence
$$\norm{[v_\p v^*_\q, \phi(x)\otimes 1_Q]}<\delta'_2/2<\delta_2,\quad \forall x\in {\cal H}'.$$
{Thus} $\mathrm{Bott}(\phi\otimes {\mathrm{id}}_Q, v_\p v^*_\q)$ is well defined on the subgroup generated by $\mathcal P$. Moreover, a direct calculation shows that
\begin{eqnarray*}
&&\mathrm{bott}_1(\phi\otimes {\mathrm{id}}_Q, v_\p v^*_\q)\circ (\imath_{\infty})_{*1}(z)\\
&=&\mathrm{bott}_1(\phi\otimes {\mathrm{id}}_Q, c_\p)\circ (\imath_\infty)_{*1}(z)+\mathrm{bott}_1(\phi\otimes {\mathrm{id}}_Q, w_\p)\circ (\imath_\infty)_{*1}(z)\\
&&+\mathrm{bott}(\phi\otimes \mathrm{id}_Q, u_\p u^*_\q))\circ (\imath_\infty)_{*1}(z) +\mathrm{bott}_1(\phi\otimes  \mathrm{id}_Q, w^*_\q)\circ (\imath_\infty)_{*1}(z)\\
&&+\mathrm{bott}_1(\phi\otimes {\mathrm{id}}_Q, c^*_\q)\circ (\imath_\infty)_{*1}(z)\\
&=&{(j_\p)_{*0}\circ \mathrm{bott}_1(\phi\otimes {\mathrm{id}}_{M_\p}, c_\p)\circ (\imath_\p)_{*1}(z)+ (j_\p)_{*0}\circ \mathrm{bott}_1(\phi\otimes {\mathrm{id}}_{M_\p}, w_\p)\circ (\imath_\p)_{*1}(z)}\\
&&+\mathrm{bott}(\phi\otimes {\mathrm{id}}_{Q}, u_\p u^*_\q)\circ (\imath_\infty)_{*1}(z)+{(j_\q)_{*0}\circ} \mathrm{bott}_1(\phi\otimes {\mathrm{id}}_{M_\q}, w^*_\q)\circ (\imath_\q)_{*1}(z)\\
 &&+{(j_q)_{*0}\circ }\mathrm{bott}_1(\phi\otimes {\mathrm{id}}_{M_\q}, c^*_\q)\circ (\imath_\q)_{*1}(z)\\
 &=& {(j_\p)_{*0}\circ }\mathrm{bott}_1(\phi\otimes {\mathrm{id}}_{M_\p}, w_\p)\circ (\imath_\p)_{*1}(z)+\mathrm{bott}(\phi\otimes {\mathrm{id}}_{Q}, u_\p u^*_\q)\circ (\imath_\infty)_{*1}(z)\\
 &&+{(j_\q)_{*0}\circ }\mathrm{bott}_1(\phi\otimes {\mathrm{id}}_{M_\q}, w^*_\q)\circ (\imath_\q)_{*1}(z)\\
&=&-(j_\p)_{*0}\circ \theta_\p\circ (\imath_\p)_{*1}(z)+((j_\p)_{*0}\circ \theta_\p\circ (\imath_p)_{*1}+(j_\q)_{*0}\circ \theta_\q\circ (\imath_q)_{*1})-(j_\q)_{*0}\circ \theta_\q\circ (\imath_q)_{*1}(z)\\
&=&0\,\,\,\,\,\,\tforal z\in G({\cal P})_{1,0}.
\end{eqnarray*}
The same argument shows that
$\mathrm{bott}_0(\phi\otimes \mathrm{id}_Q, v_\p v_\q^*)=0$
on $G({\cal P})_{0, 0}.$
Now, for any $g\in G({\cal P})_{1, \infty,0},$ there is $z\in G({\cal P})_{1,0}$ and integers $k,m$ such that $(k/m)z=g.$
From the above,
\beq
\mathrm{bott}_1(\phi\otimes {\mathrm{id}}_Q,v_\p v^*_\q)(mg)
&=&k\mathrm{bott}_1(\phi\otimes {\mathrm{id}}_Q,v_\p v^*_\q)(z)=0.
\eneq
Since $K_0(B\otimes Q)$ is torsion free, it follows that
$$
\mathrm{bott}_1(\phi\otimes {\mathrm{id}}_Q,v_\p v^*_\q)(g)=0
$$
for all $g\in G({\cal P})_{1,\infty,0}.$ So it vanishes on ${\cal P}\cap K_1(A\otimes Q).$
Similarly,
$$
\mathrm{bott}_0(\phi\otimes {\mathrm{id}}_Q,v_\p v^*_\q){|_{{\cal P}\cap K_0(A\otimes Q)}}=0
$$
on ${\cal P}\cap K_0(A\otimes Q).$

Since $K_i(B\otimes Q, \Z/m\Z)=\{0\}$ for all $m\ge 2,$ we conclude that
$$\mathrm{Bott}(\phi\otimes {\mathrm{id}}_Q, v_\p v^*_\q){|_{\cal P}}=0$$
on the subgroup generated by $\mathcal P$.

Since
$
[\phi]=[\psi]\,\,\,\,{\textrm{in}}\,\,\, KL(A, B),\,\,\,
{\phi_{\sharp}=\psi_{\sharp}\andeqn \phi^{\ddag}=\psi^{\ddag},}
$
one has that
\beq
[\phi\otimes \mathrm{id}_{Q}]=[\psi\otimes \mathrm{id}_{Q}]\,\,\,\,{\text{in}}\,\,\, KL(A\otimes Q, B\otimes Q),\\
(\phi\otimes \mathrm{id}_{Q})_{\sharp}=(\psi\otimes \mathrm{id}_{Q})_{\sharp}\andeqn (\phi\otimes \mathrm{id}_{Q})^{\ddag}=(\psi\otimes \mathrm{id}_{Q})^{\ddag}.
\eneq

Therefore,  {by {5.10 of \cite{Lin-AU11}}, $\phi\otimes \mathrm{id}_Q$ and
$\psi\otimes\mathrm{id}_Q$ are approximately unitarily equivalent.
Thus there exists a unitary $u\in B\otimes Q$ such that
\beq\label{L3-02}
\|u^*(\phi\otimes\mathrm{id}_Q)(c)u-(\psi\otimes\mathrm{id}_Q)(c)\|
<\delta_2'/8 \tforal c\in \mathcal E\cup \mathcal H'.
\eneq
}
{It follows that
$$
\norm{uv_\p^*(\phi(c)\otimes 1_Q)v_\p u^*-\psi(c)\otimes 1_Q}<\delta_2'/2+\delta_2'/8\quad\forall c\in\mathcal G'.
$$
}
By the choice of $\dt_2'$ and ${\cal H}',$
 $\mathrm{Bott}(\phi\otimes \mathrm{id}_Q, v_\p u^*)$ is
 well defined on {$[\iota](\underline{K}(C'))$},
 and
$$|\tau(\mathrm{bott}_1(\phi\otimes \mathrm{id}_Q, v_\p u^*)(z))|<\delta_2/2,\quad \forall \tau\in\mathrm{T}(B), \forall z\in{\mathcal G}
.$$
By Theorem \ref{B1B2-alg}, there exists a unitary $y_\p\in B\otimes Q$ such that
$$\|[y_\p, (\phi\otimes\mathrm{id}_Q)(h)]\|<\delta/2,\quad\forall h\in \mathcal H,$$
and $\mathrm{Bott}(\phi\otimes\mathrm{id}_{Q}, y_\p)=\mathrm{Bott}(\phi\otimes\mathrm{id}_Q, v_\p u^*)$ on the subgroup generated by $\mathcal P$.

For each $1\leq i\leq m$, define
\beq\nonumber
&&{\zeta}_{i, y_\p u v^*_\p}\\\nonumber
&&=\overline{\langle (\mathbf 1_n-(\phi\otimes {\mathrm{id}}_Q)(p_i)+((\phi\otimes {\mathrm{id}}_Q)(p_i)) y_\p u v^*_\p)(\mathbf 1_n-(\phi\otimes {\mathrm{id}}_Q)(q_i)+((\phi\otimes {\mathrm{id}}_Q)(q_i))v_\p u^* y^*_\p) \rangle},
\eneq
and define the map $\Gamma: \Int^m\to U(B\otimes Q)/CU(B\otimes Q)$ by $\Gamma(x_i)={\zeta}_{i, y_\p uv^*_\q}$.

Applying Corollary \ref{BB-exi-mat} to $C$ and $G(\mathcal Q)$, there is a unitary $c\in B\otimes Q$ such that $$\norm{[c, (\phi\otimes \mathrm{id}_Q)(h)]}<\delta/4,\quad\forall h\in\mathcal H$$
$$\mathrm{Bott}(\phi\otimes \mathrm{id}_Q, c)|_{\mathcal P}=0$$ and for any $1\leq i\leq k$,
$$\mathrm{dist} (\zeta'_{i, c^*}, \Gamma(x_i))\leq\gamma/2 ,$$
where
$$\zeta'_{i, c^*}=\overline{\langle (\mathbf 1_n-(\phi\otimes {\mathrm{id}}_Q)(p_i)+(\phi\otimes {\mathrm{id}}_Q)(p_i)c^*)(\mathbf 1_n-(\phi\otimes {\mathrm{id}}_Q)(q_i)+(\phi\otimes {\mathrm{id}}_Q)(q_i)c) \rangle}.$$

Consider the unitary $v=cy_\p u$, one has that
$$\|[v, (\phi\otimes\mathrm{id}_Q)(h)]\|<\delta, {\rforal} h\in \mathcal H {\andeqn}
\mathrm{Bott}(\phi\otimes\mathrm{id}_{Q}, vv^*_\p)=0$$
on the subgroup generated by $\mathcal P$, and for any $1\leq i\leq m$,

\begin{equation}\label{bu-small-end}
\mathrm{dist}(\zeta_{i, vv_\p^*}, \overline{1_n})<\gamma/2,
\end{equation}
where
$${\zeta}_{i, vv_\p^*}=\overline{\langle(\mathbf 1_n-(\phi\otimes {\mathrm{id}}_Q)(p_i)+((\phi\otimes {\mathrm{id}}_Q)(p_i))vv^*_\p)(\mathbf 1_n-(\phi\otimes {\mathrm{id}}_Q)(q_i)+((\phi\otimes {\mathrm{id}}_Q)(q_i))v_\p v^*)\rangle}.$$

{By the construction of $\Delta$, it is clear that $$\mu_{\tau\circ(\psi\otimes 1)}(O_a)\geq\Delta(a)$$ for all $a$, where $O_a$ is any open ball of $X$ with radius $a$; in particular, it holds for all $a\geq d$.
 Applying Theorem \ref{hp-mat} to $C$ and
 $(\phi\otimes\mathrm{id}_Q)|_C$},
 one obtains a continuous path of unitaries $v(t)$ in $B\otimes Q$
such that
$v(0)=1$ and $v(t_1)=vv^*_\p $, and
\begin{equation}\label{com-p}
\norm{[z_p(t), (\phi\otimes \mathrm{id}_Q)(c)]}<\eps/2\quad\forall x\in\mathcal E,\ {\forall} t\in[0, t_1].
\end{equation}

Note that
\begin{eqnarray}\label{zero-bott}
\mathrm{Bott}(\phi\otimes\mathrm{id}_Q, v_\q v^*)&=&\mathrm{Bott}(\phi\otimes\mathrm{id}_Q, v_\q v_\p^* v_\p v^*)\\
&=&\mathrm{Bott}(\phi\otimes\mathrm{id}_Q, v_\q v_\p^*)+\mathrm{Bott}(\phi\otimes\mathrm{id}_Q, v_\p v^*)\\
&=&0+0=0
\end{eqnarray}
on the subgroup generated by $\mathcal P$, and for any  $1\leq i\leq m$,
\begin{eqnarray}
&&\mathrm{dist}(\zeta_{i, v_\q v^*}, \overline{1})\\
 &\leq& \mathrm{dist}(\zeta_{i, v_\q v_\p^*}, \overline{1})+\mathrm{dist}(\zeta_{i, v_\p v^*}, \overline{1})\\
 &=&\gamma,\quad \textrm{(by \eqref{bu-small-mid} and \eqref{bu-small-end})}
\end{eqnarray}
where
$$\zeta_{i, v_qv^*}=\overline{ \langle (1-(\phi\otimes\mathrm{id}_{Q})(p_{i})+ (\phi\otimes\mathrm{id}_{Q})(p_{i}) v_\q v^* ) (1-(\phi\otimes\mathrm{id}_{Q})(q_{i})+ (\phi\otimes\mathrm{id}_{Q})(q_{i}) vv_\q^* ) \rangle }$$

Since $$\|[vv_\q^*, (\phi\otimes\mathrm{id}_Q)(c)]\|<\delta,\quad\forall c\in \mathcal H,$$
 Theorem \ref{hp-mat} implies that there is a path of unitaries $z_\q(t): [t_{m-1}, 1]\to\mathrm{U}(A\otimes Q)$ such that $z_\q(t_{m-1})=vv^*_\q$, $z_\q(1)=1$ and
\begin{equation}\label{com-q}
\|[ z_\q(t), \phi\otimes\mathrm{id}_Q(c)]\|<\eps/8,\quad\forall t\in[t_{m-1}, 1],\ \forall c\in \mathcal E.
\end{equation}

Consider the unitary
\begin{displaymath}
v(t)=\left\{
\begin{array}{ll}
z_\p(t)v_\p, &\textrm{if}\ 0\leq t\leq t_1,\\
v, &\textrm{if}\ t_1\leq t\leq t_{m-1},\\
z_\q(t)v_\q, &\textrm{if}\ t_{m-1}\leq t\leq t_m.
\end{array}
\right.
\end{displaymath}

Then, for any $t_i$, $0\leq i\leq m$, one has that
\begin{equation}\label{eval}
\|v^*(t_i)(\phi\otimes\mathrm{id}_Q)(c) v(t_i)- (\psi\otimes\mathrm{id}_Q)(c)\|<\ep/2,\quad\forall c\in\mathcal E.
\end{equation}

Then for any $t\in[t_j, t_{j+1}]$ with $1\leq j\leq m-2$, one has
\beq\label{middle}
&&\|v^*(t)(\phi\otimes {\mathrm{id}}(a\otimes b(t)))v(t)-\psi\otimes{\mathrm{id}}(a\otimes b(t))\|\\
&=&\|v^*(\phi(a)\otimes b(t))v-\psi(a)\otimes b(t)\|\\
&<&\|v^*(\phi(a)\otimes b(t_j))v-\psi(a)\otimes b(t_j)\|+\ep/4\\
&<&\ep/4+\ep/4<\ep/2.
\eneq

For any $t\in[0, t_1]$, one has that for any $a\in\mathcal F_1$ and $b\in\mathcal F_2$,
\beq\label{left}
&&\|v^*(t)(\phi\otimes {\mathrm{id}}(a\otimes b(t)))v(t)-\psi\otimes{\mathrm{id}}(a\otimes b(t))\|\\
&=&\|v^*_\p z_\p^*(t)(\phi(a)\otimes b(t))z_\p(t) v_\p-\psi(a)\otimes b(t)\|\\
&<&\|v^*_\p z_\p^*(t)(\phi(a)\otimes b(t_0))z_\p(t) v_\p-\psi(a)\otimes b(t_0)\|+\ep/2\\
&<&\|v^*_\p(\phi(a)\otimes b(t_0))v_\p-\psi(a)\otimes b(t_0)\|+3\ep/4\\
&<&3\ep/4+\ep/4=\ep.
\eneq

The same argument shows that for any $t\in[t_{m-1}, 1]$, one has that for any $a\in\mathcal F_1$ and $b\in\mathcal F_2$,
\beq\label{right}
\|v^*(t)(\phi\otimes {\mathrm{id}}(a\otimes b(t)))v(t)-\psi\otimes{\mathrm{id}}(a\otimes b(t))\|<\ep.
\eneq

Therefore, one has
$$\|v(\phi\otimes {\mathrm{id}}(f))v-\psi\otimes{\mathrm{id}}(f)\|<\ep\rforal f\in {\cal F}.$$
\end{proof}

{
\begin{rem}
In fact, using the same argument as the lemma above, one has the following: Let $A$ and $B$ be two unital stably finite C*-algebras. 
Assume that,  for any UHF-algebra $U$ of infinite type,
\begin{enumerate}
\item the approximately unitarily equivalence classes of the monomorphisms from $A\otimes U$ to $B\otimes U$ is classified by the induced elements in $KL(A\otimes U, B\otimes U)$, {the} induced maps on traces, together with the induced maps from $U_\infty(A\otimes U)/CU_\infty(A\otimes U)$ to $U_\infty(B\otimes U)/CU_\infty(B\otimes U)$,
\item $B\otimes U$ satisfies Theorem \ref{B1B2-alg} with respect to any embedding of $A\otimes U$,
\item $B\otimes U$ satisfies a homotopy lemma, such as Theorem \ref{hp-mat} or Lemma 8.4 of \cite{Lin-hmtp},  for any embedding of $A\otimes U$ to $B\otimes U$,
\end{enumerate}
then, for any monomorphisms $\phi, \psi: A\to B$, the maps $\phi\otimes\mathrm{id}$ and $\psi\otimes\mathrm{id}$ from $A\otimes\mathcal Z_{\p, \q}$ to $B\otimes\mathcal Z_{\p, \q}$ are approximately unitarily equivalent if and only if
\beq\label{rem3-1}
[\phi]=[\psi]\,\,\,\,{\text{in}}\,\,\, KL(A, B),\
\phi_{\sharp}=\psi_{\sharp}\andeqn \phi^{\ddag}=\psi^{\ddag}.
\eneq
\end{rem}
}

%

\begin{thm}\label{uniq}
Let $A$ be a $\mathcal Z$-stable C*-algebra such that
$A\otimes M_\fr$ is an AH-algebra for
any supernatural number $\fr$ of infinite type,
and let ${B}\in {\cal C}$ be a {unital}  separable $\mathcal Z$-stable C*-algebras. If $\phi$ and $\psi$ are two monomorphisms from $A$ to $B$ with
\beq\label{Thm-2}
[\phi]=[\psi]\,\,\,\,{\text in}\,\,\, KL(A, B),\,\,\,
{\phi}_{\sharp}={\psi}_{\sharp}\andeqn {\phi}^{\ddag}={\psi}^{\ddag}{,}
\eneq
{t}hen, for any $\ep>0$ and any finite subset ${\cal F}\subseteq A,$ there exists a unitary
$u\in B$  such that
\beq\label{T2-1}
\|{u}^*\phi(a) {u}-\psi(a)\|<\ep\rforal a\in {\cal F}.
\eneq
\end{thm}
\begin{proof}
Let $\alpha: A\to A\otimes \mathcal Z$ {and} $\beta: \mathcal Z\to\mathcal Z\otimes\mathcal Z$ be {isomorphisms}. Consider the map \begin{displaymath} \Gamma_A:
\xymatrix{
A\ar[r]^-{\alpha}& A\otimes\mathcal Z\ar[r]^-{\mathrm{id}\otimes \beta} & A\otimes\mathcal{Z}\otimes\mathcal Z\ar[r]^-{\alpha^{-1}\otimes\mathrm{id}}&A\otimes\mathcal Z
}.
\end{displaymath}
Then $\Gamma$ is an isomorphism. However, since $\beta$ is approximately unitarily equivalent to the map
$$\mathcal Z\ni a\mapsto a\otimes 1\in\mathcal Z\otimes\mathcal Z,$$
{the} map $\Gamma_A$ is approximately unitarily equivalent to the map
$$A\ni a\mapsto a\otimes 1\in A\otimes\mathcal Z.$$
Hence the map $\Gamma_B\circ\phi\circ\Gamma_A$ is approximately unitarily equivalent to ${\phi\otimes \mathrm{id}_{\cal Z}}.$ The same argument shows that $\Gamma_B\circ\psi\circ\Gamma_A$ is approximately unitarily equivalent to {$\psi\otimes{\mathrm{id}}_{\cal Z}.$} Thus, in order to prove the theorem, it is enough to show that $\phi\otimes\mathrm{id}_{{{\cal Z}}}$ is approximately unitarily equivalent to $\psi\otimes\mathrm{id}_{{\cal Z}}$.

Since $\mathcal Z$ is an inductive limit of C*-algebras $\mathcal Z_{\mathfrak{p}, \mathfrak{q}}$, it is enough to show that $\phi\otimes\mathrm{id}_{\mathcal Z_{\mathfrak{p}, \mathfrak{q}}}$ is approximately unitarily equivalent to $\psi\otimes\mathrm{id}_{\mathcal Z_{\mathfrak{p}, \mathfrak{q}}}$, and this follows from Lemma \ref{N2N}.
\end{proof}

%

\section{The range of approximate equivalence classes of
\hm s}

Now let $A$ and $B$ be two unital  \CA s in ${\cal N}\cap {\cal C}.$  Theorem \ref{uniq} states
 that  two unital monomorphisms are approximately unitarily equivalent if they induce the same
 element in $KLT_e(A, B)^{++}$ and the same map on $U(A)/CU(A).$ {In t}his section, {we} will discuss the following problem:
 Suppose that {one has} $\kappa\in KLT_e(A,B)^{++}$ and a continuous \hm\, $\gamma: U(A)/CU(A)\to U(B)/CU(B)$
 which is compatible with $\kappa$. Is there always a unital monomorphism
 $\phi: A\to B$ such that $\phi$ induces $\kappa$ and $\phi^{\ddag}=\gamma?$  At least in the case
 that $K_1(A)$ is free, Theorem \ref{Ext1} states that such $\phi$ always exists.

\begin{lem}\label{2L1}
Let $A$ and $B$ be two  unital infinite dimensional separable  stably finite \CA s
whose tracial {simplexes} are non-empty.
 Let $\gamma: {U_{\infty}(A)}/CU_{\infty}(A)\to U_{\infty}(B)/CU_{\infty}(B)$
be a continuous \hm, $h_i: K_i(A)\to K_i(B)$ ($i=0, 1$) be  homomorphisms for which
$h_0$ is positive,  and
{let} ${\lambda}: \mathrm{Aff}(\tr(A))\to \mathrm{Aff}(\tr(B))$ be an affine map {so that
$(h_0, h_1, {\lambda}, {\gamma})$}  are compatible. Let $\mathfrak{p}$ be a supernatural number.
Then ${\gamma}$ induces a unique \hm\, ${\gamma}_\p: U_{\infty}(A_\p)/CU_{\infty}(A_\p)\to
U_{\infty}(B_\p)/CU_{\infty}(B_\p)$ which is
compatible with
$(h_\p)_i$ ($i=0,1$) and $\gamma_\p$, {where $A_\p=A\otimes M_{\mathfrak{p}}$ and $B_\p=B\otimes M_{\mathfrak{p}}$, and $(h_\p)_i: K_i(A)\otimes {\Q}_\p\to  K_i(B)\otimes {\Q}_\p$ is induced
by $h_i$ ($i=0,1$). M}oreover, the diagram
$$
\begin{array}{ccc}
U_{\infty}(A)/CU_{\infty}(A) & \stackrel{{\gamma}}{\to} & U_{\infty}(B)/CU_{\infty}(B)\\
\downarrow_{{\imath}_\p^{\ddag}} & & \downarrow_{({\imath}'_\p)^{\ddag}}\\
U_{\infty}(A_\p)/CU(A_\p)& \stackrel{{\gamma}_\p}{\to} & U_{\infty}(B_\p)/CU_{\infty}(B_\p)
\end{array}
$$
commutes, where ${\imath}_{\p}: A\to A_\p$ and ${\imath}'_{\p}: B\to B_\p$ are the maps induced by $a\mapsto a\otimes 1$ and $b\mapsto b\otimes 1$, respectively.

\end{lem}

\begin{proof}
Denote by $A_0=A,$ $A_\p=A\otimes M_{\mathfrak{p}}$, $B_0=B$ and
$B_\p=B\otimes M_{\mathfrak{p}}.$
By a result of K. Thomsen (\cite{Thomsen-rims}), using the de la Harpe and Skandalis determinant, one has the following
short exact sequences:
$$
0\to \mathrm{Aff}(\tr(A_i))/\overline{\rho_A(K_0(A_i))}\to U_{\infty}(A_i)/CU_{\infty}(A_i)\to K_1(A_i)\to 0,\,\,\,i=0, \p,
$$
and
$$
0\to \mathrm{Aff}(\tr(B_i))/\overline{\rho_A(K_0(B_i))}\to U_{\infty}(B_i)/CU_{\infty}(B_i)\to K_1(B_i)\to 0,\,\,\,i=0, \p.
$$
Note that, in all these cases,  $\mathrm{Aff}(\tr(A_i))/\overline{\rho_A(K_0(A_i))}$ and
$\mathrm{Aff}(\tr(B_i))/\overline{\rho_A(K_0(B_i))}$ are divisible groups, $i=0, \p$.  Therefore
the exact sequences above splits.
Fix splitting maps $s_i: K_1(A_i)\to U_{\infty}(A)/CU_{\infty}(A_i)$ and
$s'_i: K_1(B_i)\to U_{\infty}(B_i)/CU_{\infty}(B_i),$ $i=0,\p$, for the above
two splitting short exact sequences.
Let ${\imath}_\p: A\to A_\p$ be the homomorphism defined by ${\imath}_\p(a)=a\otimes 1$ for all $a\in A$ and
${\imath}_\p': B\to B_\p$ be the homomorphism defined by ${\imath}_\p'(b)=b\otimes 1$ for all $b\in B.$
Let ${\imath}_\p^{\ddag}: U_{\infty}(A)/CU_{\infty}(A)\to U_{\infty}(A_\p)/CU_{\infty}(A)$ and $({\imath}'_\p)^{\ddag}: U_{\infty}(B)/CU_{\infty}(B)\to U_{\infty}(B_\p)/CU_{\infty}(B_\p)$
be the induced maps.   The map ${\imath}_\p$ induces the following commutative diagram:
$$
\begin{array}{ccccc}
0 \to & \mathrm{Aff}(\tr(A))/\overline{\rho_A(K_0(A))} & \to U_{\infty}(A)/CU_{\infty}(A) & \to K_1(A)& \to 0\\
& \downarrow_{\overline{({\imath}_\p)_{\sharp}}}&  \downarrow_{{\imath}_\p^{\ddag}} & \downarrow_{({\imath}_\p)_{*1}} \\
0\to & \mathrm{Aff}(\tr(A_\p))/\overline{\rho_A(K_0(A_\p))} &\to U_{\infty}(A_i)/CU_{\infty}(A_\p) & \to K_1(A_\p) &\to 0.
\end{array}
$$

Since there is only one tracial state on $M_{\mathfrak{p}},$ one may identify
$\tr(A)$ with $\tr(A_\p)$ and $\tr(B)$ with $\tr(B_\p).$ One may also
identify ${\overline{\rho_{A_\p}(K_0(A_\p))}}$ with ${\overline{\R \rho_A(K_0(A))}}$ which is the closure
of those elements $r {\widehat{[p]}}$ with $r\in \R.$
Note that $(h_\p)_i: K_i(A\otimes M_{\mathfrak{p}})\to K_i(B\otimes M_{\mathfrak{p}})$ ($i=0,1$) is
given by the K\"unneth formula.
{Since ${\gamma}$ is compatible with ${\lambda},$}  ${\gamma}$ maps $\overline{\R \rho_A(K_0(A))}/{\overline{\rho_A(K_0(A))}}$ into ${\overline{\R \rho_{B}(K_0(B))}}/{\overline{\rho_B(K_0(B))}}.$ Note that
\beq\label{2L1-1}
&&{\mathrm{ker}}({\imath}_\p)_{*1}=\{x\in K_1(A): px=0\,\,\, \text{for some factor  $p$ of} \,\,\, {\mathfrak{p}}\}\,\,\,\text{and}\\
&&{\mathrm{ker}}({\imath}'_\p)_{*1}=\{x\in K_1(B): px=0\,\,\, \text{for some factor  $p$ of} \,\,\, {\mathfrak{p}}\}.
\eneq
Therefore
\beq\label{2L1-2}
&&{\mathrm{ker}}({\imath}_\p^{\ddag})=\{x+s_0(y): x\in {\overline{\R\rho_A(K_0(A))}}/\overline{\rho_A(K_0(A))},
y\in {\mathrm{ker}}(({\imath}_\p)_{*1})\}\,\,\,\text{and}\\
&&{\mathrm{ker}} ({\imath}_\p')^{\ddag}=\{x+s_0'(y): x\in {\overline{\R\rho_A(K_0(B))}}/{\overline{\rho_B(K_0(B))}},\,
y\in {\mathrm{ker}}(({\imath}_\p')_{*1})\}.
\eneq
If $y\in {\mathrm{ker}}(({\imath}_\p)_{*1}),$ then, for some
factor $p$ of ${\mathfrak{p}},$ $py=0.$ It follows that
$p{\gamma}(s_0(y))=0.$ Therefore ${\gamma}(s_0(y))$ must be in ${\mathrm{ker}}(({\imath}'_\p)^{\ddag}).$ It follows that
\beq\label{2L1-3}
{\gamma}({\mathrm{ker}}({\imath}_\p^{\ddag}))\subset {\mathrm{ker}}(({\imath}'_\p)^{\ddag}).
\eneq
This implies that {$\gamma$} induces a unique \hm\, {$\gamma_\p$} such that
the following diagram commutes:
$$
\begin{array}{ccc}
U_{\infty}(A)/CU_{\infty}(A) & \stackrel{{\gamma}}{\to} & U_{\infty}(B)/CU_{\infty}(B)\\
\downarrow_{{\imath}_\p^{\ddag}} & & \downarrow_{({\imath}'_\p)^{\ddag}}\\
U_{\infty}(A_\p)/CU_{\infty}(A_\p)& \stackrel{{\gamma}_\p}{\to} & U_{\infty}(B_\p)/CU_{\infty}(B_\p).
\end{array}
$$
The lemma follows.
\end{proof}

\begin{lem}\label{UCUfiber}
Let $A$ and $B$ be two  unital infinite dimensional separable  stably finite \CA s
whose tracial {simplexes} are non-empty.
 Let ${\gamma}: U_{\infty}(A)/{CU_\infty}(A)\to U_{\infty}(B)/CU_\infty(B)$
be a {continuous} \hm, $h_i: K_i(A)\to K_i(B)$ ($i=0, 1$)  be \hm s\, and
${\lambda}: \aff(\tr(A))\to \aff(\tr(B))$ be an affine \hm\, which are compatible. Let $\mathfrak{p}$ and $\mathfrak{q}$ be  {{  two}} relatively prime supernatural numbers  such that $M_\p\otimes M_\q=Q.$ Denote by $\mathfrak{\infty}$ the supernatural number associated
with the product ${\mathfrak{p}}$ {and}  ${\mathfrak{q}}.$ Let
$E_B: B\to B\otimes \mathcal Z_{{\mathfrak{p}},{\mathfrak{q}}}$ be {the embedding} defined
by $E_B(b)=b\otimes 1$, $\forall b\in B.$ Then
\beq\label{UCUfiber-1}
(\pi_t\circ E_B)^{\ddag}\circ {\gamma}&=& {\gamma}_{{\mathfrak{\infty}}}\circ
{\imath}_{\infty}^{\ddag}\tforal t\in (0,1),\\ \label{UCUfiber-1+}
(\pi_0\circ E_B)^{\ddag}\circ {\gamma} &=&{\gamma}_{\mathfrak{p}}\circ {\imath}_{\mathfrak{p}}^{\ddag},\andeqn \\ \label{UCUfiber-1++}
(\pi_1\circ E_B)^{\ddag}\circ {\gamma} &=&{\gamma}_{\mathfrak{q}}\circ {\imath}_{\mathfrak{q}}^{\ddag},
\eneq
{with} the notation of \ref{2L1}, where $\pi_t: \mathcal Z_{{\mathfrak{p}},{\mathfrak{q}}}\to Q$ is the point-evaluation
at $t$.

\end{lem}

\begin{proof}
Fix $z\in U_{\infty}(B)/CU_{\infty}(B).$
Let $u\in U_n(B)$ for some integer $n\ge 1$ such that $\overline{u}=z$ in $U_{\infty}(B)/CU_{\infty}(B).$
Then
\beq\label{2L2-02}
E_B^{\ddag}(z)=\overline{u\otimes 1}.
\eneq
In other words, $E_B^{\ddag}(z)$ is represented by
$w(t)\in M_n(B\otimes \mathcal Z_{{\mathfrak{p}},{\mathfrak{q}}})$ for which
\beq\label{2L2-03}
w(t)=u\otimes 1\tforal t\in [0,1].
\eneq
Therefore, for any $t\in (0,1),$  $\pi_t\circ E_B^{\ddag}(z)$ may be written as
\beq\label{2L2-04}
\pi_t\circ E_B^{\ddag}(z)=\overline{u\otimes 1}\,\,\,{\mathrm{in}}\,\,\,
U_{\infty}(B\otimes Q)/CU_{\infty}(B\otimes Q).
\eneq
This implies that
\beq\label{2L2-05}
\pi_t\circ E_B^{\ddag}(z)=({\imath}_{\infty})^{\ddag}(z)\tforal z\in U_{\infty}(B)/CU_{\infty}(B),
\eneq
where ${\imath}_{\infty}: B\to B\otimes Q$ {is} defined by
${\imath}_{\infty}(b)=b\otimes 1$ for all $b\in B.$
It follows from \ref{2L1} that
\beq\label{2L2-06}
(\pi_t\circ E_B)^{\ddag}\circ {\gamma}= {\gamma}_{{\mathfrak{\infty}}}\circ
{\imath}_{\infty}^{\ddag}\tforal t\in (0,1).
\eneq
The identities (\ref{UCUfiber-1+}) and (\ref{UCUfiber-1++}) for end points exactly follow from the same arguments.
\end{proof}

The following is standard (see the proof of 9.6 of \cite{Lnclasn}).

\begin{lem}\label{Rotadd}
Let $C$ and $A$ be two  unital separable stably finite \CA s,\ and let $\phi_1, \phi_2, \phi_3: C\to A$ be three
unital \hm s.  Suppose that
\beq\label{Rotadd-1}
[\phi_1]=[\phi_2]=[\phi_3]\,\,\,\text{in}\,\,\,KL(C,A),\\
(\phi_1)_{\sharp}=(\phi_2)_{\sharp}=(\phi_3)_{\sharp}.
\eneq
Then
\beq\label{Rotadd-2}
{\overline{R}_{\phi_1, \phi_3}=\overline{R}_{\phi_1, \phi_2}+\overline{R}_{\phi_2, \phi_3}}.
\eneq
\end{lem}

%


\begin{lem}\label{L-n} (cf. Theorem 4.2 of \cite{L-N})
Let $A$ be a unital infintie dimensional separable simple \CA\, with $\tr(A)\le 1,$  let
$C\subset A$ be a unital $C^*$-subalgebra which is a unital AH-algebra an let
$\imath : C\to A$ be the embedding.  For any $\lambda\in \mathrm{Hom}(K_0(C),
\overline{\rho_A(K_0(A))}),$ there exists $\phi\in \overline{\mathrm{Inn}}(C, A)$ such that there are \hm s $\theta_i: K_i(C)\to K_i(M_{\imath, \phi})$ with
$(\pi_0)_{*i}\theta_i={\mathrm id}_{K_i(C)},$ $i=0,1,$ and the rotation
map $R_{\imath, \phi}: K_1(C)\to \mathrm{Aff}(T(A))$ is given by
\beq\label{L-n-1}
R_{\imath, \phi}(x)=\rho_A(x-\theta_1((\pi_0)_{*1}(x))+\lambda\circ (\pi_0)_{*1}(x))\tforal x\in K_1(M_{\imath, \phi}).
\eneq
In other words,
\beq\label{L-n-2}
[\phi]=[\imath]\,\,\,{\mathrm in}\,\,\, KK(C,A)
\eneq
and the rotation map $R_{\phi, \psi}: K_1(M_{\imath, \psi})\to {\mathrm Aff}(T(A))$ is given by
\beq\label{L-n-3}
R_{\imath, \phi}(a,b)=\rho_A(a)+\lambda(b)
\eneq
for some identification of $K_1(M_{\imath, \psi})$ with $K_0(A)\oplus K_1(C).$

\end{lem}

\begin{proof}
This follows from the proof of Theorem 4.2 of \cite{L-N}. In Theorem 4.2 of \cite{L-N}, it is assumed that $\rho_A(K_0(A))$ is dense in $\mathrm{Aff}(T(A)).$ However, in fact, it is the condition $\lambda(K_1(C))\subset \overline{\rho_A(K_0(A))}$ that is used. Note that, by Theorem 3.10 of \cite{LN2}, $A$ has property (B1) and (B2) associated $C$ and a constant $\Delta_C$ (3.6 and 3.8 of \cite{L-N}) . Thus this lemma follows exactly the same proof.
  \end{proof}

\begin{lem}\label{L10}
Let $A$ be a unital AH-algebra and let $B$ be a unital separable simple amenable \CA\  with $TR(B)\le 1.$  Suppose that
$\phi_1, \phi_2: A\to B$ are two
{monomorphisms}  such that
\beq\label{L10-0}
[\phi_1]=[\phi_2]\,\,\,\text{in}\,\,\, KK(A,B),\,\,\,
(\phi_1)_{\sharp}=(\phi_2)_{\sharp}\andeqn \phi_1^{\ddag}=\phi_2^{\ddag}.
\eneq
Then there exists
a {monomorphism} $\bt: \phi_2(A)\to B$ such that $[\bt\circ \phi_2]=[\phi_2]$  in
$KK(A,B),$ $ ({\bt\circ \phi_2})_{\sharp}={\phi_2}_\sharp, \,({\bt\circ \phi_2})^{\ddag}=\phi_2^\ddag $  and $\bt\circ \phi_2$ is asymptotically unitarily
equivalent to $\phi_1.$ Moreover, if $H_1(K_0(A), K_1(B))=K_1(B),$
they are strongly asymptotically unitarily equivalent, where
$H_1(K_0(A),K_1(B))={\{x\in K_1(B): \psi([1_A]) = x\,\,\, {\mathrm for\,\,\, some}\,\,\,\psi\in \mathrm{Hom}(K_0(A), K_1(B))\}}.$

\end{lem}

\begin{proof}
By Lemma \ref{L-n}, there is a monomorphism $\beta\in\overline{\mathrm{Inn}}(\phi_2(A), B)$ such that $[\beta]=[\imath]$ in $KK(\phi_2(A), B)$
and
$$\overline{R}_{\imath, \beta}=
-\overline{R}_{\phi_1, \phi_2}
$$
where $\imath$ is the embedding of $\phi_2(A)$ to $B$ and
$\overline{R}_{\imath, \beta}$ is viewed as a \hm\, from $K_1(A)=K_1(\phi_2(A))$
to $\aff(\tr(B)).$
In other words
\beq\label{L10-n1}
\overline{R}_{\phi_2, \bt\circ \phi_2}=-\overline{R}_{\phi_1,\phi_2}.
\eneq
One also has that
\beq
&&[\phi_2]=[\beta\circ\phi_2]\,\,\,\text{in}\,\,\,KK(A, B), \\
&&(\bt\circ\phi_2)_{\sharp}=(\phi_2)_{\sharp}\andeqn (\bt\circ \phi_2)^{\ddag}=\phi_2^{\ddag}.
\eneq

Thus
\beq
&&[\phi_1]=[\beta\circ\phi_2]\,\,\,\text{in}\,\,\, KK(A,B),\\
&& (\phi_1)_{\sharp}=(\bt\circ\phi_2)_{\sharp}
\andeqn \phi_1^{\ddag}=(\bt\circ \phi_2)^{\ddag}.
\eneq

It follows from \ref{Rotadd} and \eqref{L10-n1} that
$$
\overline{R}_{\phi_1, \bt\circ \phi_2}=\overline{R}_{\phi_1, \phi_2}+
\overline{R}_{\phi_2, \bt\circ \phi_2}=0.
$$
Therefore, it follows from Theorem 4.2  of \cite{LN2} that the map $\phi_1$ and $\beta\circ\phi_2$ are asymptotically unitarily equivalent.

In the case that $H_1(K_0(A), K_1(B))=K_1(B),$ it follows from Theorem 4.4  of \cite{LN2}
 that $\bt\circ \phi_2$ and $\phi_1$ are strongly
asymptotically unitarily equivalent.
\end{proof}

\begin{lem}\label{Lrot=0}
Let $C$ and $A$ be two unital separable stably finite \CA s.
Suppose that $\phi, \psi: C\to A$ are two unital monomorphisms
such that
\beq\label{Lrot0-1}
[\phi]=[\psi]\,\,\,\text{in}\,\,\,KL(C,A),\
\phi_{\sharp}=\psi_{\sharp}\andeqn
{\overline{R}}_{\phi, \psi}=0.
\eneq
Suppose that $\{U(t): t\in [0,1)\}$ is a piecewise smooth and continuous
path of unitaries in $A$
with $U(0)=1$ such that
\beq\label{Lrot0-2}
\lim_{t\to 1}U^*(t)\phi(u)U(t)=\psi(u)
\eneq
for some $u\in U(C)$ and suppose that there exists $w\in U(A)$ such that
$\psi(u)w^*\in U_0(A).$ Let
$$Z=Z(t)=U^*(t)\phi(u)U(t)w^*\,\,\,\text{if}\,\,\,t\in [0,1)
$$ and
$Z(1)=\psi(u)w^*.$
Suppose {also} that there is a piecewise smooth continuous path of unitaries
$\{z(s): s\in [0,1]\}$ in $A$ such that
$z(0)=\phi(u)w^*$ and $z(1)=1.$
Then, for any   piecewise smooth continuous path
$\{Z(t,s): s\in [0,1]|\}\subset C([0,1],A)$ of unitaries such that
$Z(t,0)=Z(t)$ and $Z(t,1)=1,$ there is $f\in \rho_A(K_0(A))$ such that
\beq\label{Lrot0-3}
{1\over{2\pi \sqrt{-1}}}\int_0^1 \tau({dZ(t,s)\over{ds}}Z(t,s)^*) ds
={1\over{2\pi \sqrt{-1}}}\int_0^1 \tau({dz(s)\over{ds}}z(s)^*) ds +f(\tau)
\eneq
for all $t\in [0,1]$ and $\tau\in T(A).$
\end{lem}

\begin{proof}
Define
\beq\label{Lrot0-3+1}
Z_1(t,s)=\begin{cases} U^*(t-2s)\phi(u)U(t-2s)w^* & \,\,\,\text{for}\,\,\, {s}\in [0, t/2)\\
                                  \phi(u)w^* & \,\,\,\text{for}\,\,\, {s}\in [t/2, 1/2)\\
                                  z(2s-1) & \,\,\,\text{for}\,\,\,{s}\in [1/2,1]\end{cases}
                                  \eneq
                                  for $t\in [0,1)$ and define
 \beq\label{Lrot0-3+2}
 Z_1(1,s)=\begin{cases}  \psi(u)w^* & \text{for}\,\,\, s=0\\
                                      U^*(1-2s)\phi(u)U(1-2s)w^* & \,\,\,\text{for}\,\,\, {s}\in (0, 1/2)\\
                                      z(2s-1) & \,\,\,\text{for}\,\,\, {s}\in [1/2, 1].\end{cases}
                                      \eneq
Thus $\{Z_1(t,s) : s\in [0,1]\}\subset C([0,1],A)$ is a piecewise smooth continuous path of unitaries
such that $Z_1(t,0)=Z(t)$ and $Z_1(t,1)=1.$
Thus,
   there is an element $f_{1}\in \rho_A(K_0(A)),$
   such that
\beq\label{Lrot0-3+3}
 f_{1}(\tau)=  {1\over{2\pi\sqrt{-1}}}\int_0^1 \tau({dZ(t,s)\over{ds}}Z(t,s)^*)ds-{1\over{2\pi\sqrt{-1}}}\int_0^1 \tau({dZ_1(t,s)\over{ds}}Z_1(t,s)^*ds
 \eneq
 for all $\tau\in T(A)$ an for all $t\in [0, 1]$.

On the other hand, let $V(t)=U(t)^*\phi(u)U(t)$ for $t\in [0,1)$ and $V(1)=\psi(u).$
{For any $s\in [0,1),$ since $U(0)=1,$ $U(t)\in U(C([0,s], A))_0$ (for $t\in [0,s]$). There there are $a_1,a_2,...,a_k\in U([0,s],A)_{s.a.}$
such that
$$
U(t)=\prod_{j=1}^k\exp(i a_j(t))\tforal t\in [0,s]
$$
Then a straightforward calculation shows that
\begin{equation}\label{Lrot-4n}
\int_0^{s}\frac{dV(t)}{dt}V^*(t)dt=0. 
\end{equation}
}
{ We also have
\beq\label{Lrot-4n2}
{\frac{1}{2\pi {\sqrt{-1}}}}\int_0^{1}{\tau}(\frac{dV(t)}{dt}V^*(t))dt=R_{\phi, \psi}([{V}]){(\tau)}=:f(\tau)\in \rho_A(K_0(A))
\eneq
for all $\tau\in T(A).$}

Then
\beq\label{Lrot0-4}
{1\over{2\pi\sqrt{-1}}}\int_0^{1/2} \tau({dZ_1(1,s)\over{ds}}Z_1(1,s)^* )ds &=&
{1\over{2\pi\sqrt{-1}}}\int_0^{1/2}\tau( {dV(2s-1)\over{ds}}V(2s-1)^*) ds\\
&=&{{R}}_{\phi, \psi}([{V}])(\tau)={f(\tau)}\tforal \tau\in T(A).
\eneq

One computes that, for any $\tau\in T(A)$ and for any  $t\in [0,1),$
by applying (\ref{Lrot0-4}),
\beq\label{Lrot0-5}
&&{1\over{2\pi \sqrt{-1}}}\int_0^1 \tau({dZ_1(t,s)\over{ds}}Z_1(t,s)^*) ds\\
 &&=
{1\over{2\pi \sqrt{-1}}}[\int_0^{{t/2}}\tau({(d(U^*(t-2s)\phi(u)U(t-2s)w^*)\over{ds}}(U^*(t-2s)\phi(u)U(t-2s)w^*)^*)ds+\\
&&\hspace{0.6in} \int_{{t/2}}^{1/2} \tau({dZ_1(t,s)\over{ds}}Z_1(t,s)^*) ds+\int_{1/2}^1 \tau({dz(s-1)\over{ds}}z(2s-1)^*) ds]\\
&&={1\over{2\pi \sqrt{-1}}}[\int_0^{{t/2}} {dV(t-2s)\over{ds}}V(t-2s)^*ds
+\int_{1/2}^1 \tau({dz(2s-1)\over{ds}}z(2s-1)^*) ds]\\
&&={0}
 +{1\over{2\pi\sqrt{-1}}}\int_{1/2}^1 \tau({dz(2s-1)\over{ds}}z(2s-1)^*) ds\\
&&={1\over{2\pi \sqrt{-1}}}\int_0^1 \tau({dz(s)\over{ds}}z(s)^*) ds.
\eneq

It then follows { from} (\ref{Lrot0-4}) that
\beq\label{Lrot0-6}
&&{1\over{2\pi \sqrt{-1}}}\int_0^1 \tau({dZ_1(1,s)\over{ds}}Z_1(1,s)^* ds\\
&&={1\over{2\pi \sqrt{-1}}}[\int_0^{1/2} \tau({dZ_1(1,s)\over{ds}}Z_1(1,s)^* ds
+\int_{1/2}^1 \tau({dz(2s-1)\over{ds}}z(2s-1)^*)ds]\\
&&={f(\tau)}+{1\over{2\pi \sqrt{-1}}}\int_{0}^1 \tau({dz(s)\over{ds}}z(s)^*)ds
\eneq
The lemma follows.
\end{proof}

{\begin{rem}\label{RLrot}
Note that the lemma \ref{Lrot=0} applies to $M_n(C)$ and $M_n(A)$ for all integer $n\ge 1.$
So it works for all $u\in U_n(C).$
\end{rem}
}


\begin{lem}\label{LExt-K}
{ Let $A$ be a unital C*-algebra satisfying that $A\otimes M_\fr$ is an AH-algebra for all supernatural number $\fr$ with infinite type (in particular, all AH-algebra satisfies this property),}
and let $B$ be a  unital  simple \CA\, in $ {\cal N}\cap {\cal C}.$
Let $\kappa\in KL_e(A, B)^{++}$ {and} ${\lambda}:  \mathrm{Aff}(\tr(A))\to \mathrm{Aff}(\tr(B))$ be an affine \hm\,
{which are compatible {(see Definition \ref{DKL})}.}
Then there exists a unital \hm\, $\phi: A\to B$ such that
$$
[\phi]=\kappa\andeqn (\phi)_{\sharp}={\lambda}.
$$
Moreover, if {${\gamma}\in U_{\infty}(A)/CU_{\infty}(A)\to
U_{\infty}(B)/CU_{\infty}(B)$}
is a continuous
\hm\, which is compatible with $\kappa$ {and $\lambda,$} then one may also require that
\beq\label{LExt-K-1}
\phi^{\ddag}|_{{U_{\infty}(A)_0/CU_{\infty}(A)}}={\gamma}|_{U_{\infty}(A)_0/CU_{\infty}(A)}\andeqn
(\phi)^{\ddag}
\circ s_1={\gamma}\circ s_1-{\bar{h}},
\eneq
where $s_1: K_1(A)\to {U_{\infty}(A)/CU_{\infty}(A)}$ is a
splitting map (see \ref{DDet}),
and
$$\bar{h}:  K_1(A)  \to  \overline{\R \rho_B(K_0(B))}/
\overline{\rho_B(K_0(B))}
$$ is a \hm.

Moreover,
\beq\label{LExtK-2}
(\phi\otimes {\mathrm{id}}_{\mathcal Z_{{\mathfrak{p}},{\mathfrak{q}}}})^{\ddag}
\circ s_1=E_B\circ {\gamma}\circ s_1-{\bar{h}},
\eneq
{where $E_B$ is as defined in \ref{UCUfiber}.}
\end{lem}


\begin{proof} 
Let ${\mathfrak{p}}$ and ${\mathfrak{q}}$ be two relative prime supernatural numbers of infinite type such that
$Q=M_{\mathfrak{p}}\otimes M_{\mathfrak{q}}.$
Let $A_\p=A\otimes M_{\mathfrak{p}}$, $A_\q=A\otimes M_{\mathfrak{q}}$, {$B_\p=B\otimes M_{\mathfrak{p}}$ and $B_\q=B\otimes M_{\mathfrak{q}}.$ Then $A_\p$ and $A_\q$ are AH-algebras, and $TR(B_\p)\leq 1$ and $TR(B_\q)\leq 1$.}
Let $\kappa_\p\in KL(A_\p, B_\p),$ $\kappa_\q\in KL(A_\p, B_\p),$ ${\lambda}_\p: \mathrm{Aff}(\tr(A_\p))\to \mathrm{Aff}(\tr(B_\p)),$
${\lambda}_\q: \mathrm{Aff}(\tr(A_\q))\to \mathrm{Aff}(\tr(B_\q)),$
${\gamma}_\p: U(A_\p)/CU(A_\p)\to U(B_\p)/CU(B_\p)$  and
${\gamma}_\q: U(A_\q)/CU(A_\q)\to U(B_\q)/CU(B_\q)$
be induced by $\kappa,$ ${\lambda}$ and ${\gamma},$ respectively.
Note that $A_\p,$ $A_q,$ $B_\p$ and $B_q$ are all unital AH-algebras.
Moreover, since $M_{\mathfrak{r}}\cong M_{\mathfrak{r}}\oplus M_{\mathfrak{r}},$ for any supernatural number
${\mathfrak{r}}$ of infinite type,  $B_\p$ and $B_\q$ are unital simple AH-algebras of slow dimension growth.  It follows from {Corollary 6.11 of \cite{Lin-AU11}} that there is a unital \hm\,
$\phi_{\mathfrak{p}}: A_\p\to B_\p$ such that
\beq\label{2L2-1}
{[\phi_{\mathfrak{p}}]=\kappa_\p\,\,\,{\mathrm{in}}\,\,\,KL(A_\p, B_\p),
(\phi_{\mathfrak{p}})^{\ddag}={\gamma}_\p\andeqn
(\phi_{\mathfrak{p}})_{\sharp}={\lambda}_\p.}
\eneq
For the same reason, there is also a unital \hm\,
$\psi_{\mathfrak{q}}: A_\q\to B_\q$ such that
\beq\label{2L2-2}
{[\psi_{\mathfrak{q}}]=\kappa_\q\,\,\,{\mathrm{in}}\,\,\,KL(A_\q, B_\q),
(\psi_{\mathfrak{q}})^{\ddag}={\gamma}_\q\andeqn
(\psi_{\mathfrak{q}})_{\sharp}={\lambda}_\q.}
\eneq
Define $\phi=\phi_{\mathfrak{p}}\otimes {\mathrm{id}}_{M_{\mathfrak{q}}}$ and
 $\psi=\psi_\q\otimes {\mathrm{id}}_{M_{\mathfrak{p}}}.$
 From above,
 one has that
 $$
 [\phi]=[\psi]\,\,\,\text{in}\,\,\, KL(A\otimes Q, B\otimes Q), \phi_{\sharp}=\psi_{\sharp}\andeqn
 \phi^{\ddag}=\psi^{\ddag}.
 $$
Since both $K_i(B\otimes Q)$ are divisible ($i=0,1$), one actually has
$$
[\phi]=[\psi]\,\,\,\text{in}\,\,\, KK(A\otimes Q, B\otimes Q).
$$
 It follows from \ref{L10} that there is $\bt_0\in \overline{\mathrm{Inn}}(\psi(A\otimes Q), B\otimes Q)$ such that if $\imath_{\psi(A\otimes Q)}$ denotes the embedding of $\psi(A\otimes Q)$ into $B\otimes Q$,
\beq\label{2L2-2--}
&&[\bt_0]=[\imath_{\psi(A\otimes Q)}]\,\,\,\text{in}\,\,\,
KK(\psi(A\otimes Q), {B\otimes Q}),\\
&&(\bt_0)_{\sharp}=(\imath_{\psi(A\otimes Q)})_{\sharp}\andeqn
(\bt_0)^{\ddag}=(\imath_{\psi(A\otimes Q)})^{\ddag}
\eneq
such that $\phi$ and $\bt_0\circ \psi$ are strongly asymptotically
unitarily equivalent (since in this case $H_1(K_0({A\otimes} Q), K_1(B\otimes Q))=K_1(B\otimes Q)$).
Note that  one may identify $T(B_\q),$ $T(B_\p)$ and $T(B\otimes Q).$ Moreover,
$$
\overline{\rho_{B\otimes Q}(K_0(B\otimes Q))}=\overline{\R\rho_B(K_0(B))}=\overline{\rho_{B_\q}(K_0(B_\q))}.
$$
Denote by $\imath_\p: {B}_\q\to {B}\otimes Q$ the embedding
$a\mapsto a\otimes 1_{M_\p}$, and note that the image of
$\imath_{{\p}}\circ\psi_\q$ is in the image of $\psi$.  Thus,  by \ref{lem-dense},
$R_{\bt_0\circ\imath_{{\p}}\circ \psi_\q, \imath_\p\circ\psi_\q}$
{is
in
$\mathrm{Hom}((K_1(M_{\bt_0\circ\imath_\p\circ \psi_\q, \imath_\p\circ\psi_\q}), \overline{\rho_{B_\q}(K_0(B_\q))})$.} Note that
$$[\bt_0\circ\imath_\p\circ \psi_\q]=[\imath_\p\circ\psi_\q]\quad \textrm{in}\ KK(A_\q, B_\q).$$
By \ref{L-n}, there exists $\alpha\in\overline{\mathrm{Inn}}(\psi_\q(A_\q), B_{{\q}})$ such that $$[\alpha]=[\imath_{\psi_\q(A_\q)}]\quad\textrm{in}\ KK(B_\q, B_\q),$$ where $\imath_{\psi_\q(A_\q)}$ is the embedding of $\psi_\q(A_\q)$ into $B_\q$, and $$
{\overline{R}_{\alpha, \imath_{\psi_\q(A_\q)}}=
-\overline{R}_{\bt_0\circ\imath_{{\p}}\circ \psi_\q, \imath_\p\circ\psi_\q}.}
$$

As computed in the proof of \ref{L10}, one has that
\beq\label{2L2-2-}
[\imath_{{\p}}\circ\af\circ\psi_q]=[\beta_0\circ\imath_{{\p}}\circ\psi_\q]\,\,\,\text{in}\,\,\, KK(A_q,B\otimes Q),
\eneq
\beq\label{2L2-2-1}
(\imath_{{\p}}\circ\af\circ\psi_q)_\sharp=(\beta_0\circ\imath_\p\circ\psi_\q)_\sharp\,\,\,\text{and}\,\,\, (\imath_\p\circ\af\circ\psi_q)^\ddag=(\beta_0\circ\imath_\p\circ\psi_\q)^\ddag,
\eneq
and
$$
{
\overline{R}_{\imath_{{\p}}\circ\af\circ\psi_q,
\beta_0\circ\imath_\p\circ\psi_\q}=0.
}
$$

%
%
%
It follows from 7.2 and Theorem 4.2 of \cite{LN2} that
$\imath_\p\circ\alpha\circ\psi_q$ and
$\beta_0\circ\imath_\p\circ\psi_q$ are strongly  asymptotically
unitarily equivalent.

Consider maps $$(\beta_0\circ\imath_{{\p}}\circ\psi_\q)\otimes\mathrm{id}_{{M_\p}},\ \imath\circ\beta_0\circ\psi: A\otimes M_\q\otimes M_\p\to (B\otimes M_\q\otimes M_\p)\otimes M_\p,$$ where $\imath: B\otimes Q\to (B\otimes Q)\otimes M_\p$ is the embedding $b\to b\otimes 1_{M_\p}$ for all $b\in B\otimes Q$.

Identify $\beta_0\circ\psi(B\otimes M_\q\otimes M_\p)\otimes M_\p$ with $\beta_0\circ\psi(B)\otimes \beta_0\circ\psi(M_\q)\otimes \beta_0\circ\psi(M_\p)\otimes M_\p$, and consider the automorphism $\theta$ on $\beta_0\circ\psi(B)\otimes \beta_0\circ\psi(M_\q)\otimes \beta_0\circ\psi(M_\p)\otimes M_\p$ defined by
$$\theta:\ a\otimes b\otimes c\otimes d \mapsto a\otimes b\otimes d\otimes c.$$
Then
$$[\theta|_{{\beta_0(M_q)\otimes}\beta_0(M_\p)\otimes M_\p}]=[{\mathrm{id}}_{{\beta_0(M_q)\otimes}\beta_0(M_\p)\otimes M_\p}]\,\,\,
\text{in}\,\,\, KK({\beta_0(M_q)\otimes}\beta_0(M_\p)\otimes M_\p, {\beta_0(M_q)\otimes}\beta_0(M_\p)\otimes M_\p).$$
Since $K_1({\beta_0(M_q)\otimes}\beta_0(M_\p)\otimes M_\p)=\{0\},$ it follows from Theorem 4.2 of \cite{LN2} that
$\theta|_{{\beta_0(M_q)\otimes}\beta_0(M_\p)\otimes M_\p}$ is strongly asymptotically unitarily equivalent to the identity map.
Therefore $\theta$ is strongly asymptotically unitarily equivalent to the identity map.
Note that for any $a\in {A}$, $b\in M_\q$, and $c\in M_\p$, one has
\begin{eqnarray}\label{connect}
\theta(((\beta_0\circ\imath_{{\p}}\circ\psi_\q)\otimes\mathrm{id}_{M_\p})(a\otimes b\otimes c))&=&\theta(\beta_0(\psi_\q(a\otimes b)\otimes 1_{M_\p})\otimes c)\\
&=&\beta_0(\psi_\q(a\otimes b)\otimes c)\otimes 1_{M_\p}\\
&=&\imath\circ\beta_0\circ\psi(a\otimes b\otimes c).
\end{eqnarray}
Thus, the map $(\beta_0\circ\imath_{{\p}}\circ\psi_\q)\otimes\mathrm{id}_{M_\p}$ is strongly asymptotically unitarily equivalent to $\imath\circ\beta_0\circ\psi$.

Define a map $\Psi_q: A\otimes M_\q\otimes M_\p\to B\otimes M_\q\otimes M_\p\otimes M_\p$ by
\begin{equation}\label{switch}
\Psi_\q: a\otimes b\otimes c\mapsto\alpha(\psi_\q(a\otimes b))\otimes c\otimes 1_{M_\p}.
\end{equation}
Note that for all $a\otimes b\otimes c\in A\otimes M_\q\otimes M_\p$,
\begin{equation}
((\imath_\p\circ\alpha\circ\psi_q)\otimes\mathrm{id}_{M_\p})(a\otimes b\otimes c)=\alpha(\psi_\q(a\otimes b))\otimes 1_{M_\p}\otimes c
\end{equation}
Then the same argument as above shows that $\Psi_q$ is strongly asymptotically unitarily equivalent to $(\imath_\p\circ\alpha\circ\psi_q)\otimes\mathrm{id}_{M_\p}$.

Since $\phi$ and  $\beta_0\circ\psi$ are  strongly  asymptotically unitarily equivalent, one has that the map $\imath\circ\phi$ is strongly asymptotically unitarily equivalent to $\imath\circ\beta_0\circ\psi$, and hence strongly asymptotically unitarily equivalent to $(\beta_0\circ\imath_{{\p}}\circ\psi_\q)\otimes\mathrm{id}_{M_\p}$, and therefore strongly asymptotically unitarily equivalent to $(\imath_\p\circ\alpha\circ\psi_q)\otimes\mathrm{id}_{M_\p}$. {It follows that} the map $\imath\circ\phi$ is strongly asymptotically unitarily equivalent to $\Psi_\q$. {Thus} there is a continuous path of unitaries $\{w(t):\ t\in[0, 1)\}$ in $B\otimes M_\q\otimes M_\p\otimes M_\p$ with $w(0)=1$ such that $$\lim_{t\to 1}w^*(t)(\imath\circ\phi(a))w(t)=\Psi_\q(a),\quad\forall a\in A\otimes Q.$$

Pick an isomorphism $\chi': M_\p\otimes M_\p\to M_\p$, and consider the induced isomorphism $\chi: B\otimes M_\q\otimes M_\p\otimes M_\p\to B\otimes M_\q\otimes M_\p$. Note that $(\chi')^{-1}$ is strongly asymptotically unitarily equivalent to the map $\imath': M_\p\to M_\p\otimes M_\p$ {defined} by $a {\mapsto} 1\otimes a$.
Then, it is straightforward to verify that $\chi\circ\imath\circ\phi$ is strongly asymptotically unitarily equivalent to $\phi$, and $\chi\circ\Psi_\q$ is strongly asymptotically unitarily equivalent to $(\alpha\circ\psi_q)\otimes\mathrm{id}_{M_\p}.$
Thus, there is a continuous path of unitaries $u(t)$ in $B\otimes M_\p\otimes M_\q$ (one can be made it into piecewise smooth---see Lemma 4.1 of \cite{Lnclasn}) such that $u(0)=1$ and
\beq\label{2L2-2+}
\lim_{t\to 1}{\mathrm{ad}}\,u(t)\circ \phi(a)=(\af\circ \psi_{\mathfrak{q}})\otimes {\mathrm{id}}_{M_{\mathfrak{p}}}(a)\tforal a\in A\otimes Q.\eneq


This provides a unital \hm\, $\Phi: A\otimes \mathcal{Z}_{{\mathfrak{p}}, {\mathfrak{q}}}\to B\otimes \mathcal{Z}_{{\mathfrak{p}}, {\mathfrak{q}}}$ such that, for each
$t\in (0,1),$
\beq\label{2L2-3}
\pi_t\circ \Phi(a)={\mathrm{ad}}\, u(t)\circ \phi(a(t))\tforal a\in A\otimes \mathcal Z_{\p, \q}.
\eneq

Denote by ${\vartheta}$ a unital embedding  $\mathcal Z\to \mathcal Z_{\p, \q}$, and let $j:\mathcal Z_{{\mathfrak{p}},{\mathfrak{q}}}\to \mathcal Z$ be a unital homomorphism
induced by the stationary inductive limit
$$
\mathcal Z_{{\mathfrak{p}},{\mathfrak{q}}}
\stackrel{{\vartheta}}{\to} \mathcal Z_{{\mathfrak{p}},{\mathfrak{q}}}\stackrel{{\vartheta}}{\to} \mathcal Z_{{\mathfrak{p}},{\mathfrak{q}}}
\stackrel{{\vartheta}}{\to}\,\,\cdots\to\,\,\, \mathcal Z
$$
given by 3.4 of \cite{RW-Z}, where the map $\vartheta$ is regarded as its restriction to $\mathcal Z_{\p, \q}$.

As in the proof of 7.1 of \cite{Winter-Z} {(note that it follows from the same proof that Proposition 4.6 of \cite{Winter-Z} also works for homomorphisms which are not necessary being injective)},
\beq\label{2L2-4}
((\mathrm{id}_B\otimes j)\circ \Phi\circ(\mathrm{id}_A\otimes {\vartheta}))_{*i}=\kappa_i, \,\,\,i=0,1,\\\label{2L2-4+}
((\mathrm{id}_B\otimes j)\circ \Phi\circ(\mathrm{id}_A\otimes {\vartheta}))_{\sharp}={\lambda}.
\eneq

In fact, {one has} that
\beq\label{2L2-4++}
\Phi_{\sharp}(a\otimes b)(\tau\otimes \mu)=\gamma(a(\tau))\mu(b)\tforal
a\in A_{s.a.}\andeqn b\in (\mathcal Z_{{\mathfrak{p}},{\mathfrak{q}}})_{s.a.}.
\eneq

By considering $((\mathrm{id}_B\otimes j)\circ \Phi\circ(\mathrm{id}_A\otimes i))\otimes {\mathrm{id}}_{C(X_k)}: A\otimes C(X_k)
\to  B\otimes C(X_k)$ for some suitable compact metric spaces $X_k,$ the same argument shows
 that, in fact,
\beq\label{2L2-5}
[(\mathrm{id}_B\otimes j)\circ \Phi\circ(\mathrm{id}_A\otimes {\vartheta})]=\kappa.
\eneq

Define the map ${H}=(\mathrm{id}_B\otimes j)\circ \Phi\circ(\mathrm{id}_A\otimes {\vartheta})$. Then $[{H}]=\kappa$ in $KL(A, B)$ and ${H}_\sharp={\lambda}$.


Note that it follows from (\ref{2L2-4++}) that
\beq\label{2L2-5+}
\Phi^{\ddag}|_{U(A)_0/CU(A)}=E_B^{\ddag}\circ {\gamma}|_{U(A)_0/CU(A)}.
\eneq
Let $z\in U(A)/CU(A).$ Then, one has
\beq\label{2L2-6}
{H}^{\ddag}={\gamma}_{\infty}={\imath}_{\infty}^{\ddag}\circ {\gamma}.
\eneq
On the other hand, for each $z\in U(A)/CU(A),$
there is a unitary $w\in B\otimes \mathcal Z_{{\mathfrak{p}}, {\mathfrak{q}}}$ such that
\beq\label{2L2-7}
\pi_t(w)=\pi_{t'}(w)\tforal t, t'\in [0,1]\andeqn
{E_B^{\ddag}\circ{\gamma}(z)=\overline{w}.}
\eneq
Since $\pi_t(w)\in B$ is constant, {one may} use $w$ {for} its evaluation {at $t$}. Let $v_0\in U(A)$ {be} such that $\overline{v_0}=z.$
For any $t\in (0,1), $ define
\beq\label{2L2-7+1}
Z(t)=\pi_t\circ \Phi(v_0)w^*=u(t)^*\phi(v_0)u(t)w^*.
\eneq
Let $Z(t,s)$ be a piecewise smooth continuous path of unitaries in
$B\otimes \mathcal Z_{{\mathfrak{p}}, {\mathfrak{q}}}$ such that $Z(t,0)=Z(t)$ and $Z(t,1)=1.$
Denote by  $\tau_0$  the unique tracial state in $\mathrm{T}(M_{{\mathfrak{r}}}),$ where
$\mathfrak{r}$ is a supernatural number. For each $s_\mu\in \mathrm{T}({\mathcal Z}_{\p,\q}),$ one may write
$$
s_\mu(a)=\int_0^1 \tau_0(a(t))d\mu(t),
$$
where $\mu$ is a probability Borel measure on $[0,1].$

Then, for $\tau\in T(B)$ and $s_\mu\in T(\mathcal{Z}_{{\mathfrak{p}}, {\mathfrak{q}}}),$
by applying \ref{Lrot=0},
\beq\label{2L2-7+3}
{\mathrm{Det}}(Z)(\tau\otimes s_\mu)&=&
{1\over{2\pi {\sqrt{-1}}}}\int_0^1 (\tau\otimes s_ \mu)({dZ(t,s)\over{ds}}Z(t,s)^*) ds\\
&=&{1\over{2\pi{\sqrt{-1}}}}\int_0^1 \int_0^1(\tau\otimes \tau_0{)}({dZ(t,s)\over{ds}}Z(t,s)^*)d\mu(t) ds\\
&=& \int_0^1({1\over{2\pi{\sqrt{-1}}}}\int_0^1 (\tau\otimes \tau_0{)}({dZ(t,s)\over{ds}}Z(t,s)^*))ds)d\mu(t)\\
&=&\int_0^1{\mathrm{Det}}(\phi(v_0)w^*)(\tau)d\mu(t) +f(\tau)\,\,\,{\text{for\,\,\,some}\ f\in \rho_B(K_0(B)).}
\eneq
By \ref{UCUfiber} and (\ref{2L2-3}),
\beq\label{2L2-7+4}
{\mathrm{Det}}(Z)(\tau\otimes s_\mu)={\mathrm{Det}}(\phi(v_0)w^*)(\tau)+f(\tau)\in
\overline{\R\rho_B(K_0(B))}\subseteq \aff(\tr(B\otimes \mathcal{Z}_{{\mathfrak{p}}, {\mathfrak{q}}}).
\eneq

Thus, $\Phi^{\ddag}(z)(E_B\circ \lambda(z)^*)$ defines a \hm\,
from $U(A)/CU(A)$ into $\overline{\mathbb{R}\rho_B(K_0(B))}/\overline{\rho_B(K_0(B))}$ which will be denoted
by $h.$ By (\ref{2L2-5}),
\beq\label{2L2-9}
h|_{U(A)_0/CU(A)}=0.
\eneq
Thus $h$ induces a \hm\, ${\bar h}: K_1(A)\to \overline{\mathbb{R}\rho_B(K_0(B))}/\overline{\rho_B(K_0(B))}.$
\end{proof}

In \cite{Lnclasn}, it was shown that, given two unital separable simple \CA s  $A$ and $B$ in ${\cal N}\cap {\cal C},$
if there is an isomorphism on the Elliott invariant, i.e.,
$$
(K_0(A), K_0(A)_+, [1_A], K_1(A), \tr(A), r_A)\cong (K_0(B), K_0(B)_+, [1_B], \tr(B), r_B){{,}}
$$
then $A\cong B.$ The following corollary is a more general statement.

\begin{cor}\label{hom}
Let $A$ and $B$ be two unital separable \CA s in ${\cal N}\cap {\cal C}.$ Suppose
that there is a \hm\, $\kappa_i: K_i(A)\to K_i(B)$ such that
$\kappa_0$ is order preserving and $\kappa_0([1_A])\le [1_B]$ and there
is a continuous affine map $\lambda: \aff(\tr(A))\to \aff(\tr(B))$ which {is compatible
with $\kappa_0.$}
Then
there is a \hm\, $\phi: A\to B$ such that
$$
(\phi)_{*i}=\kappa_i,\,\,\,i=0,1\andeqn \phi_{\sharp}=\lambda.
$$
\end{cor}

\begin{proof}
Consider the splitting short exact sequence:
$$
0\to \textrm{Ext}_\Z(K_*(A), K_{*+1}(B))\to KK(A, B)\to \textrm{Hom}(K_*(A), K_*(B))\to 0.
$$
There exists an element $\kappa\in KK(A, B)$ such that the image of $\kappa$ in $Hom(K_*(A), K_*(B))$
is exactly the same as that $\kappa_*.$  Let $\bar \kappa$ in $KL(A,B)$ be the image of $\kappa.$
There is a projection $p\in B$ such that $[p]=\kappa_0([1_A]).$  Let $B_1=pBp.$
Then $\bar \kappa\in KL_e(A, B_1)^{++}$ and $\lambda$ and ${\bar \kappa}$ are compatible.  It follows from  \ref{LExt-K}
that there is a unital \hm\, $\phi: A\to B_1\subset B$ such that
$$
[\phi]={\bar\kappa}\andeqn \phi_{\sharp}=\lambda.
$$
\end{proof}

\begin{thm}\label{Ext1}
Let $C$ be a unital {C*-algebra such that $C\otimes M_\fr$ {is} an AH-algebra for all supernatural number $\fr$ with infinite type}, and let $A$ be a unital  simple \CA\, in ${\cal N}\cap {\cal C}$ which is ${\cal Z}$-stable.
Then, for any $\kappa\in {KLT}_e(C,A)^{++}$ and
{a continuous \hm\,} ${\gamma}: U_{\infty}(C)/CU_{\infty}(C)\to  U_{\infty}(A)/CU_{\infty}(A)$
which are compatible, there is a unital monomorphism
$\phi: C\to A$ such that
$$
{([\phi], \phi_{\sharp})}=\kappa\andeqn
\phi^{\ddag}={\gamma},
$$
provided that
{
\begin{enumerate}
\item $K_1(C)$ is a free group, or
\item $\overline{\R\rho_A(K_0(A))}/\overline{\rho_A(K_0(A))}=\{0\},$ or
\item $\overline{\R\rho_A(K_0(A))}/\overline{\rho_A(K_0(A))}$ is torsion free and $K_1(C)$ is finitely generated.
    \end{enumerate}
}
\end{thm}

\begin{proof}

It follows from \ref{LExt-K} that there is a unital monomorphism $\psi: C\to A$ such that
\beq\label{Ext1-1}
\hspace{0.4in}(\psi, \psi_{\sharp})=\kappa,\,\,\,
\psi^{\ddag}|_{U(C)_0/CU(C)}=\lambda|_{U(C)_0/CU(C)}\andeqn (\psi\otimes {\mathrm{id}}_{{\mathcal{Z}}_{\p, \q}})^{\ddag}\circ s_1=
E_B^{\ddag}\circ {\gamma}\circ s_1-{\bar h},
\eneq
where ${\bar h}: K_1(C)\to \overline{\R\rho_A(K_0(A))}/\overline{\rho_A(K_0(A))}$ is a \hm.
If
${K_1(C)}$ is free, there exists a \hm\, $h_1: K_1(C)\to \overline{\R\rho_A(K_0(A))}$
which
induces $h_1.$  In the case that $\overline{\R\rho_A(K_0(A))}/\overline{\rho_A(K_0(A))}$ is torsion
free and ${K_1(C)}$ is finitely generated, then one also obtains a such $h_1.$
Since $\overline{\R\rho_A(K_0(A))}$ is torsion free, ${h_1}$ induces
a \hm\, ${{\bar h}_1: K_1(C)/({\mathrm{Tor}}(K_1(C)))}\to
\overline{\R\rho_A(K_0(A))}.$ Since the map from
$K_1({C})/({\mathrm{Tor}}(K_1({C})))\to (K_1(A)/({\mathrm{Tor}}(K_1(A)))\otimes \Q_\p$ is injective, one obtains a \hm\, ${h_{1,\p}}:
K_1({C}\otimes M_\p)\to \overline{\R\rho_A(K_0(A))}$ such that
\beq\label{Ext1-n}
{h_1=h_{1,\p}}\circ (\imath_\p)_{*1},
\eneq
where $\imath_\fr: A\to A\otimes M_\fr$ is the embedding
so that $\imath_r(a)=a\otimes 1$ for all $a\in A$ ($\fr$ is a supernatural number).
Similarly, there is a \hm\, ${h_{1,\q}: K_1(C\otimes M_\q)}\to \overline{\R\rho_A(K_0(A))}$ such that
\beq\label{Ext1-n1}
{h_1}=h_{1,\q}\circ (\imath_\q)_{*1}.
\eneq

{Put $C'_\fr=(\psi\otimes{\mathrm{id}}_{M_\fr})(C\otimes M_\fr)),$
where $\fr$ is a supernatural number.}
It follows from  \ref{L-n} that there is a {monomorphism
$\beta_0\in \overline{\mathrm{Inn}}(C'_\p, A_\p)$} such that
\begin{equation}\label{Ext1-4}
{[\beta_0]=[\imath_{C'_\p}]\,\,\,{\mathrm{in}}\,\,\,KK({C_\p', A_\p}),\,\,\,(\bt_0)_{\sharp}={\imath_{C'_\p}}_{\sharp},\,\,
\bt_0^{\ddag}={\imath_{C'_\p}}^{\ddag}\andeqn {\overline{R}_{{\psi\otimes {\mathrm{id}}_{M_\p}, \bt_0\circ (\psi\otimes {\mathrm{id}}_{M_\p})}}}={h_{1,\p}},}
\end{equation}
{where $\imath_{C'_\p}$ is the embedding of $C'_\p$.}

Similarly, there is a {monomorphism} $\beta_1\in \overline{\mathrm{Inn}}({C'_\q, A_\q})$ such that
\begin{equation}\label{Ext1-5}
{ [\beta_1]=[\imath_{C'_\q}]\,\,\,{\mathrm{in}}\,\,\,KK({C_\q', A_\q}),\,\,\,(\bt_1)_{\sharp}={\imath_{C'_\q}}_{\sharp},\,\,
\bt_1^{\ddag}={\imath_{C'_\q}}^{\ddag}\andeqn {\overline{R}_{{\psi\otimes {\mathrm{id}}_{M_\q}, \bt_1\circ (\psi\otimes {\mathrm{id}}_{M_\q})}}}={h_{1,\q}},}
\end{equation}
where $\imath_{C'_\q}$ is the embedding of $C'_\q$.
%
%

As in the proof of \ref{LExt-K}, by applying \ref{L10}  and its proof, one has a {monomorphism
$\bt_2\in \overline{\mathrm{Inn}}(\beta_1\circ(\psi\otimes\mathrm{id}_{M_\q})(C_\q), A_\q)$} and a piecewise smooth continuous path of unitaries $\{U(t): t\in [0,1)\}$ of $A\otimes Q$
such that $U(0)=1$ and
\begin{equation}\label{Ext1-5+}
{ [\bt_2\circ\beta_1\circ(\psi\otimes\mathrm{id}_{M_\q}))]=[\beta_0\circ(\psi\otimes\mathrm{id}_{M_\p})]\quad{\mathrm{in}}\ KK(C_{\q}, A_{\q}),}
\end{equation}

\begin{equation}
(\bt_2\circ\beta_1\circ(\psi\otimes\mathrm{id}_{M_\q})))_\sharp=(\beta_0\circ(\psi\otimes\mathrm{id}_{M_\p}))_\sharp
\end{equation}
and
\begin{equation}
(\bt_2\circ\beta_1\circ(\psi\otimes\mathrm{id}_{M_\q})))^\ddag=(\beta_0\circ(\psi\otimes\mathrm{id}_{M_\p}))^\ddag.
\end{equation}

{{Moreover, if denote by $\psi_0=\beta_0\circ(\psi\otimes\mathrm{id}_{M_\p})$ and $\psi_1= \bt_2\circ\beta_1\circ(\psi\otimes\mathrm{id}_{M_\q}))$, one has that}
\begin{equation}
\lim_{t\to 1} U(t)^*(\psi_0\otimes {\mathrm{id}}_{M_\q})(a)U(t)=(\psi_1\otimes {\mathrm{id}}_{M_\p})(a)
\end{equation}
%
%
for all $a\in A\otimes Q.$
In particular,
\beq\label{Ext1-7}
{\overline{R}_{\psi_0\otimes {\mathrm{id}}_{M_\q},  \psi_1\otimes {\mathrm{id}}_{M_\p}}=0.}
\eneq
}

Let $\Phi: A\otimes {\cal Z}_{\p, \q}\to A\otimes {\mathcal Z}_{\p, \q}$ be defined by
\beq\label{Ext1-8}
\hspace{0.2in}\Phi(a\otimes b)(t)&=&U^*(t)(({\psi_0}\otimes \mathrm{id}_{M_\q}(a\otimes  b(t))U(t)\tforal t\in [0,1)
\andeqn\\
\Phi(a\otimes b)(1)&=&{\psi_1}\otimes {\mathrm{id}}_{M_\p}(a\otimes b(1)),
\eneq
for all $a\otimes b\in A\otimes {\cal Z}_{\p, \q}.$

We claim that
\beq\label{Ext1-n+}
\Phi^{\ddag}\circ (E_A\circ \psi)^{\ddag}\circ s_1=(E_A)^{\ddag}\circ {\gamma}\circ s_1.
\eneq

To compute $\Phi^{\ddag},$ let $x\in s_1(K_1(C))$ and $v_0\in U(C)$ such that
$\overline{v_0}=x.$ There is $w\in U(A\otimes {\mathcal Z}_{\p,\q})/CU(A\otimes {\cal Z}_{\p,\q})$ such that $w(t)=w(t')$
for all $t, t'\in [0,1]$ and
\beq\label{Ext1-9}
E_A^{\ddag}\circ {\gamma}\circ s_1(x)=\overline{w}.
\eneq
Let $Z=(\Phi\circ (\psi\otimes {\mathrm{id}}_{{\cal Z}_{\p,\q}})(v_0))w^*\in A\otimes {\cal Z}_{\p, \q}.$  Note that $Z\in U(A\otimes \mathcal Z_{\p, \q})_0.$
Suppose that there is a piecewise smooth continuous path $\{Z(t, s): s\in [0,1]\}\subset A\otimes  {\cal Z}_{\p,\q}$ such that
$Z(t,0)=Z(t)$ and $Z(t,1)=1.$
Then
\beq\label{Ext1-10}
&&{\mathrm{Det}}(Z(t,s))\\
&&={\mathrm{Det}}(\Phi\circ
((\psi\otimes \mathrm{id}_{{\cal Z}_{\p,\q}})(v_0))(\psi\otimes \mathrm{id}_{{\cal Z}_{\p,\q}}(v_0)^*)
+\mathrm{Det}((\psi\otimes \mathrm{id}_{{\cal Z}_{\p,\q}})(v_0)w^*)\\
&&={\mathrm{Det}}(\Phi\circ
((\psi\otimes {\mathrm{id}}_{{\cal Z}_{\p,\q}})(v_0))(\psi\otimes {\mathrm{id}}_{{\cal Z}_{\p,\q}}(v_0)^*)+{\bar h}\circ s_1(x).
\eneq
It follows from \ref{Lrot=0} that
\beq\label{Ext1-11}
{\mathrm{Det}}(\Phi\circ
((\psi\otimes {\mathrm{id}}_{{\cal Z}_{\p,\q}})(v_0))(\psi\otimes {\mathrm{id}}_{{\cal Z}_{\p,\q}}(v_0)^*)&=&{\mathrm{Det}}(\bt_0\circ \psi(v_0)
\psi(v_0)^*)+\rho_A(K_0(A))\\
&=& R_{\bt_0\circ \psi, \psi}([v_0])+\rho_A(K_0(A))\\
&=& -h_{{{1}},\p}\circ s_1(x)+\rho_A(K_0(A)).
\eneq
Therefore, by (\ref{Ext1-n}) and by (\ref{Ext1-10}),
$$
{\mathrm{Det}}(Z(t,s))(\tau\otimes s_\mu)\in \rho_A(K_0(A)).
$$
This proves the claim.

Regard $\psi$ as a map to $A\otimes\mathcal Z$. Denote by $j: {\mathcal{Z}}_{\p,\q}\to {\cal Z}$
 {the} unital homomorphism induced by the stationary inductive limit decomposition of $\mathcal Z$, and denote by ${\vartheta}:\mathcal Z\to \mathcal Z_{\p.\q}$ { the} unital embedding induced by tensoring $\mathcal Z$ ($\mathcal Z_{{\p, \q}}$ is $\mathcal Z$-stable). Consider
$$\phi=(\mathrm{id}_{A}\otimes j) \circ \Phi\circ(\mathrm{id}_A\otimes {\vartheta})\circ \psi.$$
One then checks that
$$
[\psi]=[\phi]\,\,\,{\text{in}\,\,\,} KL(C,A),\,\,\,\phi_{\sharp}=\psi_{\sharp}\andeqn
\phi^{\ddag}={\gamma}.
$$
\end{proof}


\begin{rem}\label{FR}
It follows from Proposition 3.6 of \cite{Lnunitary} that,
if $TR(A)\le 1,$ then
$$\overline{\R\rho_A(K_0(A))}/\overline{\rho_A(K_0(A))}=\{0\}.$$
So Theorem \ref{Ext1} recovers  a version of Theorem 8.6 of \cite{Lnclasn}.

{
Now suppose that in \ref{Ext1},
$$
U_{\infty}(C)/CU_{\infty}(C)=U_{\infty}(C)_0/CU_{\infty}(C)\oplus G_1\oplus {\mathrm{Tor}}(K_1(C)),
$$
where $G_1$ is identified with a free subgroup of $K_1(C).$
From the proof of Theorem \ref{Ext1}, we see that, if $\kappa\in KLT_e(C,A)^{++}$ and
${\gamma}: U_{\infty}(C)/CU_{\infty}(C)\to U(A)/CU(A)$ which is compatible to $\kappa$ are given,
there is a unital monomorphism $\phi: C\to A$ such that $([\phi], \phi_{\sharp})=\kappa$
and
$$
\phi|_{U_{\infty}(C)_0/CU_{\infty}(C)\oplus G_1}={\gamma}|_{U_{\infty}(C)_0/CU_{\infty}(C)\oplus G_1}
$$
and
$$
\phi^{\ddag}(z){{-}}{\gamma}(z)\in \overline{\R \rho_A(K_0(A))}/\overline{\rho_A(K_0(A))}
$$
for all $z\in {\mathrm{Tor}}({K_1}(C)).$}
\end{rem}

\bibliographystyle{plain}

\hspace{0.3in}


\end{document}